\documentclass[12pt,twoside,]{amsart}
\usepackage{etex}
\usepackage{amsmath,amssymb,amsthm,amsfonts}
\usepackage{mathtools}
\usepackage[utf8]{inputenc}
\usepackage{graphicx}

\usepackage{fullpage}
\usepackage{endnotes}
\let\footnote=\endnote
\usepackage{tikz}
\usepackage[all]{xy}
\usepackage{cancel}
\usepackage{comment}
\usepackage{extarrows}
\usepackage{array}
\usepackage{graphics}
\usepackage{mathrsfs}
\usepackage{booktabs}
\usepackage[normalem]{ulem}
\usepackage{amscd}
\usetikzlibrary{matrix,arrows,decorations.pathmorphing}
\tikzset{my loop/.style =  {to path={
  \pgfextra{}
  [looseness=12,min distance=10mm]
  \tikz@to@curve@path},font=\sffamily\small
  }}  
\usepackage{tikz-cd}
\usepackage{hyperref}
\usepackage{cleveref}

\theoremstyle{plain}
\newtheorem{theorem}{Theorem}[subsection]
\newtheorem{lemma}[theorem]{Lemma}
\newtheorem{sub-lemma}[theorem]{Sub-lemma}
\newtheorem{proposition}[theorem]{Proposition}
\newtheorem{set-up}[theorem]{Set-up}
\newtheorem{corollary}[theorem]{Corollary}

\newtheorem{definition}[theorem]{Definition}

\theoremstyle{definition}
\newtheorem{remark}[theorem]{Remark}

\newtheorem{example}[theorem]{Example}

\newcommand{\Pic}{\operatorname{Pic}}
\newcommand{\Z}{\mathbb Z}
\newcommand{\OO}{\mathcal O}
\newcommand{\F}{\mathbb F}

\newcommand{\G}{G}                      

\title{
Geography and Deformations of
$\mathbb{Z}_2^s$-Covers of General Type Over Weighted Projective Threefolds.
}
\date{}

\author[P. Gallardo]{Patricio Gallardo}
\address{Department of Mathematics, University of California, Riverside, Riverside, CA 92521, United States.}
\email{pgallard@ucr.edu}

\author[J. Mukherjee]{Jayan Mukherjee}
\address{Department of Mathematics, Oklahoma State University, 401 Mathematical Sciences, Stillwater, OK 74078, United States.}
\email{jayan.mukherjee@okstate.edu }

\begin{document}
\begin{abstract}
We study threefolds of general type constructed as $\mathbb{Z}_2^s$-covers of weighted projective spaces with a particular focus on their invariants, deformation theory, and the behavior of the $m$-canonical map.
For the invariants, we write the ratios of the volume to the topological and holomorphic Euler characteristics as functions of the ratios of the degree of the branch divisors with respect to the total degree. From this expression, we obtain their asymptotic behavior, bounds, and a counterexample to a conjecture made by Bruce Hunt about the non-existence of smooth threefolds in a forbidden zone. 
From the perspective of deformation theory, we extend the criterion for such covers to be general in their moduli to the case when the weighted projective threefold has isolated singularities and the cover is non- flat—i.e., the pushforward of the structure sheaf splits as a direct sum of reflexive sheaves as opposed to line bundles. As an application, we present new numerical criteria for constructing components of the moduli spaces of stable threefolds and give concrete examples illustrating their application. Finally, we introduce techniques from Fourier transforms on finite groups to completely classify when a $\mathbb{Z}_2^s$-cover is a flat pluricanonical map. For $s \geq 2$, there are $32$ deformation types. We also show that there exist non-flat canonical and bicanonical $\mathbb{Z}_2^s$-covers for arbitrarily large values of $s$. 
\end{abstract}
\maketitle

\section{Introduction}
The classification of varieties of general type and the description of their moduli spaces have long been central problems in higher-dimensional algebraic geometry.  Key questions include understanding the numerical invariants of the varieties, the geometry of the corresponding moduli spaces, and identifying special geometric behavior such as a finite (pluri)-canonical map.  While there is a substantial body of work for surfaces of general type, the threefold case remains comparatively less understood; this work addresses that case.

Our article studies threefolds with ample canonical bundle that arise as  \(\mathbb{Z}_2^{s}\)-covers of weighted projective spaces.  We build on the foundational theory of Pardini \cite{pardini1991abelian} and its extension to singular bases by Alexeev–Pardini \cite{AP12}, to address the following questions:
\begin{enumerate}
    \item[(i)] What is the asymptotic behavior of the Chern ratios of threefolds of general type constructed as \(\mathbb{Z}_2^{s}\)-covers of weighted projective spaces?
    \item[(ii)] When do such threefolds define an open subset of a component in the moduli space of stable varieties?
    \item[(iii)] When is the \(\mathbb{Z}_2^{s}\)-cover induced by the pluricanonical map?
\end{enumerate}
Next, we describe our contributions and their context within the current literature. 

\subsection{Chern ratios for \(\mathbb{Z}_2^{s}\)-covers of weighted projective spaces.}
The possible values of Chern numbers and their ratios have attracted a lot of interest for surfaces of general type. It is known that the Chern ratio \( c_1^2(S)/c_2(S) \) is bounded by the interval \([1/5,3] \) because of the BMY and Noether inequalities. Every rational point of this interval can be realized as the Chern ratio of some surface \cite{sommese1984density};
and there is a subtle
relationship between accumulation points of the ratios with  the geometry of the surfaces \cite{roulleau2015chern}.

In the case of threefolds, we need three invariants: the volume \(K_X^3\), the holomorphic Euler characteristic \(\chi(\mathcal{O}_X)\), and the topological Euler characteristic \(e(X)\)—that is, their Chern numbers when \(X\) is smooth. A novel problem for threefolds is the absence of a unique minimal model, and the fact that the Chern numbers can vary among them. However, the ratios
\begin{align*}
\frac{c_3(X)}{c_1(X)c_2(X)}
=
\frac{e(X)}{24\,\chi(\mathcal{O}_X)}
&&
\frac{c_1^3(X)}{c_1(X)c_2(X)}
=
\frac{-K_X^3}{24\,\chi(\mathcal{O}_X)},
\end{align*}
initially defined for smooth threefolds, 
are also well-defined whenever \(X\) is \(\mathbb{Q}\)-factorial (see \cite{huntthreefolds89}). These ratios form a bounded region \cite{changlopez01} when $X$ is a smooth threefold, and they satisfy a Noether-type inequality \cite{chenchenjiang20, chenhujiang25}. Moreover, in \cite[Sec 7.2.3-7.2.4]{huntthreefolds89}, Hunt calculated the Chern numbers of smooth complete intersections in $\mathbb{P}^N$. Then, he showed that their limits lie on the curve 
 $\{ y(3x+1) = 4\}$ and conjectured that there are no smooth threefolds with Chern ratios that lie above the curve.

In our first result, we describe the asymptotic behavior (Definition~\ref{def:limit}) of the Chern ratios associated to our threefolds under the hypothesis that the branch locus of the \(\mathbb{Z}_2^{s}\)-cover is generic (see Definition~\ref{def:divisor_standard set-up}). 

\begin{figure}[t!]
    \centering
    \caption{Asymptotic Chern ratios for
    $\mathbb{Z}_2^s$-covers of $\mathbb{WP}^3$}
    \label{fig:invariants}    
\includegraphics[width=0.9\linewidth]{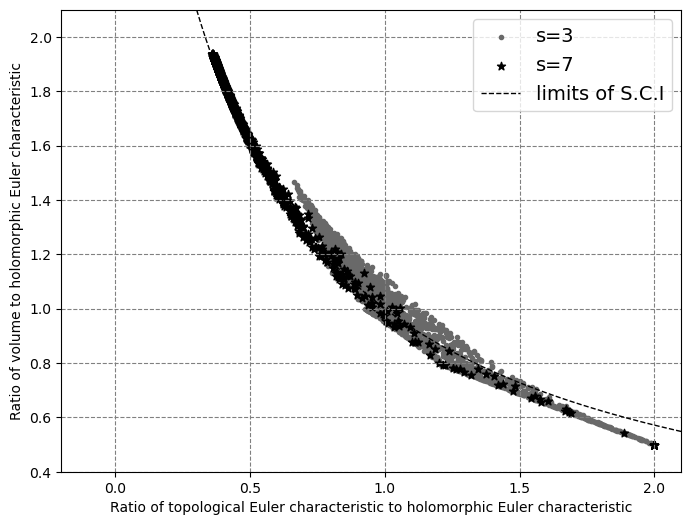}
\end{figure}

\begin{theorem}\label{thm:main_inv}
Let $X \rightarrow \mathbb{WP}^3$ be a 
$G:=(\mathbb{Z}_2)^s$-cover of a weighted projective threefold with generic branch divisors $(D_g)_{g \in G}$, 
and consider its Chern ratios 
\begin{align*}
x(D_g)&:= \frac{e(X)}{24\,\chi(\mathcal{O}_X)}.
&&
y(D_g):=\dfrac{-K_X^3}{24\,\chi(\mathcal{O}_X)}
\end{align*}
Set
\begin{align*}
r_g:= \frac{\deg(D_g)}{\sum_{h\in G \setminus \{0 \}} \deg(D_{h})},\,
&&
\operatorname{SCI}(x, y):= y(3x+1)-4
\end{align*}
where $\{ \operatorname{SCI}(x,y) =0 \}$ is the curve supporting the limits of the Chern ratios of smooth complete intersections (SCI). Then, it holds that
\begin{align*}
-\frac{1}{2}
\leq
\lim_{\deg(D_g) \to \infty}  
\operatorname{SCI} \left(x(D_g), \, y(D_g) \right)
\leq
\frac{8}{3},
\end{align*}
and
\begin{align*}
\frac{3 (2^{2 s-2}) + 2^{s + 1} - 1}{3 (2^s - 1)^2} 
\leq 
\lim_{\deg(D_g) \to \infty} x(D_g) 
\leq 2,
&&
\frac{1}{2} \leq
\lim_{\deg(D_g) \to \infty} y(D_g) \leq 2 - \frac{1}{2^{s-2}} + \frac{1}{2^{2s-1}}, 
\end{align*}
\end{theorem}
The proof of Theorem~\ref{thm:main_inv} is the main theme of Section~\ref{sec:proof_invariants_theorem}.
Next, we highlight a consequence of our work. 
In \cite{huntthreefolds89}, Hunt
defined the so-called empty zone
\begin{align*}
  \text{Zone}(E):=
  \{(x,y) \, | \, 
  \operatorname{SCI}(x,y) > 0 \},
  &&
  (x,y) =  \left(
\frac{e(X)}{24\,\chi(\mathcal{O}_X)} , 
    \frac{-K_X^3}{24\,\chi(\mathcal{O}_X)} \right)
\end{align*}
and said, ``\emph{I know of no examples of algebraic 3-folds with ratios of Chern numbers lying here... I conjecture there are none, which is why I term this the empty zone, or alternatively, if there are any, they would seem to be exotic.}" We show that such conjecture is not true by using the almost uniform covers introduced by Pardini and Alexeev. 
\begin{corollary}\label{cor:counter_example_hunt}
There are smooth $(\mathbb{Z}_2)^s$-covers of $\mathbb{P}^3$ whose Chern ratios are within the empty zone $\operatorname{Zone}(E)$.
\end{corollary}


\subsection{Deformations of $\mathbb{Z}_2^s$-covers of weighted projective spaces }
From the perspectives of both geography and moduli theory of threefolds, a very relevant and important question is the deformation theory of the abelian covers. The study helps us understand how general abelian covers are in their corresponding moduli component. When both $X$ and $Y$ are smooth, the study has been carried out in \cite{pardini1991abelian}, while the case when $Y$ is normal and $f: X \to Y$ is flat and locally simple (which implies $X$ is smooth) is carried out in \cite{Man95}. When both $X$ and $Y$ are normal and $f$ is neither flat nor locally simple, the situation is quite different. However, using techniques of Schlessinger in \cite{schlessinger1971}, Wehler in \cite{wehler86deformations} and Ran in \cite{defofmaps89}, we can handle the deformation theory when the codimension of the singular locus of both $X$ and $Y$ is at least $3$. In \cite{wehler86deformations}, a sufficient condition was introduced, under which any deformation of a variety $X$ which arises as a finite cover $X \to Y$, retains the structure of a finite cover over some deformation of $Y$. In \Cref{every def is a finite cover revised reflexive}, we generalize Wehler's result as follows:

\begin{theorem}[see \Cref{every def is a finite cover revised reflexive}]
  Let $f: X \to Y$ be a finite cover, with $X$ and $Y$ reduced and non-singular in codimension one, such that $f_*\mathcal{O}_X = \mathcal{O}_Y \oplus \mathcal{E}$, where $\mathcal{E}$ is a reflexive sheaf on $Y$. Suppose that $H^1(\mathcal{H}om(\Omega_Y, \mathcal{E})) = H^0(\mathcal{E}xt^1(\Omega_Y, \mathcal{E})) = 0$. Then the natural forgetful map of functors $\mathbf{Def}(f) \xrightarrow{\pi_0} \mathbf{Def}(X)$ is smooth. Consequently, any deformation of $X$ can be realized as a finite cover over a deformation of $Y$. 
\end{theorem}

Due to the results in \cite{Gon06}, \cite{GGP16}, when $Y$ is smooth, such results can be interpreted as the non-existence of ribbons with conormal bundle $\mathcal{E}$ on $Y$ which are parameterized by $H^1(T_Y \otimes \mathcal{E})$. When $Y$ is singular, $H^1(\mathcal{H}om(\Omega_Y, \mathcal{E}))$ is the natural generalization of $H^1(T_Y \otimes \mathcal{E})$ because $\Omega_Y$ is not locally free. However, it is worth noticing that we have an extra term $H^0(\mathcal{E}xt^1(\Omega_Y, \mathcal{E})) = 0$, which appears due to the contribution of local deformations of the singularities of $Y$. The above result fits nicely in this context, since now due to the local-to-global Ext sequence, ribbons on $Y$ can arise due to the non-vanishing of both $H^1(\mathcal{H}om(\Omega_Y, \mathcal{E}))$ and $H^0(\mathcal{E}xt^1(\Omega_Y, \mathcal{E}))$.  \par

Next we find out under what additional hypothesis the general deformation of $X$ is not just a finite cover of $Y$ but actually a $\mathbb{Z}_2^s$ cover of $Y$. 

\begin{theorem}[see \Cref{every def is an abelian cover}]
  If $f: X \to Y$ is an abelian cover that satisfies conditions of \Cref{every def is an abelian cover}, then, any deformation of $X$ can be realized as an abelian cover over $Y$. Further, in this situation, $X$ is unobstructed and represents a smooth point in its moduli.
\end{theorem}

When $f: X \to Y$ is a non-flat abelian cover, the pushforward $f_*\mathcal{O}_X$ decomposes as $\mathcal{O}_Y \oplus \mathcal{E}$, where $\mathcal{E} = \oplus_{\chi \in G^* \setminus \{1\}}L_{\chi}^{-1}$ where each $L_{\chi}^{-1}$ is a reflexive sheaf of rank one. To apply the above theorem, one needs to show the vanishing of $H^0(\mathcal{E}xt^1(\Omega_Y, L_{\chi}^{-1}))$. Using techniques introduced in \cite{schlessinger1971} we show that

\begin{proposition}[see \Cref{analog Schlessinger main}]\label{analog Schlessinger main intro}
    Let $p : Z \to X$ be a morphism of geometric local schemes with closed points $z$ and $x$ respectively. Suppose that $Z$ is smooth and that $X = Z/G$, where $G$ is a finite group of automorphisms of $Z$, having $z$ as the only fixed point. Let $L$ be a reflexive sheaf of rank one on $X$ which is locally free outside finitely many closed points. If $\operatorname{dim}(Z) \geq 3$ and $H_z^2(Z, T_Z \otimes p^*L) = 0$, then $H^0(\mathcal{E}xt^1(\Omega_X, L)) = 0$.
\end{proposition}

Next, we apply the above results to $\mathbb{Z}_2^s$-covers of $\mathbb{WP}^3$. By the results of \cite{pardini1991abelian} and \cite{AP12} (see \Cref{lem:AP_extension}), such covers are determined by a choice of a set of Weil-divisors $D_g \in |\mathcal{O}_{\mathbb{P}}(d_g)|$ with $g \in \mathbb{Z}_2^s \setminus \{0\}$ satisfying some fundamental relations (see \Cref{eq:fundamental_relations}) that depend only on $\operatorname{deg}(D_g) = d_g$. We show that under some combinatorial hypothesis in the form of linear inequalities on the space of admissible multidegrees $\{d_g\}_{g \in \mathbb{Z}_2^s \setminus \{0\}} \in \mathbb{N}_{\geq 0}$ (i.e, the set of multidegrees that satisfy the fundamental relations), every $X$ that arises as a $\mathbb{Z}_2^s$-cover over a weighted projective space with isolated singularities, branched along generic quasismooth divisors, retains the structure of a $\mathbb{Z}_2^s$-cover under deformations, i.e., the abelian covers are general in their moduli.

\begin{corollary}[see \Cref{every def is an finite cover over a WPS reflexive}, \Cref{every def is an abelian cover over a WPS} ]
      Let $f: X \to Y$ be a $\mathbb{Z}_2^s$ cover where $Y = \mathbb{P}(a_0,...,a_r)$, with $r \geq 3$, $\operatorname{gcd}(a_i, a_j) = 1$ for all $i,j$ and $X$ has at worst isolated quotient singularities. Let $f_*\mathcal{O}_X = \mathcal{O}_Y \oplus \mathcal{E}$, where $\mathcal{E} = \oplus_{\chi \in G^* \setminus \{1\}}L_{\chi}^{-1}$ with each $L_{\chi}$ 
\textcolor{black}{is a reflexive sheaf of rank one} which is locally free outside a finite set of points. Then the natural forgetful map of functors $\bf Def(f) \to \bf Def(X)$ is smooth. Consequently any deformation of $X$ can be realized as a finite cover over a deformation of $Y$. Further, suppose that 
    \begin{enumerate}
        \item the branch divisors $D_g \in | \mathcal{O}_{\mathbb{P}}(d_g)|$ associated to the data of the abelian cover be quasismooth and generic (see \Cref{def:divisor_standard set-up}).
        
        \item For each pair $(g, \chi)$ such that $\chi \cdot g = 0$, 
        that is $\chi(g) =  (-1)^{\chi \cdot g} = 1$
        and $\chi \neq \operatorname{Id} \in G^*$, 
        \begin{align*}
        d_g < \displaystyle\frac{1}{2} \Sigma_{\chi \cdot g' = 1} d_{g'}
        =
        \displaystyle\frac{1}{2} 
        \Sigma_{\chi(g') = -1} d_{g'},
        \end{align*}

        \item $\Sigma_g d_g > 2 \Sigma_{j=0}^r a_j$
    \end{enumerate}

Then any deformation of $X$ can be realized as an abelian cover over $Y$. Further, in this situation, $X$ is unobstructed and represents a smooth point in its moduli.
\end{corollary}

As a first consequence of the above theorem, we show in \Cref{new components}, how to construct new moduli components of stable threefolds such that the general member is a $\mathbb{Z}_2^s$-cover of a weighted projective threefold.

\begin{example}[see \Cref{new components}]
Let $M > 2$ be an even positive integer and  let $f: X \to \mathbb{P}(1,1,1,M)$ be a $\mathbb{Z}_2^4$-cover with branch divisors $D_{g_0} \in |\mathcal{O}_{\mathbb{P}}(2)|$ and $D_{g} \in |\mathcal{O}_{\mathbb{P}}(M)|$ for $g \neq g_0, g \neq 0$. Then one can check that 
\[
f_*(\mathcal{O}_X) = \bigoplus_{\{\chi | \chi(g_0) = -1\}} \mathcal{O}_{\mathbb{P}}(-(1+\frac{7}{2}M)) \bigoplus_{\{\chi | \chi(g_0) = 1\}} \mathcal{O}_{\mathbb{P}}(-4M).
\]
Since the pair $(\mathbb{P}(1,1,1,M), D_{g_0})$ is a stable pair, we have that $X$ is stable by \Cref{cover is stable}. Further, since $(1+\frac{7}{2}M)$ is not divisible by $M$, $f$ is a non-flat cover. This example satisfies all the conditions of \Cref{every def is an abelian cover over a WPS} and hence defines a new component of the moduli of stable threefolds.
\end{example}

As a second consequence, we give for every $s \geq 4$, at least two configurations of hyperplane arrangements in a weighted projective threefold such that there exist abelian covers with ample canonical bundle branched along such arrangements which are general in their moduli components.
\begin{corollary}\label{hyperplane arrangements intro}
Let $Y = \mathbb{P}(a_0,...,a_r)$ with $\textrm{gcd}(a_i,a_j) = 1$ for $0 \leq i < j \leq r$ and $r \geq 3$. Now consider $G = \mathbb{Z}_2^s$ covers branched along either one of the following configurations of hyperplane arrangements.
    \begin{enumerate}
        \item $s \geq 3$, $d_g = 1$ for $g \in G-\{0\}$ and $\Sigma_{j = 0}^r a_j < \displaystyle\frac{2^s-1}{2}$.
        \item $s \geq 4$, fix a linear subspace $L \subset G$, with $\operatorname{dim}(L) \geq 2$, where we look at $G$ as a $\mathbb{Z}_2$ vector space. Then $d_g = 1$ if and only if $g \in G\setminus L$. In this case assume
        $\Sigma_{j = 0}^r a_j < \displaystyle\frac{2^s-2^{\operatorname{dim}(L)}}{2}$ 
    \end{enumerate}
Then for a generic choice of branch divisors $\{D_g \in |\mathcal{O}_Y(d_g)|\}_{g \in G-\{0\}}$, such that $D_g$ is quasismooth for each $g \neq 0$, there exists a $G = \mathbb{Z}_2^s$ cover $f: X \to Y$, with $K_X$ ample, $X$ unobstructed and all deformations of $X$ are once again $\mathbb{Z}_2^s$ covers of $Y$ branched along the same configuration of hyperplane arrangements.
\end{corollary}

\subsection{
Flat $\mathbb{Z}_2^s$-covers and pluricanonical maps}
The most fundamental map naturally associated with a variety of general type is its $m$-canonical map. So naturally, the $\mathbb{Z}_2^s$-covers are all the more interesting if they arise as one of the $m$-canonical maps of the threefold. If $f: X \to Y$ is a $\mathbb{Z}_2^s$-cover then we say that $f$ is an $m$-canonical cover if $f$ fits into a diagram as follows for some $m \geq 1$.
\[
\begin{tikzcd}
    X \arrow[d, "f"] \arrow[dr, "\varphi_{|mK_X|}"] & \\
    Y \arrow[r, hook] & \mathbb{P}^N
\end{tikzcd}
\]
When $m = 1$, $Y = \mathbb{P}^3$ and $p_g(X) = 4$, the classification of such covers was obtained in \cite{dugao16}. 
In our next result, whose proof is given in  Section \ref{sec:pluricanonical}, we give 
the complete classification of all flat $m$-canonical maps that are $\mathbb{Z}_2^s$-covers of a weighted projective threefold. 

To obtain this explicit classification, we introduce some techniques from the Fourier transform for functions over finite groups (in our case the group is $\mathbb{Z}_2^s$). In \Cref{prop:pluricanonical_maps_of_threefolds}, we start with a criterion for a $\mathbb{Z}_2^s$-cover to be pluricanonical which is essentially a set of linear inequalities involving the degree of the branch divisors. Then, through a combinatorial lemma (see \Cref{lemma:ineq_comb}), we first establish upper bounds to the sum of weights with respect to the lcm of the weights.
This allows us to show upper bounds on $s$ where $2^s$ is the degree of the cover, on the positive integer $m$ for which it is the $m$-canonical map, and on the sum of weights. (see \Cref{lemma:forbidden_flat_admissible}, \Cref{m atleast 3}, \Cref{m atmost 2}). \par

From this point, the problem boils down to the following: given a set $\{m_i \in \mathbb{Z} \, | \, m_i \geq 0 \}$, does there exist a $\mathbb{Z}_2^s$-cover such that $m_i=\#\{g \neq 0:d_g=i\}$ ? And if the answer is yes, then how many such non-isomorphic covers exist ? \par

By \cite{pardini1991abelian}, this turns out to be a problem of solving a system of linear equations in the degrees $d_g$ of the branch divisors and the degrees $\ell_{\chi}$ of the line bundles which are the direct summands of $f_*\mathcal{O}_X$. But this system gets computationally cumbersome for large values of $s$. Instead we employ some techniques from Fourier transform of finite groups. We mention the highlights in the following points:

\begin{enumerate}
\item By considering the degree of the divisors $D_g$, we have a set theoretic map $d: G \to \mathbb{Z}_{\geq 0}$, $d(g) = \deg(D_g)$, and its unnormalized Fourier transform
$\widehat{d}(\chi) = \sum_g \chi(g) d(g)$.
 \item The first observation is that the function $\ell: \G^* \to \mathbb{Z}_{\geq 0}$ giving the degrees of the line bundles is related to the unnormalized Fourier transform $\widehat{d}: G^* \to \mathbb{Z}$ of the function $d: G \to \mathbb{Z}_{\geq 0}$ by 
 \[
\widehat{d}(0) - \widehat{d}(\chi) = 4 \ell(\chi),
\]
for $\chi \neq 0$ (see \Cref{lemma: Fourier}).

 \item The function $d$ can be recovered from $\widehat{d}$ and hence from $\ell$ (see \Cref{lemma: Inverse Fourier Transform}).

 \item Then we have the first, second and the cubic moment lemmas (see \Cref{lem:first-moment}, \Cref{lemma: Parseval}, \Cref{Plancherel for the unnormalized Fourier transform} and \Cref{lem:fourier-convolution-cubic}), which relate the sum of the first, second and third powers of $\widehat{d}$ and hence those of $\ell$ to the sum of the corresponding powers of $d$. 

 \item Now given a $d$-distribution $\{m_i\}$, we use the first moment lemma to find out the possible $\ell$-distributions $\{n_i\}$, where $n_i=\#\{\chi\neq 0:\ell(\chi)=i\}$. 

 \item Now the second and third moment lemmas provide natural constraints that rule out most of the possible $\ell$-distributions. 

 \item The remaining few $\ell$-distributions need to be analyzed separately. For this we use the fact that $d$ can be recovered from $\ell$ and further that $d$ is a positive integer-valued function. This adds more contraints and at the same time allows us to construct the functions $d$ whenever they are feasible (see for example \Cref{m = 2; k = 1; D = 9 possibility}, \Cref{s=4; m = 1; k = 2; D = 12 possibility}).
 
\end{enumerate}

It is worth noting that when either one of the weights is at least $2$ or when $p_g \geq 4$, the highest value of $m$ is $m = 4$ which is achieved by a $\mathbb{Z}_2^2$-cover over $\mathbb{P}^3$, while highest value of $s$ is $s = 5$ achieved by a canonical cover of $\mathbb{P}(1,1,2,2)$.

\begin{theorem}\label{thm:flat_pluricanonical_covers}
Let $X \rightarrow \mathbb{P}(a_0, a_1, a_2, a_3)$ be a flat $m$-canonical $\mathbb{Z}_2^s$-cover such that $L = \operatorname{lcm}(a_0, a_1, a_2, a_3)$. Then up to $\operatorname{GL}_s(\mathbb{F}_2)$ action,  $X$ is isomorphic to a $\mathbb{Z}_2^s$-cover of $\mathbb{P}(a_0,a_1,a_2,a_3)$ branched along the divisors $\{D_g \in |\mathcal{O}_{\mathbb{P}}(d_\mathbf{g})|\}$ with ${g \in \mathbb{Z}_2^s \setminus \{0\}}$, $m$-th plurigenus $p_m(X) = h^0(kL)$, where $\{d_g\}$, $g \in \mathbb{Z}_2^s \setminus \{0\}$ and $k$ as given in Definition \ref{def:admissible_weights}. In particular:
\begin{itemize}
 \item[(i)] For $s=1$ and $m=1$, there is a finite number of weighted projective threefolds $\mathbb{P}(a_0, a_1, a_2,a_3)$ that appear as a base, and for each of them an unbounded collection of positive integers that appear as degrees of the branch locus of the cover. 
\item[(ii)] For each of the following $s$ and $m$:
\begin{align*}
s=1, m\geq 2 && s = 2, m=1 &&    s = 2, m=2 && s = 2, m=3 \\
s = 3, m=1 && s=3, m=2 &&s= 4, m=1 &&s=5, m=1
\end{align*}
there is a finite number of weighted projective threefolds $\mathbb{P}(a_0, a_1, a_2,a_3)$ that appear as a base, and for each of them, a finite collection of tuples $\{d_g\}$, $g \in \mathbb{Z}_2^s \setminus \{0\}$ of positive integers which appear as the multidegree of the branch divisors 
$\{D_g\} \in |\mathcal{O}_{\mathbb{P}}(d_g)|$ with $g \in \mathbb{Z}_2^s \setminus \{0\} $. 
\item[(iii)] There are no such covers for any other values of $s$ and $m$.
\end{itemize}
\end{theorem}
 
We end the article by showing in \Cref{unbounded_non-flat_canonical} and \Cref{unbounded_non-flat_bicanonical} that there exists non-flat canonical and bicanonical $\mathbb{Z}_2^s$ covers for arbitrarily large values of $s$, thereby showing that there is no hope of a complete classification of such covers unless one puts conditions on the singularities of the cover.

\begin{definition}\label{def:admissible_weights}
Let $(a_0,a_1,a_2,a_3)$ be positive integers with $\gcd(a_i, a_j, a_k) = 1$ for distinct $i, j, k$. 
The weights $a_i$ and the degrees $d_g$ for each $s$ and $m$ as in Theorem~\ref{thm:flat_pluricanonical_covers} are as follows.

\begin{enumerate}
\item \textbf{$s = 1$, $m = 1$.} Write $d_{\mathbf{g}} := (d)$ and see Table \ref{tab:table1}.
\begin{table}[h!]
\begin{center}
\renewcommand{\arraystretch}{1.3}
\begin{tabular}{|
@{\hspace{0.7cm}} c @{\hspace{0.7cm}}
| @{\hspace{0.7cm}} c @{\hspace{0.7cm}}|
@{\hspace{0.7cm}} c @{\hspace{0.7cm}}|
@{\hspace{0.7cm}} c @{\hspace{0.7cm}}|}
\toprule
$(a_0,a_1,a_2,a_3)$ & $d$ & $k$ & $p_1(X)$ \\
\midrule
$(1,1,1,1)$   & $2t,\; t \ge 5$ & $t-4$ & $h^0(\mathcal{O}_{\mathbb{P}}(t-4))$   \\
$(1,1,1,3)$   & $6t,\; t \ge 3$ & $t-2$ & $h^0(\mathcal{O}_{\mathbb{P}}(3t-12))$  \\
$(1,1,2,2)$   & $4t,\; t \ge 4$ & $t-3$ & $h^0(\mathcal{O}_{\mathbb{P}}(2t-6))$   \\
$(1,1,2,4)$   & $8t,\; t \ge 2$ & $t-1$ & $h^0(\mathcal{O}_{\mathbb{P}}(4t-4))$   \\
$(1,1,4,6)$   & $24t,\; t \ge 2$ & $t-1$ & $h^0(\mathcal{O}_{\mathbb{P}}(12t-12))$ \\
$(1,2,2,5)$   & $20t,\; t \ge 2$ & $t-1$ & $h^0(\mathcal{O}_{\mathbb{P}}(10t-10))$ \\
$(1,2,3,6)$   & $12t,\; t \ge 3$ & $t-2$ & $h^0(\mathcal{O}_{\mathbb{P}}(6t-12))$  \\
$(1,2,6,9)$   & $36t,\; t \ge 2$ & $t-1$ & $h^0(\mathcal{O}_{\mathbb{P}}(18t-18))$ \\
$(1,3,4,4)$   & $24t,\; t \ge 2$ & $t-1$ & $h^0(\mathcal{O}_{\mathbb{P}}(12t-12))$ \\
$(1,3,8,12)$  & $48t,\; t \ge 2$ & $t-1$ & $h^0(\mathcal{O}_{\mathbb{P}}(24t-24))$ \\
$(1,4,5,10)$  & $40t,\; t \ge 2$ & $t-1$ & $h^0(\mathcal{O}_{\mathbb{P}}(20t-20))$ \\
$(1,6,14,21)$ & $84t,\; t \ge 2$ & $t-1$ & $h^0(\mathcal{O}_{\mathbb{P}}(42t-42))$ \\
$(2,3,10,15)$ & $60t,\; t \ge 2$ & $t-1$ & $h^0(\mathcal{O}_{\mathbb{P}}(30t-30))$ \\
\bottomrule
\end{tabular}
\end{center}
    \caption{Table for $s=1$ and $m=1$,  Definition \ref{def:admissible_weights}, item(1)}
    \label{tab:table1}
\end{table}

\item \textbf{$s=1$, $m \geq 2$.} Let $L = \operatorname{lcm}(a_0,a_1,a_2,a_3)$ and $W = \sum_i a_i$. The weights belong to the finite set
\[
\bigl\{ (a_0, a_1,a_2, a_3) \in \mathbb{N}_+^4 
\;\big|\;
m-1 \mid 2W,\; L \mid mW  
\bigr\},
\]
and for each such tuple, the degree $d_g$ belongs to the finite set
\[
\bigg\{  d_g = 2Lt \in \mathbb{N} 
\;\bigg|\;
\frac{W}{L} < t < \left(1+\frac{1}{m-1} \right)\frac{W}{L}
\bigg\}.
\]
\item \textbf{$s = 2$.} Write $d_{\mathbf{g}} := (d_{10}, d_{01}, d_{11})$ and see Table \ref{tab:table2}.

\begin{table}[h!]
\begin{center}
\renewcommand{\arraystretch}{1.2}
\begin{tabular}[h!]
{
| @{\hspace{0.7cm}} c @{\hspace{0.7cm}} |
@{\hspace{0.7cm}} c @{\hspace{0.7cm}} |
@{\hspace{0.7cm}} c  @{\hspace{0.7cm}} |
@{\hspace{0.7cm}} c @{\hspace{0.7cm}} |
@{\hspace{0.7cm}} c @{\hspace{0.7cm}}|}
\toprule
$m$ & $(a_0,a_1,a_2,a_3)$ & $(d_{10}, d_{01}, d_{11})$ & $k$ & $p_m$ \\
\midrule
$4$ & $(1,1,1,1)$ & $(3,3,3)$      & $2$ & $10$ \\
\midrule
$3$ & $(1,1,3,3)$ & $(6,6,6)$      & $1$ & $6$  \\
\midrule
    & $(1,1,1,2)$ & $(4,4,4)$      & $1$ & $7$  \\
$2$ & $(1,1,2,2)$ & $(2,6,6)$      & $1$ & $5$  \\
    & $(1,1,4,4)$ & $(8,8,8)$      & $1$ & $7$  \\
    & $(1,1,1,1)$ & $(4,4,2)$      & $2$ & $10$ \\
\midrule
    & $(1,1,2,2)$ & $(8,8,0)$      & $1$ & $5$  \\
    & $(1,1,2,2)$ & $(4,4,8)$      & $1$ & $5$  \\
    & $(1,1,1,2)$ & $(2,6,6)$      & $1$ & $7$  \\
    & $(1,1,1,3)$ & $(6,6,6)$      & $1$ & $11$ \\
    & $(1,1,2,4)$ & $(8,8,8)$      & $1$ & $10$ \\
    & $(1,1,4,4)$ & $(4,12,12)$    & $1$ & $7$  \\
    & $(1,2,3,6)$ & $(12,12,12)$   & $1$ & $8$  \\
$1$ & $(1,1,4,4)$ & $(12,12,12)$   & $2$ & $22$ \\
    & $(1,1,1,2)$ & $(6,6,6)$      & $2$ & $22$ \\
    & $(1,1,2,2)$ & $(4,8,8)$      & $2$ & $14$ \\
    & $(1,1,2,2)$ & $(8,8,8)$      & $3$ & $30$ \\
    & $(1,1,1,1)$ & $(6,6,2)$      & $3$ & $20$ \\
    & $(1,1,1,1)$ & $(4,4,6)$      & $3$ & $20$ \\
    & $(1,1,1,1)$ & $(6,6,4)$      & $4$ & $35$ \\
    & $(1,1,1,1)$ & $(6,6,6)$      & $5$ & $56$ \\
\bottomrule
\end{tabular}
\end{center}
   \caption{Table for case $s=2$, $m=1$, Definition \ref{def:admissible_weights}, item (2)}
    \label{tab:table2}
\end{table}

\item \textbf{$s = 3$.} Write $d_{\mathbf{g}} := (d_{100}, d_{010}, d_{110}, d_{001}, d_{101}, d_{011}, d_{111})$, and see Table \ref{tab:Table3}.

\begin{table}[h!]
\begin{center}
\renewcommand{\arraystretch}{1.2}
\begin{tabular}{
|@{\hspace{0.7cm}} c @{\hspace{0.7cm}} | @{\hspace{0.7cm}} c @{\hspace{0.7cm}}| @{\hspace{0.7cm}} c @{\hspace{0.7cm}} | @{\hspace{0.7cm}} c  @{\hspace{0.7cm}}| 
@{\hspace{0.7cm}} c @{\hspace{0.7cm}} |}
\toprule
$m$ & $(a_0,a_1,a_2,a_3)$ & $(d_g)$, $g \in G \setminus \{0\}$ & $k$ & $p_m$ \\
\midrule
$2$ & $(1,1,2,2)$ & $(2,2,2,2,2,2,2)$   & $1$ & $5$  \\
\midrule
    & $(1,1,1,2)$ & $(2,2,2,2,2,2,2)$   & $1$ & $7$  \\
    & $(1,1,1,2)$ & $(4,0,4,0,4,0,4)$   & $1$ & $7$  \\
$1$ & $(1,1,2,2)$ & $(2,2,0,2,4,4,2)$   & $1$ & $5$  \\
    & $(1,1,4,4)$ & $(4,4,4,4,4,4,4)$   & $1$ & $7$  \\
    & $(1,1,1,1)$ & $(2,2,2,2,2,2,2)$   & $3$ & $20$ \\
\bottomrule
\end{tabular}
\end{center}
    \caption{Table for case $s=3$, $m=1$ and $m=2$, Definition \ref{def:admissible_weights} item (iii)}
    \label{tab:Table3}
\end{table}

\item \textbf{$s = 4$, $m = 2$.} 
$(a_0,a_1,a_2,a_3) = (1,1,1,1)$, \; $k = 1$, \; $p_2 = 4$.
\[
d_g = 
\begin{cases}
2, & g = (1,1,1,1), \\[3pt]
1, & g \in \{(0,1,0,0),\, (1,0,0,0),\, (1,0,0,1),\, (1,0,1,0),\\
   & \hphantom{g \in \{}(1,1,0,0),\, (1,1,0,1),\, (1,1,1,0)\}, \\[3pt]
0, & \text{otherwise.}
\end{cases}
\]
\item \textbf{$s = 4$, $m = 1$.} Fix two non-zero characters $\chi_0, \chi_1 \in G^* \setminus \{1\}$.

\begin{enumerate}
\item[(a)] $(a_0,a_1,a_2,a_3) = (1,1,2,2)$, \; $k = 1$, \; $p_1 = 5$.
\[
d_g = 
\begin{cases}
0 & \text{if } \chi_0 \cdot g = 0, \\
2 & \text{if } \chi_0 \cdot g = 1.
\end{cases}
\]

\item[(b)] $(a_0,a_1,a_2,a_3) = (1,1,2,2)$, \; $k = 1$, \; $p_1 = 5$.
\[
d_g = 
\begin{cases}
0 & \text{if } \chi_0 \cdot g = \chi_1 \cdot g = 0, \\
1 & \text{if } \chi_0 \cdot g = 0,\; \chi_1 \cdot g = 1, \\
1 & \text{if } \chi_0 \cdot g = 1,\; \chi_1 \cdot g = 0, \\
2 & \text{if } \chi_0 \cdot g = \chi_1 \cdot g = 1.
\end{cases}
\]

\item[(c)] $(a_0,a_1,a_2,a_3) = (1,1,1,1)$, \; $k = 2$, \; $p_1 = 10$.
\[
d_g = 
\begin{cases}
0, & g \in \{(0,0,0,1),\, (0,0,1,0),\, (0,0,1,1)\}, \\
1, & \text{otherwise.}
\end{cases}
\]
\end{enumerate}

\item \textbf{$s = 5$, $m = 1$.} Fix one non-zero character $\chi_0 \in G^* \setminus \{1\}$.
$(a_0,a_1,a_2,a_3) = (1,1,2,2)$, \; $k = 1$, \; $p_1 = 5$.
\[
d_g = 
\begin{cases}
0 & \text{if } \chi_0 \cdot g = 0, \\
1 & \text{if } \chi_0 \cdot g = 1.
\end{cases}
\]
\end{enumerate}
\end{definition}

\subsection{Notation}
Unless explicitly said otherwise, the group $G$ denotes $\mathbb{Z}_2^s$ and $\chi(G)$ denotes its group of characters. The identity element of $G$ is denoted as $0 \in G$, and the trivial character as $1 \in G*$. In our case, the theory of abelian covers associates a divisor $D_g$ to each element $g \in G \setminus 0$. For simplicity, we associate the trivial divisor with the identity, so $D_0 = 0$ always. 

\subsection{Acknowledgments} 
We thank Long Horizon Research
for providing us with the Sundial platform and the AI tools to proofread the arguments and presentation of this paper.
Patricio Gallardo is partially supported by the National Science Foundation under Grant No. DMS-2316749. 
\tableofcontents

\section{Preliminaries}
\subsection{Abelian covers}\label{sec:abelianCovers}

We recall the description of abelian covers following Pardini~\cite{pardini1991abelian} and the extension to singular bases by Alexeev--Pardini~\cite{AP12}, focusing on the cases relevant to our work. 
Throughout, $G$ denotes the finite abelian group $\mathbb{Z}_2^s$ and  
$f\colon X\to Y$ a $G$-cover (i.e. a finite morphism endowed with a faithful 
$G$-action on $X$ 
such that $Y=X/G$), and $G^*:=\operatorname{Hom}(G,\mathbb{C}^*)$ its character group.

In our work, we assume that $Y$ is normal, irreducible, and $X$ is normal. The $G$-action yields the eigensheaf decomposition
\begin{equation}\label{eq:decomposition}
f_*\OO_X \;=\; \bigoplus_{\chi\in G^*} L_\chi^{-1},
\qquad
L_{1}\simeq \OO_Y,
\end{equation}
where $1\in G^*$ is the trivial character and $G$ acts on $L_\chi^{-1}$ via $\chi$.

If $Y$ is smooth, or if it's the smooth locus, then, by purity of the branch locus, the set of ramification points has divisorial image on $Y$. Let $D\subset Y$ be the reduced branch divisor and $R\subset X$ the reduced ramification divisor. Each irreducible component
$T\subset R$ has an inertia subgroup $H_T\le G$, namely the subgroup of elements fixing $T$ pointwise which in our case it is always isomorphic to $\mathbb{Z}_2$. Moreover, attached to $T$ there is a distinguished generator $\psi_T\in H_T^*$ describing the local
action in a transversal parameter. Grouping components with the same pair $(H,\psi)$ produces a decomposition
\[
R \;=\; \sum_{(H,\psi)} R_{H,\psi},
\qquad
D \;=\; \sum_{(H,\psi)} D_{H,\psi},
\quad
\text{with } R_{H,\psi}=f^{-1}(D_{H,\psi}) \text{ as reduced divisors.}
\]
We refer to the collection $\{L_\chi\}_{\chi\in G^*}$ together with the reduced effective divisors
$\{D_{H,\psi}\}$ as the \emph{building data} of the cover. In our case of $G  = \mathbb{Z}_2^s$, the generator of $H_T^*$ is unique, so it is convenient to index the divisors by
$g\in G$ and write
\begin{equation}\label{eq:Dg_decomp}
D \;=\; \sum_{g\in G} D_g,
\qquad D_0=0,
\end{equation}
where $D_g$ is the reduced divisor whose pullback is the corresponding summand of the ramification divisor.

The building data are constrained by the \emph{fundamental relations}. 
To state them for our case, fix $g\in G$ and $\chi\in G^*$ such that $g \neq 0$ and $\chi \neq 1$, there is a unique integer
$0\le r_g^\chi\le 1$ such that \(\chi(g)=\exp\!\left(\pi i\,r_g^\chi\right), \) that is
\begin{align*}
r_g^{\chi} = 
\begin{cases}
 0 & \text{$\chi$ acts trivially on $g$, that is } \chi(g) = 1
 \\
 1 & \text{$\chi$ acts non-trivially on $g$, that is } \chi(g) = -1
\end{cases}
\end{align*}
For $\chi,\chi'\in G^*$ set
\begin{align*}
 \varepsilon^{g}_{\chi,\chi'}=
&
\begin{cases}
1,& r_g^\chi+r_g^{\chi'}\ge 2,\\
0,& \text{otherwise.}
\end{cases}   
\\
=
&\begin{cases}
1, &  \text{Both $\chi$ and \( \chi'\) acts non-trivially on } g,\\
0,& \text{otherwise.}
\end{cases}  
\end{align*}
Then for all $\chi,\chi'\in G^*$ one 
has an isomorphism
\begin{equation}\label{eq:fundamental_relations_mult}
L_\chi\otimes L_{\chi'} \;\simeq\;
L_{\chi\chi'}\otimes \OO_Y\!\left(\sum_{g\in G}\varepsilon^{g}_{\chi,\chi'}\,D_g\right)
\simeq
L_{\chi\chi'}\otimes \OO_Y\!\left(\sum_{\substack{g\in G\\ \chi(g)=\chi'(g)=-1}} D_g\right).
\end{equation}
If moreover $\Pic(Y)$ is torsion-free, we may write~\eqref{eq:fundamental_relations_mult} additively in $\Pic(Y)$ as
\begin{align}
\label{eq:fundamental_relations}
L_\chi+L_{\chi'} \;\equiv\; L_{\chi\chi'} + \sum_{\substack{g\in G\\ \chi(g)=\chi'(g)=-1}} D_g.
\end{align}
Taking $\chi'=\chi$ yields the frequently used relation when $G = \mathbb{Z}_2^s$
\begin{equation}\label{eq:half_sum_relation}
2L_\chi \;=\; \sum_{\substack{g\in G\\ \chi(g)=-1}} D_g,
\qquad \chi\neq 1.
\end{equation}
In practice, it suffices to impose~\eqref{eq:half_sum_relation} for a basis $\{\chi_1,\dots,\chi_s\}$ of $G^*$ which is viewed as an $\F_2$--vector space since the remaining relations follow formally from~\eqref{eq:fundamental_relations}.
The collection of equations~\eqref{eq:half_sum_relation} for such a basis is known as \emph{reduced building data}.
Moreover, if $Y$ is smooth, irreducible with $H^0(\mathcal{O}_Y^*) = \mathbb{K}^*$ and one is given line bundles $\{L_\chi\}$ and reduced effective divisors
$\{D_g\}$ satisfying~\eqref{eq:fundamental_relations}, then there exists a unique, up to isomorphism of $G$--covers, normal $G$--cover $f\colon X\to Y$ with such building data, see \cite[Proposition $2.1$]{pardini1991abelian}.


We now summarize the adjustments needed when $Y$ is only normal. Following~\cite{AP12}, assume that $Y$ is $S_2$
and regular in codimension one, and that $f\colon X\to Y$ is a $G$--cover with $X$ an $S_2$ scheme. A basic extension principle is that $G$--covers over an open subset extend uniquely across codimension $\ge2$.
\begin{lemma}[{\cite[Lem.~1.2]{AP12}}]\label{lem:AP_extension}
Let $Y$ be an $S_2$ scheme and let $j\colon Y_0\hookrightarrow Y$ be an open subset with $\operatorname{codim}(Y\setminus Y_0)\ge2$.
If $f_0\colon X_0\to Y_0$ is a $G$--cover with $X_0$ $S_2$, then there exists a unique $G$--cover $f\colon X\to Y$
with $X$ $S_2$ whose restriction to $Y_0$ is $f_0$.
\end{lemma}

Consequently, the formalism above persists with the following replacements (loc.\ cit.):
\begin{enumerate}
\item In~\eqref{eq:decomposition}, the eigensheaves $L_\chi^{-1}$ are reflexive sheaves of rank one on $Y$
(invertible on the smooth locus of $Y$).
\item The divisors $D_g$ (or more generally $D_{H,\psi}$) are Weil divisors; their restrictions to the smooth locus are Cartier.
\item The relations~\eqref{eq:fundamental_relations}  still hold, interpreted in the
group of rank-one reflexive sheaves.
\item The morphism $f$ need not be flat. It is flat precisely when all eigensheaves $L_\chi$ are line bundles.
\item Given a $G$-cover $f : X \to Y$ and an irreducible subset $S \subset Y$ , we define the inertia subgroup $H_S$ of $S$ to be the subgroup of $G$ consisting of the elements that fix $f^{-1}(S)$ pointwise, or, equivalently since $G$ is abelian, that fix an irreducible component of $f^{-1}(S)$ pointwise. The branch locus $D_f$ of $f$ is the set of points of $Y$ whose inertia subgroup is not trivial. As mentioned in \cite{AP12},  we regard $D_f$ simply as a set, without giving it a scheme structure. If $f$ is flat, then $D_f$ is a divisor. If $f$ is not flat, then the branch locus may have non-divisorial components.
\end{enumerate}

Let $G$ be  $\mathbb{Z}_2^s$ and 
$\{ D_g \, | \, g \in G, D_0 =0 \}$ be the building data of the cover $f:X \rightarrow Y$. 
It is well-known that when $Y$ is smooth  
the canonical bundle of $X$, see   \cite[Prop 4.2]{pardini1991abelian}, is given by
\begin{equation}
\label{eq:canonical_formula_Z2s}
K_X \;=\; f^*\!\left(K_Y + \frac12\sum_{g\in G} D_g\right).
\end{equation}  
When $Y$ is normal, we use the next expression given by \cite{AP12} and the discussion before \cite[Lemma 2.10]{alexeev2009explicit})
\begin{lemma}
\label{lemma:general_canonical_formula_Z2s}
Let $f : X \to Y$ be a $\mathbb{Z}_2^s$-cover of a deminormal variety . Let $E$ be a prime divisor of Y, set
\begin{align*}
a_E = 
\begin{cases}
0 &
  \text{ 
  $f$ is generically \'etale over $E$ or  $E$ is contained in the double locus
  }
  \\
 1 & 
  \text{
$Y$ is generically smooth along $E$ but $X$ is singular along $f^{-1}(E)$
  }
  \\
   \displaystyle\frac{1}{2} &
  \text{
otherwise,
  }
\end{cases}
\end{align*}
and define the Hurwitz divisor $D_{\operatorname{Hur}}$ to be the divisor $\sum_E a_E E$.
Then, it holds that
\begin{enumerate}
\item[(i)] $K_X$ is $\mathbb{Q}$-Cartier if and only if $K_Y +D_{\operatorname{Hur}}$ is  $\mathbb{Q}$-Cartier, 
\item[(ii)] If $K_X$ is $\mathbb{Q}$-Cartier, 
$K_X =  f^*\left( K_Y + D_{\operatorname{Hur}}\right)
$,  
\item[(iii)] If $Y$ is normal, the Hurwitz divisors are given in terms of the building data by
\[
D_{\mathrm{Hur}}=\frac12\sum_{g\ne 0}D_g 
\]
\end{enumerate}
\end{lemma}

\section{Proof of Theorem \ref{thm:main_inv}}
\label{sec:proof_invariants_theorem}

We start with setting up the notation and key definitions in Section ~\ref{sec:set-up}. The calculation of the invariants 
$K_X^3$, $\chi(\mathcal{O}_X)$, and $e(X)$ for our 
$\mathbb{Z}_2^s$-covers is completed in Section \ref{sec:calculation_invariants}. We calculate each bound in Section \ref{sec:ratios}.

\subsection{Setup and notation}\label{sec:set-up}
The first paper to systematically study the geography of threefolds was that of Bruce Hunt (\cite{huntthreefolds89}).
As in that work, we suppose that our threefolds are $\mathbb{Q}$-Gorenstein threefolds of general type, and we focus on their Chern numbers. These Chern numbers can be written in terms of the volume $K_X^3$, the holomorphic Euler characteristic 
$\chi:=\chi(\mathcal{O}_X)$ and the topological Euler characteristic $e(X)$ as follows
\begin{align*}
c_1^3(X) = -K_X^3,
&&
c_1(X)c_2(X) = 24 \chi( \mathcal{O}_X),
&&
c_3(X) = e(X)
\end{align*}
Minimal models exist for threefolds, but they are not unique. So, for any $\mathbb{Q}$-Gorenstein threefold $X$ of general type, we have the birational equivalence classes $\mathcal{B}(X)$ of minimal models of $X$. By work of Kawamata (see \cite[Section $7.1.2$]{huntthreefolds89}), there is a well-defined map 
\begin{align*}
\mathcal{B}(X) \mapsto  \mathbb{P}^2(\mathbb{Q})
&&
X \mapsto \big[ 
-K_X^3  , 24\chi(\mathcal{O}_X), 
e(X)
\big] \in \mathbb{P}^2(\mathbb{Q})
\end{align*}
such that for a smooth minimal model, it coincides with the map 
\(
X \mapsto \big[c_1^3 , c_1c_2, c_3 \big].
\)

We will study the ratios of Chern numbers for the case  where
$X \rightarrow Y$ is a $\mathbb{Z}_2^s$-cover, and the branch locus $D_g$ satisfies the following hypothesis that prevents certain pathologies such as non-terminal singularities on $X$, which is a needed for defining the above map from $\mathcal{B}(X)$ to 
$\mathbb{P}^2(\mathbb{Q})$.
\begin{definition}
\label{def:divisor_standard set-up}
Let $X \rightarrow \mathbb{P}(a_0, a_1, a_2, a_3)$ be a 
$\mathbb{Z}^s_2$-cover of a well-formed weighted projective threefold. Let $D_g$ be the divisorial component (possibly a Weil divisor) of the branch locus with inertia group generated by $g \in G$. We say the branch locus is generic if the following holds:
\begin{enumerate}
\item All the branch divisors $D_g$ are well-formed and quasi-smooth.
\item All the double intersections $D_{p, g} := D_{p} \cap D_{g}$ are well-formed and quasi-smooth.
    \item All triple intersections 
    $D_{p,q,g} := D_{p} \cap D_{q} \cap D_{g}$ are smooth and away from the singular locus.
    \item There are no other intersections, that is $\bigcap_{i = 1}^t D_{g_i} = \emptyset$ for all $t \geq 4$.
\end{enumerate}
\end{definition}

We start by describing the singularities of the cover $X$ when the branch locus is generic as in \Cref{def:divisor_standard set-up}.

\begin{lemma}\label{lemma:singDpqr}\label{singularities of the cover}
 Let $X \rightarrow \mathbb{WP}^3$ be a  $\mathbb{Z}_2^s$-cover branching along 
generic divisors $D_g$. Then, the singularities of $X$ are as follows:
\begin{enumerate}
    \item If 
    $y \in \mathbb{WP}^3 \setminus 
    \text{Sing}(\mathbb{WP}^3)$. Then, $X$ is either smooth over $y$ or it has a terminal $\frac{1}{2}(1,1,1)$ singularity over $y$ only when  $y \in D_p \cap D_q \cap D_r$ and $p+q+r=0$.
    \item $y \in \text{Sing}(\mathbb{WP}^3)$, then $X$ has singularities either isomorphic to the ones of $\mathbb{WP}^3$ or a double cover of them, which is also a quotient singularity.
\end{enumerate}
In particular, $X$ is terminal if $\mathbb{WP}^3$ is terminal and the branch locus is away from the singularities of it.
\end{lemma}
\begin{proof}
Let $y \in \mathbb{WP}^3 \setminus \text{Sing}(\mathbb{WP}^3)$. 
The hypothesis on our divisors $D_g$ implies by Pardini's result \cite[Prop 3.1]{pardini1991abelian} that the $\mathbb{Z}_2^s$-cover is singular over $y$ if and only if   
\( y \in \cap_{h \in I} D_{h}\) and the map 
$\oplus_{h \in I} \langle h \rangle \rightarrow  \mathbb{Z}_2^s$ is not an injection. By hypothesis, our divisors are generic, so at most three of them intersect, that is, $|I| \leq 3$. 
On the other hand, if  $s \geq 2$, then the map $\oplus_{h \in I} \langle h \rangle \rightarrow  \mathbb{Z}_2^s$ fails to be injective only for $|I| =3$.  After a linear change of coordinates in $\mathbb{Z}_2^s$, we can assume that $I = \{h_1, h_2, h_3 \}$ where either 
$h_1 = (1,0,0 \cdots, 0)$, $h_2 = (0,1,0 \cdots, 0)$, 
and $h_3 = (0,0,1 \cdots, 0)$ or 
$h_1 = (1,0,0 \cdots, 0)$, $h_2 = (0,1,0 \cdots, 0)$, 
and $h_3 = (1,1,0 \cdots, 0)$. In the first case, the map 
$\oplus_{h \in I} \langle g_i \rangle \rightarrow  \mathbb{Z}_2^s$ is an injection for $s \geq 2$, so $X$ is also smooth over $y$.  To describe the second case, we observe that the stabilizer group of the point $y$ is $\mathbb{Z}_2^2$, so at an open $U_y \subset \mathbb{WP}^3$, the $\mathbb{Z}_2^s$ abelian cover factors as 
\[
\pi^{-1}({U_y}) \subset X \rightarrow V_y \rightarrow U_y.
\]
The map $\pi^{-1}(U_y) \rightarrow V_y $ is etale with group  $G / \langle h_1, h_2, h_3 \rangle \cong \mathbb{Z}_2^{s-2}$. In contrast, 
the map $V_y \rightarrow U_y$ is a $\mathbb{Z}_2^2$-cover. As a consequence, the singularities at $X$ over $y$, are isomorphic to the ones that appear in $V_y$.
To find the singularities on the bi-double cover 
$V_y \rightarrow U_y $
with $y \in D_{h_1} \cap D_{h_2} \cap D_{h_3}$, we write the local equation of $V_y$ over $y$, as in \cite[Eq (2.4)]{catanese1984moduli}. Let $w_i$ with $i\in \{1, 2, 3\}$ be the fiber coordinates of the line bundles $\mathcal{L}_i$ associated to the bi-double cover, and let $x_j$ with $j \in \{1, 2, 3\}$ be the local equation of the divisors $D_{h_j}$. It holds that
\begin{align*}
V_y = 
\{
(w_1, w_2, w_3, x_1, x_2, x_3) \in \mathbb{A}^6\, | \, 
w_i^2 = x_jx_k,\,  x_kw_k = w_iw_j
\}.
\end{align*}
These equations define a Veronese surface $v_2(\mathbb{P}^2)$ in $\mathbb{P}^5$. Therefore, the germ of the singularity $\left( V_y, 0 \right)$ is isomorphic to 
$\left( 
\text{Cone}\left( 
v_2(\mathbb{P}^2)
\right), 0
\right)$ which is isomorphic to $\frac{1}{2}(1,1,1)$.

Next, we consider the case $y \in  \text{Sing}(\mathbb{WP}^3)$. By our hypotheses $y$ is contained in at most one $D_h$ because $D_{h_1,h_2}$ and $D_{h_1,h_2,h_3}$ are away from the singular locus of $\mathbb{WP}^3$.   If $y$ is not contained in any $D_h$, then the cover is etale over $y$, and $X$ has the same singularities than $\mathbb{WP}^3$. If $y$ is contained in a $D_h$, the $(\mathbb{Z}_2)^s$-cover factors as 
\[
\pi^{-1}({U_y}) \subset X \rightarrow V_y \rightarrow U_y.
\]
The map $\pi^{-1}(U_y) \rightarrow V_y $ is etale with group  $G / \langle h \rangle \cong \mathbb{Z}_2^{s-1}$, and $V_y \rightarrow U_y$ is a double-cover, and a double cover of a quotient singularity
is a quotient singularities as well. 
\end{proof}

\subsection{Calculating invariants of threefolds}\label{sec:calculation_invariants}

Following \cite[Sec 6-10]{huntthreefolds89},  we focus on the self-intersection of the canonical divisor, the holomorphic Euler characteristic, and the topological Euler characteristic to classify our covers.

\begin{lemma}\label{lemma:K3}
Let $X \rightarrow \mathbb{P}(a_0,a_1,a_2,a_3)$ be a $\mathbb{Z}_2^s$-cover, $H_{gen}$ be the generator of the Picard group $Pic_{\mathbb{Q}}(\mathbb{WP}^3)$ and $d_g$ be the degree of the divisor $D_g = d_g H_{gen}$ where $d_0 =0$,
and there is at least one $d_g > 0$.
Then
\begin{align*}
    K_X^3  = 
  \frac{2^s}{\prod_{i=0}^3 a_i}   
  \left(
   \frac{1}{2}\sum_{g \in G} d_{g}
  - \sum_{i=0}^3 a_{i} 
  \right)^3
= 
 \frac{2^{s-3}}{\prod_{i=0}^3 a_i}  
 \left( \sum_{g \in G} d_g \right)^3   
 + O\left( d_{g_a}d_{g_b} \right) 
\end{align*}
\end{lemma}
\begin{proof}
By Lemma \ref{lemma:general_canonical_formula_Z2s}, the canonical self-intersection of the $\mathbb{Z}_2^s$ abelian cover
$X \rightarrow \mathbb{WP}^3$ is equal to 
\begin{align*}
  K_X^3 =  2^{s}\left( 
  K_{\mathbb{WP}^3} + \frac{1}{2}\sum_{g \in G} D_{g}
  \right)^3
\end{align*}
Since $\mathbb{WP}^3$ is a toric variety, its canonical divisor can be described in terms of the toric boundary, which yields:
\begin{align*}
  K_{\mathbb{WP}^3} + \frac{1}{2}\sum_{g \in G} D_{g}
  =
  \left(
  - \sum_{i=0}^3 a_{i} + \frac{1}{2}\sum_{g \in G} d_{g}
  \right)H_{gen}
\end{align*}
Our result now follows from 
\(
(H_{gen})^3 = \frac{1}{\prod_{i=0}^3 a_i}.
\)
by \cite[lemma 12.5.2]{cox2024toric}
\end{proof}

\begin{lemma}
\label{lemma:holomorphic_euler_limit_chi_K3}
Let $f:X \to \mathbb{P}(a_0, \ldots, a_3)$ be a $\mathbb{Z}^s_2$-cover with notation as in Lemma \ref{lemma:K3}. Then, 
\begin{align*}
 \chi(\mathcal{O}_X) = 
\frac{-1}{48\prod_{i=0}^3 a_i}
\sum_{\chi \in G^*}
\left( \sum_{\chi(g)=-1} d_g \right)^3 
+ 
O\left( \left( 
\sum d_{g})^2
\right) \right)
\end{align*}
\end{lemma}
\begin{proof}
We start with a direct calculation:
\begin{align*}
 \chi(\mathcal{O}_X) &=
 \chi \left( 
 f_* \left( 
 \mathcal{O}_X
 \right) 
 \right)
&& \text{ since $f$ is finite}
 \\
 &=
\chi \left( 
 \bigoplus_{\chi \in G^*}
 \mathcal{O}_Y \left( L_{\chi}^{-1}\right)
 \right)
  && \text{ by 
  \cite[Eq (1)]{AP12} and $\chi(L_1) = 0$
  } 
 \\
&=
1 + \sum_{\chi \in G^*} \chi(L^{-1}_{\chi})
 && 
 \text{ 
additivity of  $\chi$ and $\chi(\mathcal{O}_{\mathbb{WP}^3}) = 1$
 }
\\
&=
- \frac{1}{6}\sum_{\chi \in G^*}\deg(L_{\chi}^3) + 
O \left( \deg(L_{\chi}^2) \right)
 && \text{ by Hirzebruch–Riemann–Roch}
\\
&=
- \frac{1}{6}\sum_{\chi \in G^*}
\left(
\sum_{g \, | \, \chi(g)=-1} 
\frac{1}{2}
D_g
\right)^3
+ 
O\left( \left( 
\sum d_{g})^2
\right) \right)
 && \text{ by Equation ~\eqref{eq:half_sum_relation}}
\\
&=
\frac{-1}{48 \prod_{i=0}^3a_i}\sum_{\chi \in G^*}
\left(
\sum_{g \, | \, \chi(g)= -1} d_g
\right)^3
+ 
O\left( \left( 
\sum d_{g}
\right)^2 \right)
 && 
\end{align*}
where the last equality follows 
by writing $D_g = d_g H$ and using
 $H_{gen}^3
 = 
 \frac{1}{\prod_{i=0}^3a_i}.
 $
\end{proof}


To describe the Euler characteristic of $X$, we divide $Y$ into a disjoint union of strata $S_i$ such that the number of pre-images is fixed on every strata. Then, we use the additivity of the Euler characteristic to conclude our result.  The details are given next.
\begin{lemma}\label{lemma:strata}
Let $f: X \to Y$ be a $\mathbb{Z}_2^s$ cover with branch data equal to $\{D_g \, |\, g \in G, \, D_0 = 0\}$. 
Assume that the conditions of Definition \ref{def:divisor_standard set-up} hold, and let 
\(D_{p,q,r}:=D_p \cap D_q \cap D_r \) and
\( D_{p,q}:= D_p \cap D_q  \). 
Then, it holds that
\(  Y = \bigsqcup_{i=1}^{4} S_i \)
where 
\begin{align*}
S_4 &:=\bigsqcup_{\{(p,q,r) | \, p < q < r\}}  
    D_{p,q,r}
    &&
S_3 :=
    \bigsqcup_{\{(p,q) | p < q\}} 
    \left( 
   D_{p,q} \setminus
   \bigcup_{\{r | r \neq p, q\}}D_{p,q,r}
    \right)    
\\
S_2 &:=  
    \bigsqcup_p 
    \left( 
    D_p \setminus
    \bigcup_{\{q | q \neq p\}}D_{p,q} 
    \right)
&&     
S_1 := Y \setminus \bigcup_p D_p.  
\end{align*}
and  $\dim(S_4) = 0$, $\dim(S_3) = 1$, $\dim(S_2) = 2$, and  $\dim(S_1) = 3$.  
\end{lemma}
For the following result, recall that given three different non-zero elements
 $p,q$ and $r$ in $\Z_2^s$, we 
can interpret them as vectors over $\Z_2$; and that they are linearly dependent if and only if
$p+q+r=0$.  
\begin{lemma}\label{lemma:e(strata)}
Let the notation be as in Lemma \ref{lemma:strata}.
Then, the following holds:
\begin{align*}
e(f^{-1}(S_4)) &=
2^{s-2} 
\sum_{\substack{p<q<r \\ p+q+r =0}}
e(D_{p,q,r})
+
2^{s-3} 
\sum_{\substack{p < q <r \\ p+q+r \neq 0} }
e(D_{p,q,r})
\\
e\big(f^{-1}(S_3)\big)
&=
2^{s-2}\sum_{p < q} \left(
e(D_{p,q}) - \sum_{r \, | \, r \neq p,q}
e(D_{p,q,r})
\right)  
\\
e\bigl(f^{-1}(S_2)\bigr)
&=
2^{s-1}  \sum_{p}\left(
e(D_p) - 
\sum_{ q \, | \, q \neq p} e(D_{p,q}) + 
\sum_{\substack{q<r\\ q,r\neq p}}
e(D_{p,q,r})
  \right)
  \\
e(f^{-1}(S_1))
&= 2^{s} \left( e(Y) - 
\sum_p e\left( D_p \right) + 
\sum_{p < q} e\left( D_{p,q} \right) -
\sum_{p < q < r} e\left( D_{p, q, r} \right)
\right)
\end{align*}  
\end{lemma}
\begin{proof}
In the case $S_4$, the locus $D_{p,q,r}$ is the union of disjoint points.  To calculate $e(f^{-1}(S_4))$, we need to consider the number of closed points in the preimage of each point in $D_{p,q,r}$. Recall that for each $D_p$, the subgroup generated by $p$, that is $\langle p \rangle$, fixes the preimage of $D_p$ pointwise (the inertia subgroup of $D$). Therefore, the number of preimages 
in the intersection $D_p \cap D_q \cap D_r$ is equal to $2^{s}/|\langle p, q, r \rangle|$. 
By hypothesis $p$, $q$ and $r$ are all distinct, we have only two options: 
$p+q+s = 0$ and $p+q+s \neq 0$.
In the first case, $p+q+r = 0$, then the group generated by $\langle p,q,r \rangle$ is isomorphic to $\Z_2^2$. In contrast, when $p+q+r \neq 0$. Then, the group generated by $ \langle p,q,r \rangle$ is isomorphic to $\Z_2^3$. Now, let's talk about the number of pre-images. In the case $p+q+r =0$, we have $2^s/2^2 = 2^{s-2}$. In the other case, we have $2^{s-3}$. This implies 
\begin{align*}
e\big(f^{-1}(S_4)\big) 
&=
2^{s-2} 
\sum_{\substack{p<q<r \\ r+q+r  = 0}}
e(D_{p,q,r})
+
2^{s-3} 
\sum_{\substack{p < q <r \\ p+q+r \neq 0} }
e(D_{p,q,r})
\end{align*}
For the next case, let's recall that
\begin{align*}
S_3=
\bigsqcup_{p<q} U_{p,q}, 
&&
U_{p,q}:=
D_{p,q}\setminus \bigcup_{r\, | \, r\neq p,q} D_{p,q,r}
\end{align*}
On $U_{p,q}$ the fiber cardinality is $2^{\,s-2}$, 
and by the inclusion-exclusion principle, we have:
\begin{align*}
e\big(f^{-1}(S_3)\big) =
2^{s-2}\sum_{p<q}e(U_{p,q}),
&&
e(D_{p,q}) = 
e 
\left( U_{p,q}\right)
+
e \left(
\bigcup_{r\, | \, r\neq p,q} D_{p,q,r} 
\right)
\end{align*}
Therefore, we have that
\begin{align*}
e\big(f^{-1}(S_3)\big) 
 =
2^{s-2}\sum_{p<q} \left(
e(D_{p,q}) - \sum_{r \, | \, r \neq p,q}
e(D_{p,q,r})
\right)
\end{align*}
For the case $S_2$, we recall that over a point of 
\[
U_p := D_p \setminus \bigcup_{q \, | \, q \neq p} D_{p,q}
\]
the fiber cardinality is constant and equal to $2^{\,s-1}$. By the inclusion–exclusion principle and the fact that \(U_p \cap U_q = \emptyset\) for all \(p \neq q\), it holds that
\begin{align*}
e\bigl(f^{-1}(S_2)\bigr) =\sum_{p} 2^{\,s-1}\, e(U_p).
&&
e(U_p) 
=
e(D_p)- e\left( \bigcup_{q\, | \, p \neq q }D_{p,q} \right).
\end{align*}
In general, 
the inclusion-exclusion principle implies
\begin{align*}
   e\left( \bigcup_{q\, | \, p \neq q }D_{p,q} \right) 
=
\sum_{ q \, | \, q \neq p} 
e(D_{p,q})- 
\sum_{q < r} 
e(D_{p,q} \cap D_{p,r})
+
\sum_{q < r < s} 
e(D_{p,q} \cap D_{p,r} \cap D_{p,s}
) - \cdots
\end{align*}
However, $D_{p,q} \cap D_{p,r} \cap D_{p,s} =  D_{p,q,r,s}$ which by hypothesis is empty. Therefore
\begin{align*}
   e\left( \bigcup_{q\, | \, p \neq q }
   D_{p,q} \right) 
=
\sum_{ q \, | \, q \neq p} 
e(D_{p,q})- 
\sum_{q < r} 
e(D_{p,q} \cap D_{p,r})
=
\sum_{ q \, | \, q \neq p} 
e(D_{p,q})- 
\sum_{q < r \, |\,  p \neq q,r} 
e(D_{p,q,r})
\end{align*}
which implies
\begin{align*}
e(U_p) 
=
e(D_p) - 
\sum_{ q \, | \, q \neq p} e(D_{p,q}) + 
\sum_{q < r \,|\,  p \neq q,r} 
e(D_{p,q,r})
\end{align*}
Finally, for the case $S_1$, we have
\begin{align*}
e(f^{-1}(S_1))
&= 
2^s\left( 
e\left(Y\right) - 
e\left( \bigcup_p D_p \right)   
\right)
\end{align*}
where
\begin{align*}
 e\left( \bigcup_p D_p \right)  
 =
\sum_p e\left( D_p \right) - 
\sum_{p < q} e\left( D_p \cap D_q \right) +
\sum_{p < q < r} e\left( D_p \cap D_q  \cap D_r \right)
\end{align*}  
so our result follows.
\end{proof}

\begin{lemma}\label{lemma:e_X}
Let the notation be as in Lemma \ref{lemma:strata}, then 
 \begin{align*}
 e(X) &=
2^se(Y)
-2^{s-1}
\sum_{p}e(D_p) 
+
2^{s-2}\sum_{p<q}e(D_{p,q})
-2^{s-3}
\sum_{\substack{p < q <r \\ 
p+q+r \neq 0
} }
e(D_{p,q,r})
 \end{align*}
\end{lemma}
\begin{proof}
Since the loci $S_i$ are disjoint, it holds that
\( e(X)=
 \sum_{i=1}^{4} e(f^{-1}(S_i))\).
 Our first observation is that 
\[
\sum_{p<q}
\left( 
e(D_{p,q}) - \sum_{r \, | \, r \neq p,q}
e(D_{p,q,r})
\right)
=
\sum_{p <q}
\left(
e(D_{p,q}) 
- 3 \sum_{\substack{r \,| p<q<r}}
e(D_{p,q,r})
\right)
\]
because the term $D_{i,j,k}$ with $i < j <k$ appears three times in the sum: $p=i, q=j, r=k$, and $p=i, q = k, r =j$,  and $p=j, q=k, r = i$. Therefore 
\[
e\big(f^{-1}(S_3) \big)
=
2^{s-2}\sum_{p < q} 
e(D_{p,q}) 
- 3(2^{s-2})\sum_{\substack{p <q < r }}
e(D_{p,q,r})
\]
Similarly, we rewrite \(e\big(f^{-1}(S_2) \big)\) as follows: 
\begin{align*}
e\big(f^{-1}(S_2) \big)
&=
2^{s-1}\sum_{p} \left( e(D_p) - 
\sum_{ q \, | \, q \neq p} e(D_{p,q}) + 
\sum_{\substack{q<r\\ q,r\neq p}}
e(D_{p,q,r})
\right)
\\
&= 
2^{s-1}\left(\sum_{p} e(D_p) - 
2 \sum_{ p < q } e(D_{p,q}) + 
\sum_{p}
\sum_{\substack{q<r\\ q,r\neq p}}
e(D_{p,q,r})\right)
\\
&=
2^{s-1}\left(\sum_{p} e(D_p) - 
2 \sum_{ p < q } e(D_{p,q}) + 
3
\sum_{\substack{p<q<r}}
e(D_{p,q,r})\right)
\end{align*}
which implies
\begin{align*}
e\bigl(f^{-1}(S_2)\bigr)
&=
2^{s-1}   \sum_{p}
e(D_p) - 
2^s\sum_{ p< q } e(D_{p,q}) + 
3(2^{s-1})\sum_{\substack{p<q<r}} e(D_{p,q,r})
\end{align*}
By using the previous formulas within this proof and Lemma \ref{lemma:e(strata)} and 
\( e(X)=
 \sum_{i=1}^{4} e(f^{-1}(S_i))\), we obtain
\begin{align*}
 e(X) &=  2^{s-2} 
\sum_{\substack{p<q<r \\ p + q +r  = 0}}
e(D_{p,q,r})
+
2^{s-3} 
\sum_{\substack{p < q <r \\ p+q+r \neq 0} }
e(D_{p,q,r})
\\
& \quad +
2^{s-2}\sum_{p < q} 
e(D_{p,q}) 
- 3(2^{s-2})\sum_{\substack{p <q < r }}
e(D_{p,q,r})
\\
& \quad + 
2^{s-1}   \sum_{p}e(D_p) 
- 2^s\sum_{ p< q } e(D_{p,q}) + 
3(2^{s-1})\sum_{\substack{p<q<r}} e(D_{p,q,r})
\\
& \quad +
2^{s} e(Y) - 
2^{s} 
\sum_p e\left( D_p \right) + 
2^{s} \sum_{p < q} e\left( D_{p,q} \right) -
2^{s} \sum_{p < q < r} e\left( D_{p, q, r} \right)
\end{align*}
Our task is now to simplify this expression.
\begin{align*}
e(X)&=
2^se(Y)+
\left( 
2^{s-1}-2^s
\right)
\sum_{p}e(D_p) 
+
\left( 2^s-2^s +2^{s-2}\right)\sum_{p<q}e(D_{p,q})
\\
& \quad + 
\left(
3(2^{s-1})-2^s-3(2^{s-2})
\right)
\sum_{p<q<r}
e(D_{p,q,r})
\\
& \quad +
2^{s-2} 
\sum_{\substack{p<q<r \\ p+q+r = 0}}
e(D_{p,q,r})
+
2^{s-3} 
\sum_{\substack{p < q <r \\ 
p+q+r \neq 0
} }
e(D_{p,q,r})
\\
&=
2^se(Y)
-2^{s-1}
\sum_{p}e(D_p) 
+
2^{s-2}\sum_{p<q}e(D_{p,q})
-2^{s-2}
\sum_{p<q<r}
e(D_{p,q,r})
\\
& \quad 
 +
2^{s-2} 
\sum_{\substack{p<q<r \\ p+q+r = 0}}
e(D_{p,q,r})
+
2^{s-3} 
\sum_{\substack{p < q <r \\ 
p+q+r \neq 0} }
e(D_{p,q,r}).
\end{align*}
We observe that 
\[
2^{s-2} 
\sum_{\substack{p < q <r \\ 
p+q+r \neq 0} }
e(D_{p,q,r})
=
2^{s-3} 
\sum_{\substack{p < q <r \\ 
p+q+r \neq 0} }
e(D_{p,q,r})
+
2^{s-3} 
\sum_{\substack{p < q <r \\ 
p+q+r \neq 0} }
e(D_{p,q,r})
\]
Therefore,
\begin{align*}
  2^{s-2} 
\sum_{\substack{p<q<r \\ p+q+r = 0}}
e(D_{p,q,r})
+
2^{s-3} 
\sum_{\substack{p < q <r \\ 
p+q+r \neq 0} }
e(D_{p,q,r})
=
2^{s-2} 
\sum_{\substack{p<q<r}}
e(D_{p,q,r})
-
2^{s-3} 
\sum_{\substack{p < q <r \\ 
p+q+r \neq 0} }
e(D_{p,q,r})
\end{align*}
and we obtain
\begin{align*}
e(X) &=
2^se(Y)
-2^{s-1}
\sum_{p}e(D_p) 
+
2^{s-2}\sum_{p<q}e(D_{p,q})
-
2^{s-3}
\sum_{\substack{p < q <r \\ p+q+r \neq 0} }
e(D_{p,q,r})
\end{align*}
So our result follows.
\end{proof}
\begin{lemma}\label{lemma:strata_explicit}
Let $f: X \rightarrow \mathbb{P}(a_0,a_1,a_2,a_3)$ be a generic and flat $\mathbb{Z}_2^s$-cover, see Definition \ref{def:divisor_standard set-up},  and denote
$A:=a_0a_1a_2a_3$ and $A_\Sigma:=a_0+a_1+a_2+a_3$. Then
\begin{align*}
  e(D_{p,q,r}) &= \displaystyle\frac{d_pd_qd_r}{A} 
  &&
  \\
  e(D_{p,q})  &=  \frac{d_p d_q}{A}\,(A_\Sigma-d_p-d_q) 
  &&
  \\
  e(D_{p})  &= 
\frac{d_p}{A} \left(
d_p^2
- d_p\sum_{i=0}^{i=3} a_i
+ \sum_{i<j} a_ia_j 
\right)
-
\sum_{y \in Sing(D_p)} \left( 1 - \frac{1}{|G_y|} \right)
&&
\end{align*}
where $|G_y|$ is the order of the quotient singularity supported at $y$.
\end{lemma}
\begin{proof} Let $Y:=\mathbb{P}(a_0,a_1,a_2,a_3)$, 
by \cite[Sec I.2.12]{fletcher00}, there is a finite map $g_{cv}: \mathbb{P}^3 \to Y$ of degree $a_0a_1a_2a_3$. Since $g_{cv}^*\mathcal{O}_Y(d) = \mathcal{O}_{\mathbb{P}^3}(d)$, we have $D_pD_qD_r = \displaystyle\frac{1}{a_0a_1a_2a_3}d_pd_qd_r$. By hypothesis, $D_p, D_q, D_r$ intersect transversely and are away from the singular locus, then we have $D_pD_qD_r$ is equal to the number of points the hypersurfaces $D_p$, $D_q$ and $D_r$ intersect. Hence 
    \begin{equation}\label{triple intersection}
    e(D_{p,q,r}) =  
       e(D_p \cap D_q \cap D_r) = \displaystyle\frac{d_pd_qd_r}{a_0a_1a_2a_3} 
    \end{equation}
    We consider the following case.  By hypothesis, $D_{p,q}$ is well-formed and quasi-smooth. By \cite[I.3.9]{fletcher00}, $D_{p,q}$ does not contain a codimension $3$ singular stratum of $\mathbb{P}(a_0,a_1,a_2,a_3)$. Therefore, any possible singularity of $D_{p,q}$ is away from the singular locus of the weighted projective space, and quasi-smoothness implies $D_{p,q}$ is smooth.
    Since $Y$ is a  well–formed weighted projective threefold, it holds that  $K_Y\equiv -A_\Sigma H$ and
$H^3=1/A$ where $H$ is the generator of $\Pic_{\mathbb{Q}}(Y)$. 
By construction $D_{p,q}$ is a smooth complete–intersection of the divisors $D_p = d_pH$ and $D_q =d_qH$. Therefore, its genus is given by 
\[
2g(D_{p,q})-2 = \deg\!\left((K_Y+d_pH+d_qH)\cdot d_pH\cdot d_qH\right)
=
\left(
-A_{\Sigma}d_pd_q +d_p^2d_q+d_q^2d_p
\right)H^3
\]
and the result follows from $e(D_{p,q})=2-2g(D_{p,q})$.

To calculate $e(D_p)$, we follow the strategy used in 
\cite[Section A.2]{brown2020hodge}. 
$D_p$ is a surface with isolated cyclic singularities of the form $(D_p, y) \cong \left( \mathbb{C}^2/ G_y , 0 \right)$. 
Associated to $D_p$, we have the so called 
orbifold Euler characteristic $e_{orb}(D_p)$, see  \cite{hirzebruch1990euler},  which is the coefficient of $h^2$ in the series expansion of 
\[
g(h) := \frac{d_p}{A}\left( \frac{\prod_{i=0}^3  (1+ a_ih)}{1+ d_ph}\right).
\]
A direct calculation shows that
\begin{align*}
\prod_{i=0}^3  (1+ a_ih) & = 1 + \sum_{i=0}^{3} a_i h + \sum_{i < j} a_i a_j h^2 + \mathcal{O}(h^3)
\\
\frac{1}{1+ d_ph} &= 1 - d_ph + d_p^2 h^2 + \mathcal{O}(h^3)
\end{align*}
Therefore, by calculating the coefficient of $h^2$ in $g(h)$, we have
\begin{align*}
e_{orb}(D_p) = 
\frac{d_p}{A} \left(
d_p^2
- d_p\sum_{i=0}^{3} a_i
+ \sum_{i<j} a_ia_j 
\right).
\end{align*}
Now we use \cite[Sec 2.11-2.14]{blache1996chern} that relates the orbifold Euler characteristic with the standard one:
\begin{align*}
 e_{orb}(D_p) = 
 e(D_p) 
+ \sum_{y\in \text{Sing}(D_p)} \frac{|G_y| -1}{|G_y|} 
\end{align*}
to conclude our results.
\end{proof}
\begin{proposition}
\label{prop:e_k3}
Let $f: X \rightarrow \mathbb{P}(a_0,a_1,a_2,a_3)$ be a generic and flat $\mathbb{Z}_2^s$-cover with branch data
$D_{g} =d_gH$ for $g \in G$, see Definition \ref{def:divisor_standard set-up}. 
Then, 
\begin{align*}
e(X) &=
-\frac{2^{s-3}}{a_0a_1a_2a_3}
\left( 
\sum_{p}4d_p^3 
+
2\sum_{p<q}d_p d_q(d_p+d_q)
+
\sum_{\substack{p < q <h  \\ p+q+h \neq 0 } }
d_pd_qd_h
\right)
+ \mathcal{O}\left( 
\left(
\sum d_i \right)^2
\right)
 \end{align*}  
\end{proposition}

\begin{proof}
To obtain the expression for $e(X)$, we replace the expressions from Lemma \ref{lemma:strata_explicit} into Lemma \ref{lemma:e_X}, and then we identify the monomials on $d_h$ raised to the highest powers.  We provide the details next. 
Let's recall that we denote
$A:=a_0a_1a_2a_3$ and $A_{\Sigma}:=\sum a_i$. The expressions from Lemma \ref{lemma:strata_explicit} and the fact that the index of the singularities only depends on $a_i$ and not on the degrees $d_h$ imply the following:
\begin{align*}
  e(D_{p,q,h}) &= \displaystyle\frac{d_pd_qd_h}{A} 
  &&
  \\
  e(D_{p,q})  &=  \frac{d_p d_q}{A}\,(A_\Sigma-d_p-d_q) 
  =
 - \frac{d_p d_q}{A}\,(d_p +d_q)  
 + \mathcal{O}\left( 
\left(
\sum d_i \right)^2
\right)
  \\
  e(D_{p})  &=  
\frac{d_p}{A} \left(
d_p^2
- d_p\sum_{i=0}^{3} a_i
+ \sum_{i<j} a_ia_j 
\right)
-
\sum_{y \in Sing(D_p)} \left( 1 - \frac{1}{|G_y|} \right)
=
\frac{d_p^3}{A} 
+ \mathcal{O}\left( 
\left(
\sum d_i \right)^2
\right)
 \end{align*} 
We now replace in the expression of $e(X)$ to obtain
\begin{align*}
 e(X) = & \; 2^se(Y)
-2^{s-1}
\sum_{p}e(D_p) 
+
2^{s-2}\sum_{p<q}e(D_{p,q})
- 
2^{s-3} \sum_{\substack{p < q < h \\ 
p+q+h \neq 0
} }
e(D_{p,q,h})
\\
&= \; 
-2^{s-1} \sum_{p} \frac{d_p^3}{A}  
- 2^{s-2}\sum_{p < q} \frac{d_p d_q}{A}\,(d_p+d_q) 
-
 2^{s-3}
\sum_{\substack{p < q < h  \\ 
p+q+h \neq 0
} }
 \displaystyle\frac{d_pd_qd_h}{A} 
+ \mathcal{O}\left( 
\left(
\sum d_i \right)^2
\right)
 \\
 & = 
- \frac{2^{s-3}}{A} \left(
4 \sum_{p} d_p^3  
+ 2\sum_{p < q}  d_p d_q\,(d_p+d_q) 
+
\sum_{\substack{p < q < h \\ 
p+q+h \neq 0
} }
 \displaystyle d_pd_qd_h 
 \right)
+ \mathcal{O}\left( 
\left(
\sum d_i \right)^2
\right)
\end{align*}
\end{proof}


\subsection{Finding the asymptotics of Chern ratios}\label{sec:ratios}
A key aspect of our work is that we focus on the ratio among degrees of the branch locus, as defined below, rather than on its particular values 
\begin{definition}\label{definition of ratios}
Let $\{ D_g \;| \; g \in \mathbb{Z}_2^s \}$ be building data, then we define its ratios as 
\begin{align*}
 r_g:= \frac{\deg(D_g)}{\sum_{h \in G\setminus 0} \deg(D_{h})},
 &&
 r_0 := 0
\end{align*}
 Notice that $0 \leq r_h \leq 1$ for all $h \in G$, and  $\sum_{h \in G} r_h =1$.
\end{definition}
\begin{definition}\label{def:limit}
Given $n=2^{s}-1$, let 
$\mathbf{x}_g = (x_g \, | \, g \in \mathbb{Z}^s_2 \setminus \{0\})$ be a vector indexed by the elements of the group, for example, the degrees of the divisors $D_g$.
Let $P:\mathbb{C}^n \to \mathbb{C}$ be a function, we denote 
\[
 P( \mathbf{x}_g)
 \sim Q( \mathbf{x}_g)
\] 
 if $\lim_{\mathbf{x}_g \rightarrow \infty} \displaystyle\frac{P( \mathbf{x}_g)}{Q( \mathbf{x}_g)} = 1$, where by ${\mathbf{x}_g \rightarrow \infty}$, we mean $x_g \rightarrow \infty$ for at least one $g \in \mathbb{Z}^s_2 \setminus \{0\}$
 \end{definition}

\begin{lemma}
\label{lemma:ratio_holomorphic_euler_limit_chi_K3}
Let $f:X \to \mathbb{P}(a_0, \ldots, a_3)$ be a $\mathbb{Z}^s_2$-cover with notation as in Lemma \ref{lemma:K3}. Then, 
 \begin{align*}
 \frac{\chi(\mathcal{O}_X)}{K_X^3}
  &\sim \frac{ -1}{3(2^{s+1})} 
 \sum_{\chi \in G^*} 
  \left( 
  \sum_{\chi(g) =-1}   
  r_g
  \right)^3 
 \end{align*}
 with $\sum_{g \in G \setminus 0}  r_g = 1$ and $0 \leq r_g \leq 1$.
\end{lemma}
\begin{proof}
 By Lemmas \ref{lemma:K3} and 
 Lemma \ref{lemma:holomorphic_euler_limit_chi_K3}
 our previous calculation we have
 \begin{align*}
     \frac{\chi(\mathcal{O}_X)}{K_X^3}
  =
-  \left( \frac{ 1}{3(2^{s+1})} \right)
 \sum_{\chi \in G^*} 
  \left( 
  \sum_{g \, | \, \chi(g) = -1 }   
  \frac{   d_g}
    {\sum_{g \in G} d_{g}}
  \right)^3 
 \end{align*}
 So the statement follows from the expression of $r_g$ given in Definition \ref{definition of ratios}.   
\end{proof}

Here, we prove the expression for the ratio between $-K_X^3$ and $24 \chi\left( \mathcal{O}_X \right)$ that is given in Theorem \ref{thm:main_inv}.
\begin{lemma}
Let $f: X \rightarrow \mathbb{P}(a_0,a_1,a_2,a_3)$ be a generic and flat $\mathbb{Z}_2^s$-cover with branch data
$D_g$ and corresponding ratios $r_g$, see Definition \ref{definition of ratios}.
Then, it holds that
\begin{align*}
    \frac{e(X)}{K_X^3}
&\sim   
-
\left(
\sum_{p}2r_p^3 
+
2\sum_p r_p^2 
+
\frac{1}{6}
\sum_{p, q, h \in Q(s)}
r_pr_qr_h
\right)
\end{align*}
where 
$Q(s) = 
\{
(p,q,h) \in (\mathbb{Z}_2)^{3s}
\; | \;
p, q, h \text{ pairwise distinct, and }\;
p+q+h \neq 0 
\}$. 
\end{lemma}
\begin{proof}
We use the expressions from Lemma \ref{lemma:K3} and Proposition \ref{prop:e_k3}. Then, we simplify. The details are given next. 
We first use the symmetry of the indices 
\begin{align*}
2\sum_{p<q}r_p r_q(r_p+r_q)
=
\sum_{p \neq q}r_p r_q(r_p+r_q)
&&
\sum_{\substack{p < q < h \\ p+q+h \neq 0} }
r_pr_qr_h
=
\frac{1}{6}
\sum_{\substack{p, q, h \\ 
p \neq q, p \neq h, q\neq h 
\\ p+q+h \neq 0 } }
r_pr_qr_h
\end{align*}
to obtain:
\begin{align*}
   \sum_{p}4r_p^3 
+
2\sum_{p<q}r_p r_q(r_p+r_q)
+
\sum_{\substack{p < q <h \\ p+q+h \neq 0} }
r_pr_qr_h
=
   \sum_{p}4r_p^3 
+
\sum_{p \neq q}r_p r_q(r_p+r_q)
+
\frac{1}{6}
\sum_{\substack{p, q, h \\ 
p \neq q, p \neq h, q\neq h 
\\
p+q+h \neq 0} }
r_pr_qr_h
\end{align*}
Next, we use that
\begin{align*}
 \sum_{\substack{p,q \\ p \neq q}} r_p^2r_q  
 =
 \sum_p r_p^2\sum_{q \,|q \neq p}r_q
 =
  \sum_p r_p^2(1-r_p)
\end{align*}
implies
\begin{align*}
 \sum_{p \neq q}r_p r_q(r_p+r_q)
 =
\sum_{p \neq q}r_p^2 r_q
+
\sum_{q \neq p}r_p r_q^2
 =
2\sum_p r_p^2(1-r_p) 
=
2\sum_p r_p^2- 2 \sum_p r_p^3
\end{align*}
To obtain the last simplification in the statement:
\begin{align*}
   \sum_{p}4r_p^3 
+
\sum_{p \neq q}r_p r_q(r_p+r_q)
+
\frac{1}{6}
\sum_{\substack{p, q, h \\ 
p \neq q, p \neq h, q\neq h 
\\ p+q+h \neq 0} }
r_pr_qr_h
&=
   \sum_{p}4r_p^3 
+
2\sum_p r_p^2(1-r_p) 
+
\frac{1}{6}
\sum_{\substack{p, q, g \\ 
p \neq q, p \neq h, q\neq h 
\\p+q+h \neq 0} }
r_pr_qr_h
\\
&=
   \sum_{p}2r_p^3 
+
2\sum_p r_p^2 
+
\frac{1}{6}
\sum_{\substack{p, q, g 
\\ 
p \neq q, p \neq h, q\neq h 
\\
p+q+h \neq 0} }
r_pr_qr_h
\end{align*}
\end{proof}

\subsection{Upper and lower bounds for the Chern ratios}
We need some preliminary results. For each nontrivial character $\chi \in G^* \setminus\{1\}$, we define
\begin{align}\label{eq:A_chi_S_chi}
A_\chi := \sum_{g | \substack{ \chi(g) = -1}} r_g
&&
S_\chi := \sum_{g \in G} \chi(g)\, r_g.
\end{align}
and denote for convenience that $r_0 := 0$ and 
$A_1 = 0$.
\begin{lemma}\label{lemma:A_chi}
Let $A_\chi$ and $S_\chi$ be as in Equation \eqref{eq:A_chi_S_chi}, then  
\begin{align*}
2A_\chi = 1 - S_\chi,
&&
\sum_{\chi \neq 1} A_\chi = 2^{s-1},
&&
A_\chi \leq 1.
\end{align*}
Moreover, if $s \geq 2$, and $r_g > 0$ for all $g \neq 0$, then  $A_{\chi} < 1$.
\end{lemma}
\begin{proof}
We have that \(\chi(g) \in \{1,-1\}\) because
$G = \mathbb{Z}_2^s$. Therefore if we denote
\( 
B_\chi := \sum_{g | \chi(g) = 1}
 r_g
\)
we obtain:
\begin{align*}
A_\chi + B_\chi = \sum_{g \in G} r_g = 1,
&&
S_\chi
= \sum_{g \, | \, \chi(g) = 1} r_g + 
\sum_{g \, | \,\chi(g) = -1} (-1)r_g
=  B_\chi- A_\chi
\end{align*}
We obtain $2A_\chi = 1 - S_\chi$ 
and $2B_\chi = 1 + S_\chi$ 
by solving these two last equations.

For the second expression, we observe that
\[
\sum_{\chi \neq 1} A_\chi
= \sum_{\chi \neq 1} 
\left( 
\sum_{\substack{g \, | \, \chi(g) = -1}} r_g
\right)
= \sum_{g \in G \setminus 0} 
\left(
r_g  \#\{\chi \in G^* \setminus\{1\} \, |\, \chi(g) = -1\}
\right).
\]
Fix $g \neq 0$.
Since $G$ is a vector space over $\mathbb{F}_2$, its character group $G^*$ has $2^s$ characters.
For a fixed nonzero $g$, half of the characters satisfy $\chi(g) = -1$ and half satisfy $\chi(g)=1$.
The trivial character has $\chi(g)=1$, so among the nontrivial characters there are still exactly $2^{s-1}$ with $\chi(g)=-1$
and $2^{s-1}-1$ with $\chi(g)=1$.
Hence
\[
\#\{\chi \in G^* \setminus\{1\} : \chi(g) = -1\} = 2^{s-1},
\]
which implies (recall that $r_0 = 0$):
\[
\sum_{\chi \neq 1} A_\chi
= \sum_{g \in G \setminus 0 } r_g \cdot 2^{s-1}
= 2^{s-1} \sum_{g \in G } r_g
= 2^{s-1}.
\]
To prove the inequality, we observe that by construction, $A_\chi$ is a sum of some of the $r_g$, then we have
\(
0 \leq  A_\chi \le 1.
\)
We now show $A_\chi < 1$ for all $\chi \neq 1$. Suppose for contradiction that $A_\chi = 1$ for some nontrivial $\chi$.
Then
\begin{align*}
\sum_{\substack{g \, | \, \chi(g) = -1}} r_g = 1
=
\sum_{g \in G} r_g 
\end{align*}
By hypothesis, $r_g>0$ for all $g\neq 0$, this would force $\chi(g) = -1$ for every $g \neq 0$.
By hypothesis $s \geq 2$, so we can find two non-distinct elements $g_1$ and $g_2$ with $g_1 +g_2 \neq 0$ and $\chi(g_1)=\chi(g_2)=-1$. Such elements will imply
\[
\chi(g_1+g_2) = \chi(g_1)\chi(g_2) = (-1)(-1) = 1.
\]
This contradicts the requirement that $\chi(g)=-1$ for all $g\neq 0$.
Thus, no such $\chi$ exists, and we conclude our inequality.
\end{proof}
\begin{lemma}
\label{lemma:bound_sum_cubs}
Let $s \geq 2$, and define
 \[
 Q(r):=
  \sum_{\chi \in G^*} 
   \left( 
  \sum_{g \, | \, \chi(g) =-1}   
  r_g
  \right)^3.
 \]
Then
\begin{align*}
\frac{2^{3s-3}}{(2^s-1)^2}    \;  \leq Q(r)  \;  \leq 2^{s-1}.
\end{align*}
The maximum is realized at the vertices, that is $r_h=1$ for a fixed $h$ and $r_g=0$ for all $g\neq h$. The minimum is realized at the barycenter 
\[
r_g = \frac{1}{2^s-1} \quad \text{for all } g \neq 0
\]
\end{lemma}
\begin{proof}
By definition
\[
Q(r) = \sum_{\chi \in G^* } A_\chi^3.
\]
If $r_g >0$, then Lemma \ref{lemma:A_chi}
implies $A_\chi< 1$. Since, for $x \in (0,1)$ we have $x^3 < x$ and we obtain
\[
Q(r)
= \sum_{\chi \neq 1} A_\chi^3
< \sum_{\chi \neq 1} A_\chi
= 2^{s-1}.
\]
When $r_g = 0$ for some $g \in G$, some of the strict inequalities become non-strict, so we have
$Q(r) \leq 2^{s-1}$. To reach the equality 
 we fix some $g_0 \in G\setminus\{0\}$, we take $r_{g_0}=1$ and $r_g=0$ for $g\neq g_0$. This would give
\[
A_\chi =
\begin{cases}
1 & \text{if } \chi(g_0) = -1,\\
0 & \text{if } \chi(g_0) = 1,
\end{cases}
\]
and since exactly $2^{s-1}$ nontrivial characters satisfy $\chi(g_0)=-1$, we obtain that 
\(
Q(r) = 2^{s-1}.
\)
Next, we work on the lower bound.  Each $A_\chi(r)$ is an affine linear function of $(r_g)$, and
the map $x \mapsto x^3$ is strictly convex on $[0,\infty)$.
Thus $A_\chi(r)^3$ is convex in $r$, and so
\[
Q(r)
= \sum_{\chi \neq 1} A_\chi(r)^3
\]
is a convex function on the simplex
\[
\Delta = \left\{ (r_g)_{g\neq 0} \, | \;  r_g \ge 0,\ \sum r_g = 1 \right\},
\]
and in particular on its interior defined by $r_g > 0$. We notice that $Q$ is invariant under automorphisms of $G$ because 
\[
\chi(g) = -1 \iff 
\left( \chi \circ \varphi^{-1} \right) (\varphi(g)) =-1.
\]
It follows that $Q(\varphi \cdot r) = Q(r)$, and if
$r$ is any minimizer of $Q$ on $\Delta$, then $\varphi\cdot r$ is also a minimizer.

We consider the average
\[
\bar r_g \coloneqq \frac{1}{|\mathrm{Aut}(G)|}
\sum_{\varphi \in \mathrm{Aut}(G)} 
r_{\varphi^{-1}(g)} 
\]
whereby construction, $\bar r_g$ is independent of $g\neq 0$ because the automorphism group acts transitively on the non-zero elements of $(\mathbb{Z}_2)^s$. 
 Jensen's inequality for convex functions applies to $Q$. It implies that if $r_{unf}$ is the vector with all entries
 equal to $\frac{1}{2^s-1}$, then:
\[
Q(r_{unf})
\le \frac{1}{|\mathrm{Aut}(G)|}
\sum_{\varphi} Q(\varphi \cdot r )
= Q(r),
\]
By Lemma \ref{lemma:A_chi}, $A_\chi <1$ on the interior of $\Delta$ and $x \to x^3$ is strictly convex within the interior of $[0,1]$. This implies that $Q$ is strictly convex on the interior of $\Delta$, so the unique minimizer in the
region $r_g>0$ is
\[
r_g = \frac{1}{2^s-1}\quad \forall g\neq 0.
\]
and we have
\[
Q(r) \ge Q(r_{\mathrm{unf}})
= \frac{2^{3s-3}}{(2^s - 1)^2},
\]
with $r_{\mathrm{unf}}$ the only minimizer in the interior of $\Delta$.
To show the uniqueness of the minimizer, we show that it cannot be on the boundary of $\Delta$. For that, let $\hat{r}$ be a minimizer on the boundary. Since all the entries are non-negative. Then, 
$\frac{\hat{r}+r_{unf}}{2}$ is in the interior of $\Delta$ and 
\[
Q\left( 
\frac{\hat{r}+r_{unf}}{2}
\right)
<
Q(r_{unf})
\]
which is a contradiction.
\end{proof}

We continue with the ratio of the topological Euler characteristic to holomorphic Euler characteristic. First, we need a couple of preliminary lemmas.
\begin{lemma}\label{lemma:sum_s_chi}
Let
\begin{align*}
S_{\chi} := \sum_{h \in G} \chi(h)\, r_h,
&&
b := \sum_{h \in G} r_h^2
&& 
T := \sum_{\substack{p,q,h \\ 
p \neq q, p \neq h, q \neq h
\\
p+q+h = 0}}
r_p r_q r_h,
\end{align*}
where we recall that $r_0=0$.
Then
\begin{align*}
\sum_{\chi \in G^{\ast}
\setminus \{1 \}
} S_{\chi} &= -1,
&&
\sum_{\chi \in G^{\ast}
\setminus \{1 \}
} S_{\chi}^2 = 2^s\, b - 1,
&&
\sum_{\chi \in G^{\ast}
\setminus \{1 \}
} S_{\chi}^3
= 2^s S_{idp} - 1.
\end{align*}
where $1$ denotes the trivial character on $G$.

\end{lemma}
\begin{proof}
We first sum $S_{\chi}$ over all characters of $G$:
\[
\sum_{\chi \in G^{\ast}} S_{\chi}
= \sum_{\chi \in G^{\ast}} \sum_{h \in G } \chi(h)\, r_h
= \sum_{h \in G } r_h \sum_{\chi \in G^{\ast}} \chi(h),
\]
where we have interchanged the order of summation. Recall that for a fixed $g$, it holds that:
\[
\sum_{\chi \in G^{\ast} 
} \chi(g)
=
\begin{cases}
|G^{\ast}| = |G| = 2^s, & \text{if } g = 0,\\[4pt]
0, & \text{if } g \neq 0.
\end{cases}
\]
Since our sum runs only over $g \in G \setminus \{0\}$, for each such $g \neq 0$ we have
\begin{align*}
\sum_{\chi \in G^{\ast}
} \chi(g) = 0,
&&
\sum_{\chi \in G^{\ast}
} S_{\chi}
= \sum_{h \in G \setminus \{0\}} r_h \cdot 0
= 0
\end{align*}
Now we separate the contribution of the trivial character $1 \in G^{\ast}$. For this character,
\[
S_{1} = \sum_{h \in G } 1 \cdot r_h
= \sum_{h \in G } r_h
= 1.
\]
Therefore,
\[
\sum_{\chi \in G^{\ast}} S_{\chi} 
=
S_{1}
+
\sum_{\chi \in G^{\ast} \setminus \{1\}} S_{\chi} 
= 0,
\]
and our first result follows.

Next, we continue with the second identity.
By definition,
\[
S_{\chi}^2
= \left( \sum_{l \in G} \chi(l)\, r_l \right)
  \left( \sum_{h \in G } \chi(h)\, r_h \right)
= \sum_{l \in G } \sum_{h \in G }
    \chi(l)\chi(h)\, r_glr_h.
\]
Summing over all $\chi \in G^{\ast}$, interchanging the order of summation and using $\chi(l)\chi(h) = \chi(l+h)$, we obtain
\[
\sum_{\chi \in G^{\ast}} S_{\chi}^2
= \sum_{l \in G } \sum_{h \in G }
    r_l r_h \sum_{\chi \in G^{\ast}} \chi(l)\chi(h)
=
\sum_{l \in G } \sum_{h \in G}
    r_l r_h \sum_{\chi \in G^{\ast}} \chi(l + h)
\]
Since
\[
\sum_{\chi \in G^{\ast}} \chi(l)\chi(h)
=
\begin{cases}
2^s, & \text{if } l = h,\\
0, & \text{if } l \neq h.
\end{cases}
\]
It follows that only the diagonal terms $l=h$ contribute, and hence
\[
\sum_{\chi \in G^{\ast}} S_{\chi}^2
= \sum_{h \in G} r_h^2 \cdot 2^s
= 2^s \sum_{h \in G } r_h^2
= 2^s b.
\]
Therefore,
\[
\sum_{\chi \in G^{\ast} \setminus \{1\}} S_{\chi}^2
= \sum_{\chi \in G^{\ast}} S_{\chi}^2 - S_1^2
= 2^s b - 1.
\]
The last identity is pretty similar. We start with
\[
S_{\chi}^3
= \left( \sum_{p \in G } \chi(p)\, r_p \right)
  \left( \sum_{q \in G } \chi(q)\, r_q \right)
  \left( \sum_{h \in G } \chi(h)\, r_h \right)
= \sum_{p \in G }
  \sum_{q \in G }
  \sum_{h \in G }
  \chi(p)\chi(q)\chi(h)\, r_p r_q r_h.
\]
As in the previous cases,
\[
\sum_{\chi \in G^{\ast}} \chi(p)\chi(q)\chi(h)
=
\begin{cases}
2^s, & \text{if } p+q+h = 0,\\[4pt]
0, & \text{otherwise}.
\end{cases}
\]
Substituting this back, we find
\begin{align*}
\sum_{\chi \in G^{\ast}} S_{\chi}^3
= \sum_{\substack{p,q,h \in G  \\ 
p \neq q, p \neq h, q \neq h
\\
p+q+h = 0}}
   r_p r_q r_h \cdot 2^s
= 2^s T,
&&
\sum_{\chi \in G^{\ast} \setminus \{1\}} S_{\chi}^3
= \sum_{\chi \in G^{\ast}} S_{\chi}^3 - S_1^3
= 2^s T - 1.
\end{align*}
\end{proof}
\begin{lemma}\label{lemma:cubic_rg}
Notation as before, it holds that
\begin{align}
\sum_{\chi \in G^{\ast} \setminus \{1\}} 
\left(
\sum_{h | \chi(h)=-1}r_h
\right)^3
 &= 2^{s-3}\bigl( 3b - T + 1 \bigr).
\end{align}
\end{lemma}
\begin{proof}
By using the notation and results of Lemma \ref{lemma:A_chi}, for every 
non-trivial $\chi \in G^{\ast}$ we have
\[
\left(\sum_{h | \chi(h)=-1}r_h \right)^3
=
A_{\chi}^3 = \left(\frac{1-S_{\chi}}{2}\right)^3
= \frac{1}{8}\bigl(1 - 3S_{\chi} + 3S_{\chi}^2 - S_{\chi}^3\bigr).
\]
We simplify and use Lemma \ref{lemma:sum_s_chi} to obtain:
\begin{align*}
\sum_{\chi \in G^{\ast} \setminus \{1\}} 
\left(
\sum_{h | \chi(h)=-1}r_h
\right)^3
&= \sum_{\chi \in G^{\ast} \setminus \{1\}} A_{\chi}^3
\\
&= \frac{1}{8}\sum_{\chi \in G^{\ast} \setminus \{1\}}
   \bigl(1 - 3S_{\chi} + 3S_{\chi}^2 - S_{\chi}^3\bigr) 
\\
&=
\frac{1}{8}
\left(
(2^s-1)
-3(-1) + 3(2^sb-1) -(2^s T -1) 
\right)
\\
& = 
\frac{1}{8}
\left( 
2^{s} + 3(2^sb) - 2^sT
\right),
\end{align*}
and our claim follows. 
\end{proof}

\begin{lemma}\label{lemma:e_chi_rewrite}
Let notation be as before,  and define
\begin{align*}
Q(s) &:= 
\{
(p,q,h) \in (\mathbb{Z}_2)^{3s}
\; | \;
p, q, h \text{ pairwise distinct, and }\;
p+q+h \neq 0 
\}
\\
E(s)
&:= \{
(p,q,h) \in (\mathbb{Z}_2)^{3s}
\; | \;
p, q, h \text{ pairwise distinct, and }\;
p+q+h = 0 
\}
\end{align*}
together with 
\begin{align*}
 a:= \sum_{h\in G} r_h^3   
 &&
 b:= \sum_{h\in G} r_h^2 
 &&
 T = 
 \sum_{p,q,h \in E(s)}
r_p r_q r_h,
&&
S_{idp} = 
 \sum_{p,q,h \in Q(s)}
r_p r_q r_h,
\end{align*}
Then, it holds that
 \begin{align}\label{eq:ration_e_O}
  \frac{e(X)}{24\chi(\mathcal{O}_X)}
\sim
\frac{14a+9b+1-T}{3(3b-T+1)}
\end{align}   
\end{lemma}
\begin{proof}
By Lemma \ref{lemma:holomorphic_euler_limit_chi_K3} and Proposition \ref{prop:e_k3},  we have that 
\begin{align}
  \frac{e(X)}{24\chi(\mathcal{O}_X)}
&=\left( \frac{-e(X)}{K^3}\right)
\left( - \frac{K^3}{24\chi(\mathcal{O}_X)}
\right)
\notag
\\
& \sim \left( \frac{-e(X)}{K^3}\right)
\left( \frac{2^{s-2}}{\sum_{\chi \in G^*} 
  \left( 
  \sum_{h \, | \, \chi(h) =-1}   
  r_h
  \right)^3 }
\right)
\notag
\\
& \sim 
\frac{2^{s-2}}{\sum_{\chi \in G^*} 
  \left( 
  \sum_{h \, | \, \chi(h) =-1}   
  r_h
  \right)^3 } 
\left(
\sum_{p}2r_p^3 
+
2\sum_p r_p^2 
+
\frac{1}{6}
S_{idp}
\right)
\end{align}

Next, we show that 
\begin{align}\label{eq:numerator_e_chi}
\sum_{p} 2 r_p^3
\;+\; 2\sum_{p} r_p^2
\;+\; \frac{1}{6}
S_{idp}
\;&=\;
\frac{1}{6}\bigl(14a + 9b + 1 - T\bigr)
\end{align}
First, we rewrite
\begin{align}\label{eq:inter_rp}
\sum_p 2r_p^3
  \;+\; 2\sum_p r_p^2
  \;+\; \frac{1}{6} S_{idp}
= 2a + 2b \;+\; \frac{1}{6} S_{idp}, 
\end{align}
Throughout, all indices lie in 
$G\setminus\{0\}$, and all triple sums are over ordered triples. First observe that
\begin{align*}
\sum_{p,q,g} r_p r_q r_h
&= \Bigl(\sum_{p} r_p\Bigr)
  \Bigl(\sum_{q} r_q\Bigr)
  \Bigl(\sum_{h} r_h\Bigr)
= 1
\\
&=
\sum_{p} r_p^3 
+ 3\sum_{p, q | p \neq q} r_p^2r_q 
+  \sum_{\substack{p,q,h\\
                    p \neq q, p \neq h,q \neq h \\
                  p+q+h = 0
                   }}r_p r_q r_h
+
   \sum_{\substack{p,q,h\\
                    p \neq q, p \neq h,q \neq h \\
                  p+q+h \neq 0
                   }}r_p r_q r_h  
\\
&= a + 3(b-a) + T + S_{idp}
\end{align*}
where we used
\[ 
\sum_{p,q | p\ne q} 3 r_p^2 r_q
= 3\sum_p r_p^2\!\! \, \sum_{q | q\ne p} r_q
= 3\sum_p r_p^2(1 - r_p)
= 3\bigl(b - a\bigr).
\]
Replacing
\(
S_{\mathrm{idp}}
=
1 - 3b + 2a - T
\)
in Equation \eqref{eq:inter_rp} yields equation \eqref{eq:numerator_e_chi}. By Lemma \ref{lemma:cubic_rg}, we have
\begin{align*}
\frac{e(X)}{24\chi(\mathcal{O}_X)}
\sim
2^{s-2}
\frac{14a+9b+1-T}{6(2^{s-3})(3b-T+1)}
=
\frac{14a+9b+1-T}{3(3b-T+1)}
\end{align*}
\end{proof}
\begin{lemma}\label{lemma:inq_Fr}
Let notation be as in Lemma \ref{lemma:e_chi_rewrite}, 
and denote
\[
F(r):= \frac{14a+9b+1-T}{3(3b-T+1)}
\]
Then it holds that
 \begin{align*} 
F(r)
\leq 
2
 \end{align*}  
 \end{lemma}
\begin{proof}
By using Cauchy-Schwarz, we obtain 
\begin{align}\label{eq:a_b_ineq}
  \frac{1}{(2^s-1)^2} \leq b^2 \leq a \leq b \leq 1.
\end{align}
Every inequality is sharp at either the barycenter
$r_g = \frac{1}{2^s-1}$ or at the vertex 
$r_g =1 $ for a single $g$.
Since $F(r)$ is increasing in $a$, we can write
\[
\frac{14b^2+9b+1-T}{3(3b-T+1)} \leq F(r) 
\]
To evaluate $F(r)$ at a vertex $v$ of $\Delta$, say $r_{g_0}=1$, $r_g=0$ for $g\ne g_0$, we observe that at point $a = b =1$ and $T=0$, therefore
\[
F(r_{g_0}) = 2 
\]
Next, we show $F(\mathbf{r}) \leq 2$ for all our $\mathbf{r}$. This claim  is equivalent to
\[
14a + 9b + 1 - T \;\leq\; 6(3b - T + 1),
\]
which can be written as
\begin{equation}\label{eq:F_leq_2}
9b - 14a + 5 - 5T \;\geq\; 0.
\end{equation}
We replace the constant $5$ using $\left(\sum_{g} r_g\right)^3 = 1$. By the decomposition established in Equation~\eqref{eq:inter_rp},
\[
1 = a + 3(b - a) + T + S_{\mathrm{idp}},
\]
so $5 = -10a + 15b + 5T + 5S_{\mathrm{idp}}$. Substituting into \eqref{eq:F_leq_2} yields
\[
9b - 14a + 5 - 5T = 24(b - a) + 5\,S_{\mathrm{idp}}.
\]
By Equation \eqref{eq:a_b_ineq} and its definition, both terms are non-negative. Therefore, Equation \eqref{eq:F_leq_2} holds and $F(\mathbf{r}) \leq 2$. 
\end{proof}

\subsection{Proof of Theorem \ref{thm:main_inv} and corollaries}
\label{sec:proof_thm_inv}
\begin{proof}[Proof of Theorem \ref{thm:main_inv} ]
We start by the lower bound. For that purpose, we need the following identities:
Given, 
\begin{align*}
 a = \sum_{h \in G} r_h^3
 &&
 b = \sum_{h \in G} r_h^2
 &&
 \Phi := 2^{3-s}Q(r)
\end{align*}
Then, it holds that
\begin{align}\label{eq:exp_proof}
y &= \frac{-K_X^3}{24\,\chi(\mathcal{O}_X)}
=
\frac{2}{\Phi}, 
&&
x =  \frac{e(X)}{24\,\chi(\mathcal{O}_X)}
= \frac{14a + 6b + \Phi}{3\Phi},    
\\
&&
\operatorname{SCI}(x,y)+4 &= 
y(3x+1) = \frac{4(7a + 3b + \Phi)}{\Phi^2}.
\notag
\end{align}
Indeed, by Lemma~\ref{lemma:cubic_rg}, $Q(r) = 2^{s-3}(3b - T + 1)$, so $\Phi = 3b - T + 1$ and hence $T = 3b + 1 - \Phi$. Since $Q(r) = 2^{s-3}\Phi$, the ratio $y$ becomes
\[
y = \frac{2^{s-2}}{Q(r)} = \frac{2^{s-2}}{2^{s-3}\Phi} = \frac{2}{\Phi}.
\]
Substituting $T = 3b + 1 - \Phi$ into the expression for $x$ from Lemma~\ref{lemma:e_chi_rewrite} gives
\begin{align*}
x = \frac{14a + 9b + 1 - T}{3(3b - T + 1)} = \frac{14a + 6b + \Phi}{3\Phi}, 
\quad 
3x + 1 =  \frac{14a + 6b + 2\Phi}{\Phi},
\end{align*}
Therefore,
\[
y(3x+1) =  
\frac{2}{\Phi} \cdot \frac{14a + 6b + 2\Phi}{\Phi} = \frac{2(14a + 6b + 2\Phi)}{\Phi^2} = \frac{4(7a + 3b + \Phi )}{\Phi^2}.
\]
and our identities  are verified. 
Our next step is to obtain a lower bound for the function
\begin{align*}
 H(a,b,\Phi)  := 
 \frac{4(7a + 3b + \Phi )}{\Phi^2} 
\end{align*}
Indeed, we observe that
\begin{align}\label{eq:partial_H}
\frac{\partial H(a,b, \Phi)}{ \partial \Phi}
=
\frac{-4 ( 14 a + 6 b + \Phi )}{\Phi^3}
&&
a \geq 0, \quad b \geq 0, \quad 3b- \Phi +1 = T \geq  0
\end{align}
imply $H(a,b,\Phi)$ is decreasing with respect to $\Phi$. So $3b + 1 \geq \Phi$ and $a \geq b^2$, see Equation \eqref{eq:a_b_ineq}, yield
\begin{align*}
 H(a,b,\Phi)  
 \geq 
  \frac{4(7b^2 + 6b + 1 )}{(3b+1)^2},
  &&
  \frac{\partial }{\partial b}\left(
\frac{4(7b^2 + 6b + 1 )}{(3b+1)^2}  
  \right)
  = 
  \frac{-16b}{(1 + 3 b)^3} \leq 0, \; \text{ for } b \geq 0
\end{align*}
Since, this last function is decreasing with respect to $b$ and satisfy $b \leq 1$ by Equation \eqref{eq:a_b_ineq}. We obtain
\begin{align*}
  H(a,b,\Phi)  \geq \frac{7}{2}, 
  &&
\operatorname{SCI}(x,y) \geq  -\frac{1}{2}
\end{align*}
a direct calculation shows that the inequality is attained at $r_{h} =1$ for some $h$, and $r_g =0$ for all $g \neq h$. 

Next, we focus on the upper bound. 
First, we observe that $\Phi \geq (1+b)^2$. Indeed
by Lemma~\ref{lemma:A_chi},
$\sum_{\chi \neq 1} A_\chi = 2^{s-1}$ and 
 $A_\chi = (1 - S_\chi)/2$. Using
Lemma~\ref{lemma:sum_s_chi}:
\[
\sum_{\chi \neq 1} A_\chi^2
= \tfrac{1}{4}\bigl[(2^s-1) + 2 + (2^s b - 1)\bigr]
= 2^{s-2}(1+b).
\]
Applying the Cauchy--Schwarz inequality with
$u_\chi = A_\chi^{1/2}$ and $w_\chi = A_\chi^{3/2}$, we obtain
\begin{align*}
 \bigl(\sum A_\chi^2\bigr)^2
\leq \bigl(\sum A_\chi\bigr)\bigl(\sum A_\chi^3\bigr)   
\Longrightarrow
\left( 2^{s-2}(1+b) \right)^2 \leq 
2^{s-1}Q(r)
= 2^{s-1} \left( 2^{s-3}\Phi \right)
\end{align*}
from which $\Phi \geq (1+b)^2$ follows.
In Equation \eqref{eq:partial_H}, we learned that 
$\partial H/\partial\Phi < 0$ within our domain. By Equation \ref{eq:a_b_ineq}, 
$b^2 \leq a \leq b$, therefore
\[
y(3x+1) 
=
\frac{4(7a+3b+\Phi)}{\Phi^2}
\;\leq\; g(b) \;:=\;
  \frac{4(b^2 + 12b + 1)}{(1+b)^4}.
\]
Finally, we show that $g(b) \leq \frac{20}{3}$ on $[0,1]$. Since, the inequality $g(b) \leq 20/3$ is equivalent to
\[
R(b) =
b^4 + 4 b^3 + \frac{27 b^2}{5} - 
\frac{16 b}{5} + \frac{2}{5}
,\qquad b \in [0,1].
\]
A direct inspection, or the use of mathematical software, allows us to verify that $R'(b)$ has a unique real root $b_0$ within $[0,1]$. This root is the unique minimizer of $R$ on $[0,1]$ and it satisfies $R(b_0) > 0$

Finally, we show the other inequalities in the statement of the theorem.  By Lemma \ref{lemma:holomorphic_euler_limit_chi_K3}, it holds that 
\begin{align}\label{eq:ratio_K3_chi}
 \left( \frac{-K_X^3}{24\chi(\mathcal{O}_X)} \right) 
 &\sim
 \frac{2^{s-2}}{ 
 \sum_{\chi \in G^*} 
  \left( 
  \sum_{g \, | \, \chi(g) = - 1}   
  r_g
  \right)^3}
 \end{align}
 with $\sum_{g \in G \setminus e}  r_g = 1$ and $0 \leq r_g \leq 1$. 
By Lemma
\ref{lemma:bound_sum_cubs}, we have 
the upper and lower bounds 
\[
\frac{1}{2}
\leq
 \lim_{d_g \rightarrow \infty}
\left( \frac{-K_X^3}{24\chi(\mathcal{O}_X)} \right)
\leq 
2+ \frac{1}{2^{2s-1}} + \frac{1}{2^{s-2}}
\]
The lower bound $\tfrac{1}{2}$ is attained in the limit
$r_h \to 1$ for a single $h$, and the upper bound is attained
at the barycenter $r_g = 1/(2^s - 1)$ for all $g \neq 0$.  Similarly, by Lemmas \ref{lemma:e_chi_rewrite} and ~\ref{lemma:inq_Fr}, we have that
\[
  \frac{e(X)}{24\,\chi(\mathcal{O}_X)}
  \;\leq\;
  2\,.
\]
The lower bound is using by observing that
$\frac{7}{2} \leq y(3x+1)$ and 
$0 < y \leq y_{max}$ and $0 < y$ imply 
\begin{align*}
\frac{7}{2(3y_{max})} -\frac{1}{3} \leq x,
&&
\text{ which in our case is equal to} \Rightarrow
&&
\frac{3 (2^{2 s-2}) - 2^{s + 1} - 1}{3 (2^s + 1)^2}
\leq x
\end{align*}
\end{proof}
Next, we discuss Corollary \ref{cor:counter_example_hunt}.
We recall that in \cite[Sec 7.2.3-7.2.4]{huntthreefolds89}, 
Hunt calculated the Chern numbers of smooth complete intersections in $\mathbb{P}^N$. Then, he showed that  their accumulation points lie in the curve $y(3x+1) = 4$. Moreover, he defined the zone
\begin{align*}
  \text{Zone}(E):=
  \{(x,y) \, | \, 
  y(3x+1) > 4 \},
  &&
  (x,y) =  \left(
\frac{e(X)}{24\,\chi(\mathcal{O}_X)} , 
    \frac{-K_X^3}{24\,\chi(\mathcal{O}_X)} \right)
\end{align*}
and conjectured there are none smooth threefolds with invariants in such a zone.  To construct counter-examples to Hunt conjecture, we recall a construction of Pardini and Alexeev.
\begin{theorem}[Pardini-Alexeev]
A $(\mathbb{Z}_2)^s$-cover of $\mathbb{P}^3$ with $s\geq 3$ is called \emph{almost uniform} if there exist a fixed $1 \neq \chi_0 \in G^*$ such that the building data $D_g$ of the cover satisfies:
\begin{align*}
 \deg (D_g)  = \begin{cases}
     d_g > 0 & \chi_0(g) = - 1 \\
     d_g = 0 & \chi_0(g) = 1.
 \end{cases}   
\end{align*}
If $X$ is an almost uniform cover of $\mathbb{P}^3$ and it is of general type, then $X$ is smooth, minimal and  all its small deformations are Galois.
\end{theorem}

Using almost uniform covers of $\mathbb{P}^3$, we obtain examples of smooth threefolds of general type within the forbidden zone E.
\begin{corollary}
If $X(m) \rightarrow \mathbb{P}^3$ is a family of almost uniform $(\mathbb{Z}_2)^s$-covers of $\mathbb{P}^3$ with $s \geq 3$ and branch data $D_g(m)$ that satisfies the following condition: There exist an $1 \neq \chi_0 \in G^*$, and $0 \neq g_0 \in G$ with $\chi_0(g_0) = -1$, and a rational $r_a$ with $\frac{1}{2} \leq r_a \leq \frac{68}{100}$ such that
\begin{align*}
\lim_{m \to \infty}
 \frac{\deg(D_g(m))}{\sum_{h \in G} \deg(D_h(m))}
= \begin{cases}
r_a & \text{ if  $g= g_0$ }    
\\
\frac{1-r_a}{(2^{s-1}-1)} & \text{ if $\chi_0(g) =-1$ and $g \neq g_0$}   
\\
0 & \text{if $\chi_0(g) =1$}
\end{cases}
\end{align*}
Then, the limit 
\begin{align*}
\lim_{\deg(D_g) \to \infty}  \left(
\frac{e(X)}{24\,\chi(\mathcal{O}_X)} , 
    \frac{-K_X^3}{24\,\chi(\mathcal{O}_X)} \right)
\end{align*}
is contained in the forbidden Zone E.
\end{corollary}
\begin{proof}
We first observe that for an almost uniform cover, the term 
$$
T = \sum_{\substack{p,q,r \text{ distinct} \\ p+q+r=0}} r_p r_q r_r
$$
is zero. Indeed, by construction there is a fixed character $\chi_0$ such that $r_g = 0$ unless $g \in S := \{g : \chi_0(g) = -1\}$. 
Let $p,q,r$ be a triple in $T$, then we have  
$$
1 = \chi_0(0) = \chi_0(p+q+r) = \chi_0(p)\chi_0(q)\chi_0(r) = (-1)(-1)(-1) = -1
$$
which is a contradiction. So no such triple exists, and $T = 0$. The formulas for the invariants in Theorem \ref{thm:main_inv} now become
\[
y = \frac{1}{3b + 1}, \qquad x = \frac{14a + 9b + 1}{3(3b+1)}
\]
where $a = \sum r_g^3$, $b = \sum r_g^2$. Elementary algebra yields
\begin{align*}
y(3x+1) = C(a,b):=\frac{4(7a + 6b + 1)}{(3b+1)^2}, 
&&
C(a, b) \geq 4 \iff 7a \geq 9b^2
\end{align*}
We recall that $a=\sum_{h \in G} r_h^3$ 
and 
$b=\sum_{h \in G} r_h^2$. 
Then, for $m= 2^{s-1}$ and $t := r_a$, we have
\begin{align*}
 a &= t^3 + (2^{s-1}-1) \left( \frac{(1-t)^3}{(2^{s-1}-1)^3} \right) 
 \\
 b &= t^2 + (2^{s-1}-1) \left( \frac{(1-t)^2}{(2^{s-1}-1)^2}  \right)
 \\
 7a-9b^2 &:= F(m,t) = 
 \frac{(-9 m^2 t^4 + 7 m^2 t^3 + 22 m t^3 - 18 m t^2 - 15 t^2 + 15 t - 2)}{(m - 1)^2}
\end{align*}
such that 
\begin{align*}
\frac{\partial F(r,m)}{\partial m}
= \frac{2 (1 - t)^2 (9 t^2 m - 11 t + 2)}{(m - 1)^3}
\end{align*}
so $F(r,m) \geq 0$ for $s \geq 2$ and $t \geq 0$. This means $F(m,r)$ is increasing with respect to $m$ and it is enough to consider $s=3$. A direct calculation shows that 
$$
F(4,t) = \frac{1}{9} (-144 t^4 + 200 t^3 - 87 t^2 + 15 t - 2)
$$ is possitive for $\frac{1}{2} \leq t \leq \frac{68}{100}$.
\end{proof}

\section{Deformations of $\mathbb{Z}_2^s$ covers of weighted projective spaces}
\label{sec:deformation_theory}

In this section, our objective is to find some necessary conditions under which, the general deformation of a $\mathbb{Z}_2^s$ cover of a possibly singular scheme $Y$, is once again a $\mathbb{Z}_2^s$ cover of $Y$. This problem of course has been studied for abelian covers $f: X \to Y$ when $X$ and $Y$ are smooth in \cite{pardini1991abelian} and $f: X \to Y$ flat with $Y$ normal and $f$ locally simple (which implies $X$ is smooth) in \cite{Man95}. Combining the techniques above with those introduced by Wehler in \cite{wehler86deformations} and Schlessinger in \cite{schlessinger1971}, we generalize their results for non-flat abelian covers $f: X \to Y$ when both $X$ and $Y$ are normal and the codimension of the singular locus of $X$ is at least three. As a result, we construct examples of moduli components of stable varieties whose general element is a $\mathbb{Z}_2^s$ cover of a weighted projective space (see \Cref{new components}). We also give concrete examples of non-flat $\mathbb{Z}_2^s$ covers $f: X \to Y$, with $s \geq 4$, $K_X$ ample, branched along configurations of hyperplane arrangements such that any deformation of $X$ is once again a $\mathbb{Z}_2^s$ cover with the same configuration of the branch locus. 

\subsection{Deformation of finite non-flat covers}
We start with some preliminaries. Let $f: X \to Y$ be a finite morphism and let $f_*(\mathcal{O}_X) = \mathcal{O}_Y \oplus \mathcal{E}$.

\begin{definition}
    Let $\bf Def(f)$ denote the following functor of Artin rings. For an Artin ring $A$, $\bf Def(f)(A)$ is the set of isomorphism classes of Cartesian diagrams as follows
$$  \bf Def(f)(A) := \left\{ 
    \begin{tikzcd}
     X \arrow[r, hook] \arrow[d, "f"] & \mathcal{X} \arrow[d, "F"] \\
     Y \arrow[r, hook] \arrow[d] & \mathcal{Y} \arrow[d] \\
     \textrm{Spec}(\mathbf{k}) \arrow[r, hook] & \textrm{Spec}(A) 
    \end{tikzcd}
\right\}
$$

where $\mathcal{X} \to \operatorname{Spec}(A)$ and $\mathcal{Y} \to \operatorname{Spec}(A)$ are flat morphisms. Let $\delta_1: f^*\Omega_Y \to \Omega_X$ and $\delta_0: f^*\mathcal{O}_Y \to \mathcal{O}_X$ be the two morphisms induced by $f: X \to Y$. Then by \cite[Proposition $3.1$]{defofmaps89}, $\operatorname{Ext}^1(\delta_1, \delta_0)$ is the space of first order deformations of $\bf Def(f)$ and $\operatorname{Ext}^2(\delta_1, \delta_0)$ is a space of obstructions of $\bf Def(f)$ (see \cite[Section $2$]{defofmaps89} for the definition of these vector spaces).     
\end{definition}

\begin{definition}
\label{def with fixed target}
    Let $\bf Def(X/Y)$ denote the following functor of Artin rings. For an Artin ring $A$, $\bf Def(X/Y)(A)$ is the set of isomorphism classes of Cartesian diagrams as follows
$$  \bf Def(X/Y)(A) := \left\{ 
    \begin{tikzcd}
     X \arrow[r, hook] \arrow[d, "f"] & \mathcal{X} \arrow[d, "F"] \\
     Y \arrow[r, hook] \arrow[d] & Y \times \textrm{Spec}(A) \arrow[d] \\
     \textrm{Spec}(\mathbf{k}) \arrow[r, hook] & \textrm{Spec}(A) 
    \end{tikzcd}
\right\}
$$

where $\mathcal{X} \to \operatorname{Spec}(A)$ is a flat morphism. Define $\bf Def'(X/Y)$ to be the functor of Artin rings where for an Artin ring $A$, $\bf Def'(X/Y)(A)$ is the set of Cartesian diagrams as above with the additional condition that $\mathcal{X} \to \operatorname{Spec}(A)$ defines a locally trivial deformation of $X$.  
By \cite[Theorem $3.4.8$, Lemma $3.4.7$, Definition $3.4.5$]{sernesi2006deformations}, $H^0(N_f')$ is the space of first order deformations and $H^1(N_f')$ is the space of obstructions for $\bf Def'(X/Y)$ where $N_f'$ (called the equisingular normal sheaf) is defined by the exact sequence 
$$0 \to T_X \to f^*T_Y \to N_f' \to 0$$
If $f: X \to Y$ is a $G-$ cover with $G$ abelian, then let $\bf Def^G(X/Y)$ be the functor, where for any Artin ring $A$, we require the Cartesian diagram in \Cref{def with fixed target}, to satisfy the additional condition that the map $F: \mathcal{X} \to Y \times \operatorname{Spec}(A)$ is a $G-$ cover over $Y \times \operatorname{Spec}(A)$.   
\end{definition}


\begin{remark}\label{def that are abelian}
 Let $X$ and $Y$ be normal varieties with $H^1(\mathcal{O}_Y) = 0$. Let $f: X \to Y$ be a $G-$ cover, where $G = \mathbb{Z}_2^s$. If $Y$ does not have $2$- torsion in $\operatorname{Cl}(Y)$, by \Cref{eq:half_sum_relation}, we know that $X$ is determined by a set of divisors $\{D_{H,\psi}\}$, satisfying some fundamental relations in $\operatorname{Cl}(Y)$, where $H$ is a cyclic subgroup of $G$ and $\psi$ is a generator of $H^*$. Further, any set of divisors $\{D_{H,\psi}\}$ satisfying the set of fundamental relations mentioned in \Cref{eq:half_sum_relation} in $\operatorname{Cl}(Y)$, uniquely determines $G$- cover $f: X \to Y$. Now fix such a $G$-cover given by the data $\{D_{H,\psi}\}$ such that $f_*\mathcal{O}_X = \mathcal{O}_Y \oplus \mathcal{E}$ where $\mathcal{E} = \bigoplus_{\chi \in G^* \setminus \{1\}} L_{\chi}^{-1}$. Since $H^1(\mathcal{O}_Y) = 0$, the line bundles $L_{\chi}^{-1}$ do not have any non-trivial deformations when $Y$ is fixed. Hence the versal deformation space of $\bf Def^G(X/Y)$ is given by an open set of the smooth scheme 

 $$\bigoplus_{(H, \Psi)} H^0(\mathcal{O}_Y(D_{H,\Psi}))$$

\end{remark}

\begin{remark}\label{smoothness}
We recall a well-known fact from deformation theory. Let $F$ and $G$ be two functors of Artin rings and $f: F \to G$ be a morphism of functors. If $df: F(k[\epsilon]) \to G(k[\epsilon])$ is surjective and $F$ is less obstructed than $G$, then $f$ is smooth (see \cite[Proposition $2.3.6$]{sernesi2006deformations}). Consequently, $f(B): F(B) \to G(B)$ is surjective for every $B$. Further in this situation, $F$ is unobstructed if and only if $G$ is unobstructed (see \cite[Proposition $2.2.5$ (ii), (iii)]{sernesi2006deformations}).

\end{remark}

\begin{remark}\label{local-global}
Let $Y$ be a reduced scheme. Then there is a local-global Ext sequence for coherent sheaves $\mathcal{F}$ and $\mathcal{G}$ as follows:
\begin{equation*}
 0 \to H^1(\mathcal{H}om(\mathcal{F}, \mathcal{G})) \to \operatorname{Ext}^1(\mathcal{F}, \mathcal{G}) \to H^0(\mathcal{E}xt^1(\mathcal{F}, \mathcal{G})) \to H^2(\mathcal{H}om(\mathcal{F}, \mathcal{G})) \to \operatorname{Ext}^2(\mathcal{F}, \mathcal{G})
 \end{equation*}

Set $\mathcal{F} = \Omega_Y$ and $\mathcal{G} = \mathcal{E}$ where  $\mathcal{E}$ is a locally free sheaf on $Y$. Then the above sequence gives

\begin{equation}\label{twisted local-global}
    0 \to H^1(T_Y \otimes \mathcal{E}) \to \operatorname{Ext}^1(\Omega_Y, \mathcal{E}) \to H^0(\mathcal{T}_Y^1 \otimes \mathcal{E}) \to H^2(T_Y \otimes \mathcal{E}) \to \operatorname{Ext}^2(\Omega_Y, \mathcal{E})
\end{equation}
\end{remark}


\subsection{Deformation of non-flat 
$\mathbb{Z}_2^s$-covers}

\begin{theorem}\label{every def is a finite cover revised reflexive}
    Let $f: X \to Y$ be a finite cover, with $X$ and $Y$ reduced and non-singular in codimension one, such that $f_*\mathcal{O}_X = \mathcal{O}_Y \oplus \mathcal{E}$, where $\mathcal{E}$ is a reflexive sheaf on $Y$. Suppose that $H^1(\mathcal{H}om(\Omega_Y, \mathcal{E})) = H^0(\mathcal{E}xt^1(\Omega_Y, \mathcal{E})) = 0$. Then the natural forgetful map of functors $\bf Def(f) \xrightarrow{\pi_0} \bf Def(X)$ is smooth. Consequently, any deformation of $X$ can be realized as a finite cover over a deformation of $Y$. 
\end{theorem}

\begin{proof}
    We use \cite[Proposition $1.10$]{wehler86deformations}. Since $f$ is finite, we have, by adjunction (see \cite[Proposition $5.10$]{Hartresiduesandduality}), for $i = 1,2$ \par
    \noindent $\operatorname{Ext}^i(Lf^*\mathcal{L}_{Y}^{\cdot}, \mathcal{O}_X) = \operatorname{Ext}^i(\mathcal{L}_{Y}^{\cdot}, Rf_*\mathcal{O}_X) = \operatorname{Ext}^i(\mathcal{L}_{Y}^{\cdot}, f_*\mathcal{O}_X) = \operatorname{Ext}^i(\mathcal{L}_{Y}^{\cdot}, \mathcal{O}_Y) \oplus \operatorname{Ext}^i(\mathcal{L}_{Y}^{\cdot}, \mathcal{E})$
    Hence, we have to show that the map (which is inclusion in the first coordinate)
    $$\beta_1 : \operatorname{Ext}^1(\mathcal{L}_{Y}^{\cdot}, \mathcal{O}_Y) \to \operatorname{Ext}^1(\mathcal{L}_{Y}^{\cdot}, \mathcal{O}_Y) \oplus \operatorname{Ext}^1(\mathcal{L}_{Y}^{\cdot}, \mathcal{E})$$
    
    is surjective and the map 

    $$\beta_2 : \operatorname{Ext}^2(\mathcal{L}_{Y}^{\cdot}, \mathcal{O}_Y) \to \operatorname{Ext}^2(\mathcal{L}_{Y}^{\cdot}, \mathcal{O}_Y) \oplus \operatorname{Ext}^2(\mathcal{L}_{Y}^{\cdot}, \mathcal{E})$$
    is injective. \par

    While $\beta_2$ is always injective, $\beta_1$ is surjective if and only if 
    $\operatorname{Ext}^1(\mathcal{L}_{Y}^{\cdot}, \mathcal{E}) = 0$. Now there is a Spectral sequence
    $$E_2^{p,q} = H^p(\mathcal{T}^q(Y,\mathcal{E})) \implies T^{p+q}(Y,\mathcal{E})$$
    where $\mathcal{T}^q(Y,\mathcal{E}) = \mathcal{E}xt^q(\mathcal{L}_{Y}^{\cdot}, \mathcal{E})$ and $T^q(Y,\mathcal{E}) = \operatorname{Ext}^q(\mathcal{L}_{Y}^{\cdot}, \mathcal{E})$ .
    Hence in order to show that $T^1(Y, \mathcal{E}) = 0$, it is enough to show $H^1(\mathcal{T}^0(Y,\mathcal{E})) = 0$ and $H^0(\mathcal{T}^1(Y,\mathcal{E})) = 0$. 
    

We verify that $\mathcal{T}^i(Y,\mathcal{E}) = \mathcal{E}xt^i(\Omega_Y,\mathcal{E})$ for $i=0,1$.
Locally, choose a closed embedding $Y \hookrightarrow P$ with $P$ smooth, and let
$I$ be the ideal of $Y$ in $P$. The (naive) cotangent complex is the two-term
complex $c_\bullet = [I/I^2 \xrightarrow{d} \Omega_P \otimes \mathcal{O}_Y]$ in degrees $[-1,0]$,
so $\mathcal{T}^i(Y,\mathcal{E}) = H^i(\mathcal{H}om(c_\bullet, \mathcal{E}))$ vanishes for $i \geq 2$.
Let $N_0 = \operatorname{Im}(d)$, so that
$0 \to N_0 \to \Omega_P \otimes \mathcal{O}_Y \to \Omega_Y \to 0$
is exact, and let $K = \ker(d: I/I^2 \to \Omega_P \otimes \mathcal{O}_Y)$.
Since $\Omega_P \otimes \mathcal{O}_Y$ is locally free, applying $\mathcal{H}om(-,\mathcal{E})$
to the short exact sequence above gives
\[
\mathcal{E}xt^1(\Omega_Y, \mathcal{E})
= \operatorname{coker}\bigl(\mathcal{H}om(\Omega_P \otimes \mathcal{O}_Y, \mathcal{E})
\to \mathcal{H}om(N_0, \mathcal{E})\bigr).
\]
On the other hand, by definition,
\[
\mathcal{T}^1(Y,\mathcal{E})
= \operatorname{coker}\bigl(\mathcal{H}om(\Omega_P \otimes \mathcal{O}_Y, \mathcal{E})
\to \mathcal{H}om(I/I^2, \mathcal{E})\bigr).
\]
The surjection $I/I^2 \twoheadrightarrow N_0$ with kernel $K$ induces an injection
$\mathcal{H}om(N_0, \mathcal{E}) \hookrightarrow \mathcal{H}om(I/I^2, \mathcal{E})$
and hence a canonical injection
$\mathcal{E}xt^1(\Omega_Y, \mathcal{E}) \hookrightarrow \mathcal{T}^1(Y, \mathcal{E})$
whose cokernel is a subsheaf of $\mathcal{H}om(K, \mathcal{E})$.
Since $Y$ is reduced, $d$ is injective at every generic point of $Y$,
so $K$ is a torsion sheaf. Since $\mathcal{E}$ is torsion-free,
$\mathcal{H}om(K, \mathcal{E}) = 0$, and therefore
$\mathcal{T}^1(Y, \mathcal{E}) = \mathcal{E}xt^1(\Omega_Y, \mathcal{E})$.
The case $i=0$ is immediate:
$\mathcal{T}^0(Y,\mathcal{E}) = \ker(d^*) = \mathcal{H}om(\Omega_Y, \mathcal{E})$.

\end{proof}

\begin{remark}\label{when E is a vector bundle}
    Note that if $\mathcal{E}$ is a vector bundle, $\mathcal{E}xt^q(\Omega_Y, \mathcal{E}) = \mathcal{E}xt^q(\Omega_Y, \mathcal{O}_Y) \otimes \mathcal{E} = \mathcal{T}_Y^q \otimes \mathcal{E}$. Hence in this case, if $H^1(T_Y \otimes \mathcal{E}) = H^0(\mathcal{T}_Y^1 \otimes \mathcal{E}) = 0$, then the natural forgetful map of functors $\bf Def(f) \xrightarrow{\pi_0} \bf Def(X)$ is smooth. 
\end{remark}

In the following proposition, we prove a sufficient condition to ensure the vanishing of $H^0(\mathcal{E}xt^1(\Omega_X, L))$ on a scheme $X$.

\begin{proposition}\label{analog Schlessinger main}
    Let $p : Z \to X$ be a morphism of geometric local schemes with closed points $z$ and $x$ respectively. Suppose that $Z$ is smooth and that $X = Z/G$, where $G$ is a finite group of automorphisms of $Z$, having $z$ as the only fixed point. Let $L$ be a reflexive sheaf of rank one on $X$ which is locally free outside finitely many closed points. If $\operatorname{dim}(Z) \geq 3$ and $H_z^2(Z, T_Z \otimes p^*L) = 0$, then $H^0(\mathcal{E}xt^1(\Omega_X, L)) = 0$.
\end{proposition}

To prove this Proposition, we first prove a lemma which is the analog of \cite[Lemma $(2)$]{schlessinger1971}.

\begin{lemma}\label{exact sequence with coefficients in E}
  Let $i: X \hookrightarrow Y$ be a closed immersion of geometric local schemes, and $x$ is a closed point of $X$. Let $U = X -x$. Let $L$ denote a reflexive sheaf of rank one on $X$. Assume
  \begin{enumerate}
      \item $\operatorname{depth}_xX \geq 2$
      \item $Y$ and $U$ are smooth.
  \end{enumerate}
Then there is an exact sequence 
$$0 \to H^0(\mathcal{E}xt^1(\Omega_X, L)) \to H^1(U,T_X \otimes L) \to H^1(U, i^*T_Y \otimes L)$$

\end{lemma}

\begin{proof}
    Let $\mathcal{I}$ denote the ideal sheaf of $X$ in $Y$. Then we have an exact sequence,
    $$\mathcal{I}/\mathcal{I}^2 \xrightarrow{d}  i^*\Omega_Y \to \Omega_X \to 0$$
    Let $N_0$ denote the image of $d$. Taking the $\mathcal{H}om( -, L)$, we have the exact sequence,
    $$0 \to \mathcal{H}om(\Omega_X, L) \to \mathcal{H}om(i^*\Omega_Y, L) \to \mathcal{H}om(N_0, L) \to \mathcal{E}xt^1(\Omega_X, L) \to \mathcal{E}xt^1(i^*\Omega_Y, L) = 0$$

The last vanishing is due to the fact that $Y$ is smooth and hence $\Omega_Y$ is locally free. Now since $L$ is reflexive, $L = \mathcal{H}om(L^{\vee}, \mathcal{O}_X)$. Hence if $\mathcal{F}$ is a coherent sheaf on $X$, we have 
$$\mathcal{H}om(\mathcal{F}, L) = \mathcal{H}om(\mathcal{F}, \mathcal{H}om(L^{\vee}, \mathcal{O}_X)) = \mathcal{H}om(\mathcal{F} \otimes L^{\vee}, \mathcal{O}_X) = (\mathcal{F} \otimes L^{\vee})^{\vee} $$

Therefore the last exact sequence gives

\begin{equation*}\label{exact sequence after Hom}
 0 \to (\Omega_X \otimes L^{\vee})^{\vee} \to (i^*\Omega_Y \otimes L^{\vee})^{\vee} \to (N_0 \otimes L^{\vee})^{\vee} \to \mathcal{E}xt^1(\Omega_X, L) \to 0   
\end{equation*}

Since the first three sheaves in the exact sequence are duals, by \cite[Lemma $(1)$]{schlessinger1971}, $\operatorname{depth}_x\mathcal{F}^{\vee} \geq 2$ and $H^0(U, \mathcal{F}) = H^0(X, \mathcal{F})$ for the first three sheaves. Since coherent sheaves on Noetherian affine schemes have no higher cohomology, by taking global sections of the last exact sequence, we have 

\begin{equation*}
   0 \to H^0(U,(\Omega_X \otimes L^{\vee})^{\vee}) \to H^0(U, (i^*\Omega_Y \otimes L^{\vee})^{\vee}) \to H^0(U, (N_0 \otimes L^{\vee})^{\vee}) \to H^0(\mathcal{E}xt^1(\Omega_X, L)) \to 0 
\end{equation*}

Now $U$ is smooth. Note that reflexive sheaves of rank one on a smooth scheme are locally free. Hence $L|_U$ is locally free.  Further on $U$, the sheaf $\Omega_X|_U$ is locally free. Consequently, $N_0|_U$ is also locally free. Since duals commute with tensor products on locally free sheaves, we have that

\begin{equation}\label{global sections of exact sequence of duals}
    0 \to H^0(U, T_X \otimes L) \to H^0(U, i^*T_Y\otimes L) \to H^0(U, N_0^{\vee} \otimes L) \to H^0(\mathcal{E}xt^1(\Omega_X, L)) \to 0
\end{equation}

On the other hand, since $\mathcal{E}xt^1(\Omega_X, L)$ is supported on the complement of $U$, restricting to $U$, we get

\begin{equation}\label{exact sequence of duals restrcited to U}
   0 \to (T_X \otimes L)|_U \to (i^*T_Y\otimes L)|_U \to (N_0^{\vee} \otimes L)|_U \to 0 
\end{equation}   

Taking the cohomology of the above sequence, we have

\begin{equation}\label{global sections of exact sequence of duals restricted to U}
    0 \to H^0(U, T_X \otimes L) \to H^0(U, i^*T_Y\otimes L) \to H^0(U, N_0^{\vee} \otimes L) \to H^1(U, T_X \otimes L) \to H^1(U, i^*T_Y\otimes L)
\end{equation}

Now, comparing \Cref{global sections of exact sequence of duals} and \Cref{global sections of exact sequence of duals restricted to U}, we have the exact sequence in the statement of the Theorem.

\end{proof}

\noindent\textit{Proof of \Cref{analog Schlessinger main}.} Pick a closed immersion $X \hookrightarrow Y$, where $Y$ is smooth. Let $U = X -\{x\}$. Since $\operatorname{dim}_z(Z) \geq 3$, we have that $\operatorname{depth}_xX \geq \operatorname{dim}_x(X)  \geq 3$. We show that under this situation, $H^1(U, T_X \otimes L) = 0$. Then the statement follows from \Cref{exact sequence with coefficients in E}. Let $V = Z-\{z\}$, so that $p$ induces a map $p: V \to U = V/G$. Consider the sheaf $T_Z \otimes p^*L$ restricted to $V$. Note that $p_*(T_Z \otimes p^*L) = p_*T_Z \otimes L$. Further $H^1(V,T_Z \otimes p^*L) = H^1(U, p_*T_Z \otimes L)$. By \cite[Corollary $3.1.23$ (see also Proposition $3.1.22$)]{sernesi2006deformations}, $T_X|_U$ is a direct summand of $p_*T_Z|_U$. Hence $H^1(U, T_X \otimes L)$ is a direct summand of $H^1(U, p_*T_Z \otimes L) = H^1(V,T_Z \otimes p^*L)$. So it is enough to show that $H^1(V,T_Z \otimes p^*L) = 0$. Now by the exact sequence of local-cohomology, we have 

\begin{equation}\label{local cohomology}
  H^1(Z, T_Z \otimes p^*L) \to H^1(V,T_Z \otimes p^*L) \to H_z^2(Z, T_Z \otimes p^*L)  
\end{equation}

The left hand side term is zero since $Z$ is a Noetherian affine scheme and the right hand side term is zero by the hypothesis. Hence $H^1(V,T_Z \otimes p^*L) = 0$.

\color{black}

\medskip

We now give a sufficient criterion for the vanishing of local cohomology groups which will allow us to apply \Cref{analog Schlessinger main}, to weighted projective spaces.

\begin{proposition}\label{vanishing of second local cohomology}
Let $Z$ be a smooth variety of dimension $n$ over a field, let $S\subset Z$ be a finite
set of closed points, and set $U:=Z\setminus S$. Let $\mathcal F$ be a locally free
sheaf on $Z$, and let $\mathcal G$ be a coherent sheaf on $Z$ such that
$\mathcal G|_U$ is locally free of rank $1$ (equivalently, $\mathcal G$ is generically
rank $1$ and locally free outside $S$).
Then for every $z\in Z$ one has
\[
H^i_z(\mathcal F\otimes \mathcal G)=0
\qquad\text{for all }\quad 2\le i\le n-1.
\]

\end{proposition}

\begin{proof}
We break the argument into steps. Let us denote $\mathcal G_{\operatorname{tor}}:=\operatorname{tors}(\mathcal G)$ and
$\mathcal G':=\mathcal G/\mathcal G_{\operatorname{tor}}$.

\begin{enumerate}
\item\emph{Reduction to the torsion-free quotient.}
Let $\mathcal G_{\operatorname{tor}}=\operatorname{tors}(\mathcal G)$ and $\mathcal G'=\mathcal G/\mathcal G_{\operatorname{tor}}$.
Since $\mathcal G|_U$ is locally free, it has no torsion on $U$, hence
$\operatorname{Supp}(\mathcal G_{\operatorname{tor}})\subseteq S$. In particular, for every $z\in Z$ the stalk
$(\mathcal G_{\operatorname{tor}})_z$ has finite length over $\mathcal O_{Z,z}$.
Tensoring the short exact sequence
$0\to \mathcal G_{\operatorname{tor}}\to \mathcal G\to \mathcal G'\to 0$
with $\mathcal F$ (exactness is preserved since $\mathcal F$ is locally free) yields
\[
0\to \mathcal F\otimes \mathcal G_{\operatorname{tor}}
\to \mathcal F\otimes \mathcal G
\to \mathcal F\otimes \mathcal G' \to 0.
\]
Fix $z\in Z$ and set $R=\mathcal O_{Z,z}$ and $\mathfrak m=\mathfrak m_z$.
Because $(\mathcal F\otimes \mathcal G_{\operatorname{tor}})_z\cong \mathcal F_z\otimes_R (\mathcal G_{\operatorname{tor}})_z$ and $\mathcal F_z$ is
free, the module $(\mathcal F\otimes \mathcal G_{\operatorname{tor}})_z$ also has finite length.
For a finite length $R$--module $M$ one has $H^j_{\mathfrak m}(M)=0$ for all $j>0$
(e.g.\ by the \v{C}ech complex computing local cohomology with respect to generators of $\mathfrak m$).
Therefore, applying local cohomology supported at $z$ gives isomorphisms
\[
H^i_z(\mathcal F\otimes \mathcal G)\ \cong\ H^i_z(\mathcal F\otimes \mathcal G')
\qquad\text{for all } i\ge 1.
\]
In particular, for $2\le i\le n-1$ it suffices to prove the desired vanishing with
$\mathcal G$ replaced by $\mathcal G'$; thus we may assume $\mathcal G$ is torsion-free.

\item\emph{Reduction to the reflexive hull.}
Assume from now on that $\mathcal G$ is torsion-free and that $\mathcal G|_U$ is
locally free of rank $1$.
Consider the natural morphism
\[
\mathcal G \longrightarrow \mathcal G^{\vee\vee}.
\]
Since $Z$ is smooth (hence normal) and $\mathcal G$ is torsion-free, this morphism is
injective. Moreover, it is an isomorphism on $U$, because $\mathcal G|_U$ is locally
free (hence reflexive) on $U$. Consequently the cokernel
\[
\mathcal Q:=\mathcal G^{\vee\vee}/\mathcal G
\]
is supported on $Z\setminus U=S$, i.e.\ $\operatorname{Supp}(\mathcal Q)\subseteq S$.
In particular, $\mathcal Q_z$ has finite length for every $z\in Z$.
So we have an exact sequence 
$0\to \mathcal G\to \mathcal G^{\vee\vee} \to \mathcal Q\to 0$

\item\emph{Rank-one reflexive sheaves on a smooth variety are invertible.}
Because $\mathcal G|_U$ has rank $1$, the reflexive hull $\mathcal G^{\vee\vee}$ is a
rank-one reflexive sheaf on $Z$. Since $Z$ is smooth, it is locally factorial; hence
every rank-one reflexive sheaf is invertible. Therefore
\[
\mathcal L:=\mathcal G^{\vee\vee}
\]
is a line bundle on $Z$.

\item\emph{Vanishing of local cohomology in degrees $2,\dots,n-1$.}
Tensor the exact sequence
$0\to \mathcal G\to \mathcal L\to \mathcal Q\to 0$ with $\mathcal F$:
\[
0\to \mathcal F\otimes \mathcal G \to \mathcal F\otimes \mathcal L \to \mathcal F\otimes \mathcal Q \to 0.
\]
Fix $z\in Z$ and apply $H^i_z(-)$.
Since $(\mathcal F\otimes \mathcal Q)_z$ has finite length, we have
$H^j_z(\mathcal F\otimes \mathcal Q)=0$ for all $j>0$.
Also, $\mathcal F\otimes \mathcal L$ is locally free, so at the regular local ring
$\mathcal O_{Z,z}$ of dimension $n$ one has
$H^i_z(\mathcal F\otimes \mathcal L)=0$ for all $i\neq n$; in particular this holds
for $2\le i\le n-1$.
Thus for $2\le i\le n-1$ the segment of the long exact sequence
\[
H^{i-1}_z(\mathcal F\otimes \mathcal Q)\to
H^i_z(\mathcal F\otimes \mathcal G)\to
H^i_z(\mathcal F\otimes \mathcal L)\to
H^i_z(\mathcal F\otimes \mathcal Q)
\]
shows that $H^i_z(\mathcal F\otimes \mathcal G)=0$.

\end{enumerate}
\end{proof}

\medskip

Under some additional hypothesis, we would like to strengthen \Cref{every def is a finite cover revised reflexive} so that every deformation of the abelian cover $X$ is once again an abelian cover of $Y$. To do this, we first need a lemma to compute the normal sheaf of a finite map.

\begin{lemma}\label{lem:log_tangent_exact_sequence_isolated_quotient}
Let $Y$ be a normal variety of dimension $n \geq 3$ with isolated singularities,
$D$ a reduced Weil divisor on $Y$, and
\[
j:U:=Y_{\mathrm{reg}}\hookrightarrow Y,
\qquad
D_U:=D\cap U.
\]
Assume that $D_U$ is smooth. Assume moreover that for every point
$y\in \operatorname{Sing}(Y)$ such that $D$ passes through $y$, there exists a
neighbourhood $W_y$ of $y$ with $W_y \cap \operatorname{Sing}(Y) = y$ together with a
quotient presentation
\[
f: \widetilde{W}_y \to W_y \simeq \widetilde{W}_y/G_y,
\]
where $\widetilde{W}_y$ is smooth affine, $G_y$ is finite with
$\operatorname{char}(k)\nmid |G_y|$,
$f: \widetilde{W}_y- \{f^{-1}(y)\} \to W_y- \{y\}$ is \'etale,
and the reduced pullback divisor
$D_y'\subset \widetilde{W}_y$ is smooth at every point lying over $y$.

Define
\[
T_Y(-\log D):=j_*T_U(-\log D_U).
\]
Then the sequence obtained by applying $j_*$ to the short exact sequence
\[
0\longrightarrow T_U(-\log D_U)
\longrightarrow T_U
\longrightarrow \mathcal O_{D_U}(D_U)
\longrightarrow 0
\]
is exact on the right, i.e.\ $R^1j_*T_U(-\log D_U)=0$,
and hence there is a short exact sequence
\[
0\longrightarrow T_Y(-\log D)
\longrightarrow T_Y
\longrightarrow j_*\mathcal O_{D_U}(D_U)
\longrightarrow 0.
\]
\end{lemma}

\begin{proof}
On $U$, since $U$ is smooth and $D_U$ is a smooth divisor,
we have the standard short exact sequence
\[
0 \longrightarrow T_U(-\log D_U) \longrightarrow T_U
\longrightarrow \mathcal{O}_{D_U}(D_U) \longrightarrow 0.
\]
Applying the left exact functor $j_*$ yields the exact sequence
\[
0 \longrightarrow j_*T_U(-\log D_U)
\longrightarrow j_*T_U
\longrightarrow j_*\mathcal{O}_{D_U}(D_U)
\longrightarrow R^1j_*T_U(-\log D_U).
\]
By definition, $j_*T_U(-\log D_U) = T_Y(-\log D)$.
Since $Y$ is normal of dimension at least $3$ with isolated singularities,
$\operatorname{Sing}(Y)$ has codimension at least $3 \geq 2$, and $T_Y$ is reflexive,
so $j_*T_U = T_Y$.
Therefore, if $R^1j_*T_U(-\log D_U) = 0$, we obtain the claimed
short exact sequence. Since $R^1j_*T_U(-\log D_U)$ is quasi-coherent,
it suffices to show that its stalk at every singular point $y$ vanishes,
i.e.\ that $(R^1j_*T_U(-\log D_U))_y = 0$.

Set $F := T_U(-\log D_U)$.
Let
\[
u:Y_y:=\operatorname{Spec}\mathcal O_{Y,y}\longrightarrow Y
\]
be the canonical localization morphism, and consider the cartesian square
\[
\begin{array}{ccc}
U_y:=Y_y\times_Y U & \xrightarrow{\ v\ } & U\\
\big\downarrow {j_y} & & \big\downarrow {j}\\
Y_y & \xrightarrow{\ u\ } & Y.
\end{array}
\]

Since \(F\) is quasi-coherent on \(U\), and \(j\) is quasi-compact and
quasi-separated, it follows that \(R^1j_*F\) is quasi-coherent on \(Y\).
The morphism \(u\) is flat, because on an affine neighborhood
\(V=\operatorname{Spec}A\subset Y\) with \(y\leftrightarrow \mathfrak p\),
the morphism \(u\) is induced by the localization map
$A\to A_{\mathfrak p}$, and localization is flat.
Hence flat base change yields an isomorphism
\[
u^*R^1j_*F \xrightarrow{\;\sim\;} R^1j_{y*}\,v^*F.
\]

Since \(Y_y=\operatorname{Spec}\mathcal O_{Y,y}\) is affine, the global
sections of the higher direct image compute the corresponding cohomology
on the source; thus
\[
H^0\bigl(Y_y,\,R^1j_{y*}\,v^*F\bigr)
\;=\;
H^1(U_y,\,v^*F).
\]

It remains to identify \(\bigl(R^1j_*F\bigr)_y\) with
\(H^0(Y_y,\,u^*R^1j_*F)\). More generally, if \(\mathcal{G}\) is any
quasi-coherent sheaf on \(Y\), then
\[
\mathcal{G}_y \;\cong\; H^0(Y_y,\,u^*\mathcal{G}).
\]
Indeed, choose an affine neighborhood \(V=\operatorname{Spec}A\subset Y\) of
\(y\), and let \(\mathfrak p\subset A\) correspond to \(y\). Since
\(\mathcal{G}\) is quasi-coherent, we may write
\(\mathcal{G}|_V \cong \widetilde M\)
for some \(A\)-module \(M\). Then \(Y_y\cong \operatorname{Spec}A_{\mathfrak p}\),
and by the affine description of pullback of quasi-coherent sheaves,
\(u^*\mathcal{G} \cong \widetilde{M_{\mathfrak p}}\).
Taking global sections on the affine scheme \(Y_y\), we obtain
\[
H^0(Y_y,\,u^*\mathcal{G})\;=\;M_{\mathfrak p}\;=\;\mathcal{G}_y.
\]

Applying this to \(\mathcal{G}=R^1j_*F\), we conclude
\[
\bigl(R^1j_*F\bigr)_y
\;\cong\;
H^0\bigl(Y_y,\,u^*R^1j_*F\bigr)
\;\cong\;
H^0\bigl(Y_y,\,R^1j_{y*}\,v^*F\bigr)
\;=\;
H^1(U_y,\,v^*F).
\]
Therefore
\[
\bigl(R^1j_*T_U(-\log D_U)\bigr)_y
\;\cong\;
H^1\bigl(U_y,\,T_{U_y}(-\log D_{U_y})\bigr).
\]
Now we prove
\[
H^1\bigl(U_y,\,T_{U_y}(-\log D_{U_y})\bigr)=0.
\]

By hypothesis, there is an affine neighbourhood $W_y$ of $y$ in $Y$
and a smooth affine $\widetilde{W}_y$ such that
$W_y\simeq \widetilde W_y/G_y$ at $y$.
Let $\widetilde{Y}_y = Y_y \times_Y \widetilde{W}_y$. We get a finite map
\[
\pi_y:\widetilde Y_y\longrightarrow Y_y,
\]
where $\widetilde Y_y$ is a smooth affine semilocal scheme,
$Y_y\simeq \widetilde Y_y/G_y$, and the reduced pullback divisor
\[
\widetilde D_{Y_y}:=(\pi_y^{-1}(D_{Y_y}))_{\operatorname{red}}
\]
is smooth. Let
\[
Z_y:=\pi_y^{-1}(y),
\qquad
\widetilde U_y:=\widetilde Y_y\setminus Z_y.
\]
Since $\widetilde{W}_y- \{f^{-1}(y)\} \to W_y- \{y\}$ is finite \'etale,
the restricted morphism
\[
p_y:=\pi_y|_{\widetilde U_y}:\widetilde U_y\longrightarrow U_y
\]
is finite \'etale. Let
$\widetilde D_{U_y} = \widetilde D_{Y_y} \cap \widetilde U_y$.
Since $p_y$ is \'etale and $\widetilde D_{U_y}$ is the scheme-theoretic
pullback of $D_{U_y}$ (which is already reduced, as \'etale morphisms
do not introduce multiplicities), we have
\[
p_y^*T_{U_y}(-\log D_{U_y})
\;\cong\;
T_{\widetilde U_y}(-\log \widetilde D_{U_y}).
\]
Moreover, since $\operatorname{char}(k)\nmid |G_y|$, taking
$G_y$-invariants is exact, and therefore
$T_{U_y}(-\log D_{U_y})$ is a direct summand of
$p_{y*}T_{\widetilde U_y}(-\log \widetilde D_{U_y})$.
Hence
$H^1\bigl(U_y,\,T_{U_y}(-\log D_{U_y})\bigr)$
is a direct summand of
\[
H^1\bigl(U_y,\,
p_{y*}T_{\widetilde U_y}(-\log \widetilde D_{U_y})\bigr)
\;=\;
H^1\bigl(\widetilde U_y,\,
T_{\widetilde U_y}(-\log \widetilde D_{U_y})\bigr),
\]
since $p_y$ is finite, hence affine.
It is therefore enough to show that
\[
H^1\bigl(\widetilde U_y,\,
T_{\widetilde U_y}(-\log \widetilde D_{U_y})\bigr)=0.
\]

Now $\widetilde Y_y$ is affine semilocal and
$\widetilde U_y=\widetilde Y_y\setminus Z_y$, so by comparing
cohomology of the punctured spectrum with local cohomology we get
\[
H^1\bigl(\widetilde U_y,\,
T_{\widetilde U_y}(-\log \widetilde D_{U_y})\bigr)
\;\cong\;
H^2_{Z_y}\bigl(
T_{\widetilde Y_y}(-\log \widetilde D_y)\bigr).
\]
Thus it suffices to prove that
\[
H^2_{Z_y}\bigl(
T_{\widetilde Y_y}(-\log \widetilde D_y)\bigr)=0.
\]

Since $\widetilde D_y$ is a smooth divisor on the smooth scheme
$\widetilde Y_y$, we have the short exact sequence
\[
0\longrightarrow T_{\widetilde Y_y}(-\log \widetilde D_y)
\longrightarrow T_{\widetilde Y_y}
\longrightarrow \mathcal O_{\widetilde D_y}(\widetilde D_y)
\longrightarrow 0.
\]
Applying local cohomology with support in $Z_y$, it is enough to
show that for every $x\in Z_y$,
\[
H^2_x\bigl(T_{\widetilde Y_y}\bigr)=0
\qquad\text{and}\qquad
H^1_x\bigl(\mathcal O_{\widetilde D_y}(\widetilde D_y)\bigr)=0.
\]
The first vanishing holds because
$T_{\widetilde Y_y}$ is locally free on the
smooth scheme $\widetilde Y_y$ of dimension $\ge 3$,
so it has depth at least $3$ at every point, hence
$H^i_x(T_{\widetilde Y_y}) = 0$ for $i < 3$.
For the second, if $x \notin \widetilde D_y$ then
$\mathcal{O}_{\widetilde D_y}(\widetilde D_y)$ is zero near $x$
and the vanishing is trivial. If $x \in \widetilde D_y$, then
$\mathcal O_{\widetilde D_y}(\widetilde D_y)$ is a line bundle on the
smooth divisor $\widetilde D_y$ of dimension at least $2$,
hence it has depth at least $2$ as an
$\mathcal{O}_{\widetilde Y_y,x}$-module, so
$H^1_x(\mathcal O_{\widetilde D_y}(\widetilde D_y))=0$.
Therefore the long exact sequence in local cohomology gives
\[
H^2_{Z_y}\bigl(
T_{\widetilde Y_y}(-\log \widetilde D_y)\bigr)=0,
\]
and hence
\[
H^1\bigl(U_y,\,T_{U_y}(-\log D_{U_y})\bigr)=0.
\]
\end{proof}

\begin{theorem}\label{every def is an abelian cover}

Let $f: X \to Y$ be an abelian cover with $\operatorname{dim}(X) = \operatorname{dim}(Y) \geq 3$, such that $X$ and $Y$ are normal with isolated quotient singularities. Let $f_*\mathcal{O}_X = \mathcal{O}_Y \oplus \mathcal{E}$, where $\mathcal{E} = \oplus_{\chi \in G^* \setminus \{1\}} L_{\chi}^{-1}$, 
\textcolor{black}{and each $L_{\chi}$ is a reflexive sheaf of rank one which is locally free outside a finite set of points}. Suppose that 
       \begin{enumerate}
        \item $H^1(T_Y) = 0$ and $ H^1(T_Y \otimes L_{\chi}^{-1}) = 0 $ for $\chi \in G^*-\{e\}$.
        \item $H^1(\mathcal{O}_Y) = 0$ and $H^1(L_{\chi}^{-1}) = 0$ $\forall \chi$
        \item $H^0(\mathcal{O}_Y(D_{H,\Psi}) \otimes L_{\chi}^{-1}) = 0$ for $(H,\Psi, \chi)$ such that $\chi|_H \neq \Psi^{m_H-1}$ and $\chi \neq 0$,
        \item The divisors $\{D_{H,\Psi}\}$ are normal and $D_{H,\Psi, U} = D_{H,\Psi} \cap U$ is smooth.
        \item for every divisor $D_{H,\Psi}$ and for every point
        $y\in \operatorname{Sing}(Y)$ such that $D_{H,\Psi}$ passes through $y$, there exists a neighbourhood $W_y$ of $y$ with $W_y \cap \operatorname{Sing}(Y) = y$, together with a
        quotient presentation
       \[
       f: \widetilde{W}_y \to W_y \simeq \widetilde{W}_y/G_y,
       \]
       where $\widetilde{W}_y$ is smooth affine, $G_y$ is finite, $f: \widetilde{W}_y- \{f^{-1}(y)\} \to W_y- \{y\}$ e\'tale
        and, if  $x_y\in \widetilde{W}_y$ denotes a point lying over $y$, then the reduced pullback
        divisor $D_{H,\Psi,y}'\subset \widetilde{W}_y$ of $D_{H,\Psi,y} = D_{H,\Psi} \cap W_y$  is smooth at $x_y$.
        \item Assume that there exists a smooth open set $i: U \hookrightarrow Y$ such that $\operatorname{codim}(Y-U) \geq 3$ and $f^{-1}(U) \subset X$ is smooth. Denote by $D_{H,\Psi, U} = D_{H,\Psi} \cap U$.
        \end{enumerate}
Then any deformation of $X$ can be realized as an abelian cover over $Y$. Further in this situation $X$ is unobstructed and represents a smooth point in its moduli.
\end{theorem}

\begin{proof}
We will show that the forgetful maps $\bf Def^G(X/Y) \xrightarrow{\pi_2} \bf Def'(X/Y) \xrightarrow{\pi_1} \bf Def(X)$ are smooth. We first show that the map $\bf Def^G(X/Y) \xrightarrow{\pi_2} \bf Def'(X/Y)$ is smooth. Since $\bf Def^G(X/Y)$ is unobstructed (see \Cref{def that are abelian}), it is enough to show that the map $d_{\pi_2}$ between the spaces of first order deformations is surjective. By \Cref{def with fixed target}, the locally trivial first order deformations of the map $X \to Y$ with fixed base $Y$ is given by $H^0(N_f')$ where $N_f'$ (called the equisingular normal sheaf). First we make the following claim.

\textbf{Claim:}
    Let $f : X \to Y$ be an abelian cover with $X$ and $Y$ normal with isolated quotient singularities. 
    \begin{enumerate}
        \item Assume that there exists an smooth open set $i: U \hookrightarrow Y$ such that $\operatorname{codim}(Y-U) \geq 3$ and $f^{-1}(U) \subset X$ is smooth. Denote by $D_{H,\Psi, U} = D_{H,\Psi} \cap U$.

        \item for every divisor $D_{H,\Psi}$ and for every point
        $y\in \operatorname{Sing}(Y)$ such that $D_{H,\Psi}$ passes through $y$, there exists a neighbourhood $W_y$ of $y$ with $W_y \cap \operatorname{Sing}(Y) = y$, together with a
        quotient presentation
       \[
       f: \widetilde{W}_y \to W_y \simeq \widetilde{W}_y/G_y,
       \]
       where $\widetilde{W}_y$ is smooth affine, $G_y$ is finite, $f: \widetilde{W}_y- \{f^{-1}(y)\} \to W_y- \{y\}$ e\'tale
        and, if  $x_y\in \widetilde{W}_y$ denotes a point lying over $y$, then the reduced pullback
        divisor $D_{H,\Psi,y}'\subset \widetilde{W}_y$ of $D_{H,\Psi,y} = D_{H,\Psi} \cap W_y$  is smooth at $x_y$.
    \end{enumerate}
      Then 
    $$f_*N_f' = \oplus_{\chi \in G^*} \oplus_{(H,\Psi)| \chi|_H \neq \Psi^{m_H-1}} \big(i_*\mathcal{O}_{D_{H,\Psi, U}}(D_{H,\Psi, U})\big) \otimes L_{\chi}^{-1}$$

\noindent\textit{Proof of Claim:} We have 

$$0 \to T_X \to f^*T_Y \to N_f'  \to  0$$

Pushing forward, we have 

$$0 \to f_*T_X \to T_Y \otimes f_*\mathcal{O}_X \to f_*N_f' \to 0 $$

Note that $f_*T_X$ is reflexive sheaf. Since $\operatorname{codim}(Y-U) \geq 2$, we have that $f_*T_X = i_*f_*T_{f^{-1}(U)}$. On the other hand, by \cite[Proposition $4.1$ page $207$]{pardini1991abelian}, 
$$i_*f_*T_{f^{-1}(U)} = \bigoplus_{\chi \in G^*} i_*T_U(-\operatorname{log} D_{H,\Psi}|_U; \chi|_H \neq \Psi^{m_H-1}) \otimes L_{\chi}^{-1} = \bigoplus_{\chi \in G^*} T_Y(-\operatorname{log} D_{H,\Psi}; \chi|_H \neq \Psi^{m_H-1}) \otimes L_{\chi}^{-1}$$
and the last equality follows by definition of the logarithmic tangent sheaves (see \Cref{lem:log_tangent_exact_sequence_isolated_quotient}).

So $f_*N_f'$ is the quotient of $T_Y \otimes f_*\mathcal{O}_X = \oplus_{\chi \in G^*} T_Y \otimes L_{\chi}^{-1}$ by $\oplus_{\chi \in G^*} T_Y(-\operatorname{log} D_{H,\Psi}; \chi|_H \neq \Psi^{m_H-1})\otimes L_{\chi}^{-1}$. Now the claim follows from \Cref{lem:log_tangent_exact_sequence_isolated_quotient} since tensoring is right exact.  

Now note that for a Weil divisor on $U$, we have the exact sequence

\begin{equation*}
    0 \to \mathcal{O}_U \to \mathcal{O}_U(D_U) \to \mathcal{O}_{D_U}(D_U) \to 0
\end{equation*}

Applying $i_*$, we get 

\begin{equation*}
    0 \to \mathcal{O}_Y \to \mathcal{O}_Y(D) \to i_*\big(\mathcal{O}_{D_U}(D_U)\big) \to R^1i_*\mathcal{O}_U
\end{equation*}

We show that $R^1i_*\mathcal{O}_U = 0$. We only need to check that the stalk of the sheaf is $0$ at the isolated singularities. The question is local, so we can assume that assume that $Y$ is affine with an isolated singularity at $y$ and $U = Y-\{y\}$. Note that we have $\big(R^1i_*\mathcal{O}_U\big)_y = H_y^2(\mathcal{O}_Y)$ 
But since the depth of $y$ at $Y$ is at least $3$, we have $H_y^2(\mathcal{O}_Y) = 0$. Hence we have an exact sequence 
\begin{equation*}
    0 \to \mathcal{O}_Y \to \mathcal{O}_Y(D) \to i_*\big(\mathcal{O}_{D_U}(D_U)\big) \to 0
\end{equation*}

Now under conditions $(3)$ and $(4)$ of the Theorem, we have that $$H^0(N_f') = \oplus_{(H,\Psi)}H^0\big(i_*\mathcal{O}_{D_{H,\Psi,U}}(D_{H,\Psi,U})\big) = \oplus_{(H,\Psi)}H^0\big(\mathcal{O}_{D_{H,\Psi,U}}(D_{H,\Psi,U})\big)$$

Hence $d_{\pi_2}$ surjects by \Cref{def that are abelian}. Therefore the map $\bf Def^G(X/Y) \xrightarrow{\pi_2} \bf Def'(X/Y)$ is smooth. Since $\bf Def^G(X/Y)$ is unobstructed, we  have that $\bf Def'(X/Y)$ is unobstructed as well.

Now we show that $\bf Def'(X/Y) \xrightarrow{\pi_1} \bf Def(X)$ is smooth. Since we have $\bf Def'(X/Y)$ is unobstructed, we only need to show that the map $d_{\pi_1}$ at the level of first-order deformations is surjective. 

The space of first order deformations of $X$, is given by $\operatorname{Ext}^1(\Omega_X, \mathcal{O}_X)$. 
By \cite[Theorem $3$]{schlessinger1971}, we have that $H^0(\mathcal{T}_X^1) = H^0(\mathcal{T}_Y^1) = 0$.
    Setting $\mathcal{F} = \Omega_X$ and $\mathcal{G} = \mathcal{O}_X$ in the local global Ext sequence, we have 
    \begin{equation}\label{not twisted local-global}
    0 \to H^1(T_X) \to \operatorname{Ext}^1(\Omega_X, \mathcal{O}_X) \to H^0(\mathcal{T}_X^1) \to H^2(T_X)
\end{equation}
Now using the fact that $H^0(\mathcal{T}_X^1) = 0$, we conclude that $H^1(T_X) = \operatorname{Ext}^1(\Omega_X, \mathcal{O}_X)$.

Therefore the map at the level of first-order deformations induced by the forgetful map $\bf Def'(X/Y) \to \bf Def(X)$ is obtained by taking the cohomology of the exact sequence.

$$0 \to T_X \to f^*T_Y \to N_f' \to 0$$

Taking cohomology we get 

$$H^0(N_f') \to H^1(T_X) \to H^1(f^*T_Y) = H^1(T_Y) \oplus H^1(T_Y \otimes \mathcal{E}) $$

Now the conclusion follows from the assumptions.

So we have shown that the composition of the forgetful maps $$\bf Def^G(X/Y) \xrightarrow{\pi_2} \bf Def'(X/Y) \xrightarrow{\pi_1} \bf Def(X)$$ is smooth.

\end{proof}

\subsection{Applications to  $(\mathbb{Z}_2)^s$ covers of weighted projective spaces}

We start this section with an application of \Cref{every def is a finite cover revised reflexive}.
The reflexivity condition on the following results are automatically satisfied when $X$ is normal.
\begin{corollary}\label{every def is an finite cover over a WPS reflexive}
Let $f: X \to Y$ be a $\mathbb{Z}_2^s$ cover where $Y = \mathbb{P}(a_0,...,a_r)$, $r \geq 3$ and $\operatorname{gcd}(a_i, a_j) = 1$ for all $i,j$. Let $f_*\mathcal{O}_X = \mathcal{O}_Y \oplus \mathcal{E}$, where $\mathcal{E} = \oplus_{\chi \in G^* \setminus \{1\}} L_{\chi}^{-1}$ and $L_{\chi}$ are reflexive sheaves of rank one which are locally free outside a finite set of points. Then the natural forgetful map of functors $\bf Def(f) \to \bf Def(X)$ is smooth. Consequently any deformation of $X$ can be realized as a finite cover over a deformation of $Y$. 
    \end{corollary}
\begin{proof}
    First note that if $\Omega_Y^{\vee \vee}$ denote the double dual of $\Omega_Y$, then we have that $\mathcal{H}om(\Omega_Y, L_{\chi}^{-1}) = \mathcal{H}om(\Omega_Y^{\vee \vee}, L_{\chi}^{-1})$. This is because both of them are reflexive sheaves since $L_{\chi}^{-1}$ is a reflexive sheaf (see \cite[Lemma $31.12.8$]{stacks-project}) and they agree outside of a locus of codimension two. Now from the Euler exact sequence (see \cite[Theorem $8.1.6$]{cox2024toric}, we have $H^1(\mathcal{H}om(\Omega_Y^{\vee \vee}, L_{\chi}^{-1})) = 0$ since $r \geq 3$ (see the proof of $H^1(T_Y \otimes L_{\chi}^{-1}) = 0$ in \Cref{every def is an abelian cover over a WPS}) . By \Cref{every def is a finite cover revised reflexive} it remains to show that $H^0(\mathcal{E}xt^1(\Omega_Y, \mathcal{E})) = 0$. Note that by the condition on the weights of $\mathbb{P}(a_0,...,a_r)$, $Y$ has isolated singularities. Further locally around every singular point $y$, $Y$ is isomorphic to $\mathbb{A}^3/G_y$ where $G_y$ is a finite group. Since $\mathcal{L}_{\chi}^{-1}$ are reflexive sheaves, we have by \Cref{vanishing of second local cohomology}, $H_z^2(T_{\mathbb{A}^3}\otimes p^*\mathcal{L}_{\chi}^{-1}) = 0$, where $p: \mathbb{A}^3 \to \mathbb{A}^3/G_y$ is the quotient map. Now the vanishing follows from \Cref{analog Schlessinger main}. 
\end{proof}

We now prove the following result as consequence of \Cref{every def is an abelian cover}.

\begin{corollary}\label{every def is an abelian cover over a WPS}
Let $f: X \to Y$ be a $\mathbb{Z}_2^s$ cover where $Y = \mathbb{P}(a_0,...,a_r)$, with $r \geq 3$, $\operatorname{gcd}(a_i, a_j) = 1$ for all $i,j$ and $X$ has at worst isolated quotient singularities. Let $f_*\mathcal{O}_X = \mathcal{O}_Y \oplus \mathcal{E}$, where $\mathcal{E} = \oplus_{\chi \in G^*\setminus \{1\}} L_{\chi}^{-1}$ such that each $L_{\chi}$ 
\textcolor{black}{is a reflexive sheaf of rank one} which is locally free outside a finite set of points. Suppose that 
    \begin{enumerate}
        \item the branch divisors $D_g \in | \mathcal{O}_{\mathbb{P}}(d_g)|$ associated to the data of the abelian cover be quasismooth and generic (see \Cref{def:divisor_standard set-up}).
        
        \item For each pair $(g, \chi)$ such that $\chi \cdot g = 0$, 
        that is $\chi(g) =  (-1)^{\chi \cdot g} = 1$
        and $\chi \neq \operatorname{Id} \in G^*$, 
        \begin{align*}
        d_g < \displaystyle\frac{1}{2} \Sigma_{\chi \cdot g' = 1} d_{g'}
        =
        \displaystyle\frac{1}{2} 
        \Sigma_{\chi(g') = -1} d_{g'},
        \end{align*}

        \item $\Sigma_g d_g > 2 \Sigma_{j=0}^r a_j$
    \end{enumerate}

Then any deformation of $X$ can be realized as an abelian cover over $Y$. Further, in this situation, $X$ is unobstructed and represents a smooth point in its moduli.
    
\end{corollary}

\begin{proof}
 We apply \Cref{every def is an abelian cover}. First of all, under the conditions on the weights, $\mathbb{P}(a_0,...,a_r)$ has isolated quotient singularities. By \Cref{singularities of the cover}, we have that $X$ also has isolated quotient singularities. We have by \cite[Section $1.4$]{dolgachevwps82}, $H^1(L_{\chi}^{-1}) = 0$ for all $\chi \in G^*$.
 Further from the dual of the Euler exact sequence (see \cite[Theorem $8.1.6$]{cox2024toric} we have
 $$0 \to \mathcal{O}_Y \to \bigoplus_{i = 1}^r  \mathcal{O}_Y(a_i) \to T_Y \to 0$$
 Now taking the cohomology, we have that $H^1(T_Y) = 0$ since $r \geq 3$. Now we show that $H^1(T_Y \otimes L_{\chi}^{-1}) = 0$ for $\chi \in G^*-\{0\}$. Tensoring the above sequence by $L_{\chi}^{-1})$, we have a right exact sequence 
 $$L_{\chi}^{-1} \to \bigoplus_{i = 1}^r\mathcal{O}_Y(a_i) \otimes L_{\chi}^{-1} \to T_Y \otimes L_{\chi}^{-1} \to 0$$
 Now if we denote the kernel and the image of the map $L_{\chi}^{-1} \to \bigoplus_{i = 1}^r\mathcal{O}_Y(a_i)$ by $\mathcal{K}$ and $\mathcal{I}$ respectively, we have two exact sequences
 $$0 \to \mathcal{I} \to \bigoplus_{i = 1}^r\mathcal{O}_Y(a_i) \otimes L_{\chi}^{-1} \to T_Y \otimes L_{\chi}^{-1} \to 0$$

and 

$$0 \to \mathcal{K} \to L_{\chi}^{-1} \to \mathcal{I} \to 0$$

Now note that since $L_{\chi}^{-1}$ is a line bundle outside the finitely many singular points of $Y$, we have that $\mathcal{K}$ is supported on finitely many points. Now since $H^2(L_{\chi}^{-1}) = H^3(\mathcal{K}) = 0$, we conclude that $H^2(\mathcal{I}) = 0$. Further since $H^1(\mathcal{O}_Y(a_i) \otimes L_{\chi}^{-1}) = 0$, we conclude that $H^1(T_Y \otimes L_{\chi}^{-1}) = 0$.
 
 Since $D_g$'s are assumed to be quasismooth, they are smooth when restricted to the smooth locus of $\mathbb{P}(a_0,\cdots,a_r)$ and have isolated quotient singularities and are hence normal. Since $\mathbb{P}(a_0,\cdots,a_r)$ has isolated singularities, they appear only at the vertices and the chart $(x_i \neq 0)$ gives an affine chart around the $i$-th vertex which is of the form $\mathbb{A}^{r}/\mathbb{Z}_{a_i}$. Since the expression for the weighted projective space is taken to be well-formed, this gives the quotient representation required in condition $(5)$ of \Cref{every def is an abelian cover}. Further, quasismoothness forces the reduced pullback of $D_g$'s to be smooth in a local chart at a point lying above a singular point of $D_g$. Condition $(6)$ follows from the fact that both $X$ and $Y$ have isolated singularities and that $r \geq 3$.
 The third condition of \Cref{every def is an abelian cover}, translates to the first condition since by \cite[Proposition $2.1$]{pardini1991abelian} and the fact that $\operatorname{Cl}(Y)$ does not have $2-$ torsion implies, 
$$l_{\chi} = \displaystyle\frac{1}{2} \Sigma_{\chi \cdot g' = 1} d_{g'}$$
 Now by \cite[Prop 4.2]{pardini1991abelian}, the canonical divisor of the cover $X$ is given by 
 \begin{align*}
    K_X = f^* 
       \left( 
    K_Y + \sum_{H_i, \Psi_i} \frac{(|H|-1)}{|H|} D_{H_i, \Psi_i}
       \right)
   \end{align*}
Therefore, by condition $(3)$ of this Corollary, $K_X$ is ample. Then, $H^0(T_X) = H^{\operatorname{dim}(X)}(\Omega_X^{**} \otimes K_X) = 0$ by Nakano vanishing theorem.
\end{proof}

Before we discuss our applications to moduli theory, let's recall that a threefold $X$ is of general type if $K_X$ is $\mathbb{Q}$-Cartier and big; that is, the linear system $|mK_X|$ defines a birational map for $m \gg 0$. If $X$ is a reduced, projective variety over a field $k$, we say $X$ is a stable variety if it has semi-log-canonical singularities, $K_X$ is $\mathbb{Q}$-Cartier, and $K_X$ is ample (see \cite[Definition 1.41]{kollar2017families}). We construct stable varieties via abelian covers by using the following well-known result (for instance, see \cite[Thm 6.1.5]{alexeev2015moduli}).

\begin{lemma}\label{cover is stable}
Suppose that $f: X \to Y$ is a finite cover of $S_2$ varieties with double crossings in codimension 1, $B_X$ and $B_Y$ are $\mathbb{Q}$ divisors on $X$,
 $Y$, and suppose also that, for the canonical divisors, the following formula holds:
 \[
 K_X + B_X = f^*\left( K_Y + B_Y\right)
 \]
 Then, the pair $(X,B_X)$ is stable if and only if $(Y, B_Y)$ is stable.
\end{lemma}
By the work of Koll\'ar, Shepherd-Barron,  and Alexeev, there exists a projective scheme $\overline{M}_c$ that parametrizes stable threefolds with $K^3=c$, see \cite{kollar2023families}. In our next application, we construct new moduli components of stable threefolds such that the general member is a $\mathbb{Z}_2^s$ cover of a weighted projective threefold.

\begin{example}\label{new components}

Let $M > 2$ be an even positive integer $f: X \to \mathbb{P}(1,1,1,M)$ be a $\mathbb{Z}_2^4$ cover with branch divisors $D_{g_0} \in |\mathcal{O}_{\mathbb{P}}(2)|$ and $D_{g} \in |\mathcal{O}_{\mathbb{P}}(M)|$ for $g \neq g_0, g \neq 0$. Then one can check that $f_*(\mathcal{O}_X) = \oplus_{\{\chi | \chi(g_0) = 1\}} \mathcal{O}_{\mathbb{P}}(-(1+\frac{7}{2}M)) \oplus_{\{\chi | \chi(g_0) = 0\}} \mathcal{O}_{\mathbb{P}}(-4M)$. Since the pair $(\mathbb{P}(1,1,1,M), D_{g_0})$ is a stable pair, we have that $X$ is stable by \Cref{cover is stable}. Further, $(1+\frac{7}{2}M)$ is not divisible by $M$ and hence $f$ is a non-flat cover. This example satisfies all the conditions of \Cref{every def is an abelian cover over a WPS} and hence defines a new component of the moduli of stable threefolds.
    
\end{example}

\begin{corollary}\label{hyperplane arrangements}
    Let $Y = \mathbb{P}(a_0,...,a_r)$ with $\textrm{gcd}(a_i,a_j) = 1$ for $0 \leq i < j \leq r$ and $r \geq 3$. Now consider $G = \mathbb{Z}_2^s$ covers branched along either one of the following configurations of hyperplane arrangements.
    \begin{enumerate}
        \item $s \geq 3$, $d_g = 1$ for $g \in G-\{0\}$ and $\Sigma_{j = 0}^r a_j < \displaystyle\frac{2^s-1}{2}$.
        \item $s \geq 4$, fix a linear subspace $L \subset G$, with $\operatorname{dim}(L) \geq 2$, where we look at $G$ as a $\mathbb{Z}_2$ vector space. Then $d_g = 1$ if and only if $g \in G\setminus L$. In this case assume
        $\Sigma_{j = 0}^r a_j < \displaystyle\frac{2^s-2^{\operatorname{dim}(L)}}{2}$ 
    \end{enumerate}
Then for a generic choice of branch divisors $\{D_g \in |\mathcal{O}_Y(d_g)|\}_{g \in G-\{0\}}$, such that $D_g$ is quasismooth for each $g \neq 0$, there exists a $G = \mathbb{Z}_2^s$ cover $f: X \to Y$, with $K_X$ ample, $X$ unobstructed and all deformations of $X$ are once again $\mathbb{Z}_2^s$ covers of $Y$ branched along the same configuration of hyperplane arrangements.
\end{corollary}

We break the proof of the Corollary into several parts.
Let \( V = (\mathbb{Z}/2\mathbb{Z})^s \). We consider functions \( h: V \to \{0,1\} \) with \( h(0) = 0 \) that satisfy:
\begin{enumerate}
    \item \textbf{Evenness Condition}: For every \( g \in V \),
        \[
        L_g = \sum_{\substack{g' \in V \\ g \cdot g' = 1}} h(g') \quad \text{is even}.
        \]
    This condition ensures that if we choose a collection of divisors $\{D_g\}_{g \in G^*-\{0\}}$ such that $D_g \in \lvert \mathcal{O}_{Y}(d_g) \rvert$, where $d_g = h(g)$, then the collection $\{D_g\}_{g \in G-\{0\}}$ satisfies the fundamental relations and consequently, there exists a $\mathbb{Z}_2^s$ cover $f: X \to Y$ branched along $\Sigma_g D_g$, with the inertia group of $D_g$ being generated by the cyclic subgroup generated by $g$. 
    \item \textbf{Deformation Inequality}: For all nonzero \( g, g' \in V \) with \( g \cdot g' = 0 \),
        \[
        h(g) < \frac{1}{2} L_{g'}.
        \]
    These set of inequalities ensure by \Cref{every def is an abelian cover over a WPS}, that a general deformation of the abelian cover $f: X \to Y$ constructed by a collection of divisors $\{D_g \in |\mathcal{O}_Y(d_g)|\}$ is once again an abelian cover of $Y$.  
\end{enumerate}

\begin{lemma}\label{evenness condition}
    The evenness condition is equivalent to the condition that
\[
\sum_{a \in A} a = 0 \quad \text{in } \mathbb{F}_2^s,
\]
where \( A = \{ g' \in V \setminus \{0\} : h(g') = 1 \} \).
\end{lemma}

\begin{proof}
Note that \( L_g = \sum_{\substack{g' \in V \\ g \cdot g' = 1}} h(g') \). The condition that \( L_g \) is even is equivalent to \( L_g \equiv 0 \pmod{2} \). Now observe:
\[
L_g = \sum_{g' \in V} (g \cdot g') h(g') \pmod{2},
\]
because if \( g \cdot g' = 0 \), then the term contributes 0 modulo 2, and if \( g \cdot g' = 1 \), then it contributes \( h(g') \) modulo 2. So we have:
\[
L_g \equiv \sum_{g' \in V} (g \cdot g') h(g') \pmod{2}.
\]
Define the vector
\[
v = \sum_{g' \in V} h(g') g' \in \mathbb{F}_2^s.
\]
Then
\[
\sum_{g' \in V} (g \cdot g') h(g') = g \cdot v.
\]
Thus, the evenness condition becomes:
\[
g \cdot v \equiv 0 \pmod{2} \quad \text{for all } g \in V.
\]
Since the dot product is non-degenerate over \( \mathbb{F}_2^s \), this implies \( v = 0 \). Hence,
\[
\sum_{g' \in V} h(g') g' = 0.
\]
Given that \( h(0) = 0 \), we have:
\[
\sum_{g' \in V \setminus \{0\}} h(g') g' = 0,
\]
which is exactly
\[
\sum_{a \in A} a = 0.
\]
This completes the proof of the lemma.
\end{proof}

\begin{lemma}\label{intersection with affine hyperplanes}
Let $V = (\mathbb{Z}/2\mathbb{Z})^s$, and let $h: V \to \{0,1\}$ with $h(0) = 0$.
Define 
$A = \{g \in V \setminus \{0\} : h(g) = 1\}$. Suppose that for all affine hyperplanes $H$ of $V$ with $0 \notin H$, we have:
    \[
    |A \cap H| \geq 4.
    \]
Then for all nonzero $g, g' \in V$ with $g \cdot g' = 0$, we have:
    \[
    h(g) < \frac{1}{2} L_{g'}
    \]
    
\end{lemma}

\begin{proof}
Let $g, g'$ be nonzero vectors with $g \cdot g' = 0$. Let $H_{g'} = \{k \in V : g' \cdot k = 1\}$.
Note that since $h(g) \in \{0,1\}$, we have $L_{g'} = |A \cap H_{g'}| \geq 4$ by assumption. Hence
\[
 h(g) \leq 1 < \frac{1}{2} L_{g'} 
\]


\end{proof}

\begin{proof}
 \noindent\textit{Proof of \Cref{hyperplane arrangements}.}
 First note that by \Cref{evenness condition} and \Cref{intersection with affine hyperplanes}, it is enough to prove the evenness condition and that for all affine hyperplanes $H$ of $V$ with $0 \notin H$,
    \[
    |A \cap H| \geq 4.
    \] 
    for cases $(1)$ and $(2)$. Let us first prove for case $(1)$. \\
 Define \( h(g) = 1 \) for all \( g \neq 0 \), and \( h(0) = 0 \). Let \( A = \{ g \in V \setminus \{0\} : h(g) = 1 \} = V \setminus \{0\} \). Then
\[
\sum_{a \in A} a = \sum_{v \in V} v - 0.
\]
In \( \mathbb{F}_2^s \), there are \( 2^s \) vectors. For each coordinate \( i \), exactly half of the vectors (i.e., \( 2^{s-1} \)) have a 1 in that coordinate. Since \( 2^{s-1} \) is even for \( s \geq 2 \), the sum modulo 2 in each coordinate is 0. Thus,
\[
\sum_{a \in A} a = 0,
\]
so the evenness condition holds. \\
Let \( g, g' \in V \setminus \{0\} \) with \( g \cdot g' = 0 \). Then
\[
L_{g'} = \sum_{\substack{v \in V \\ g' \cdot v = 1}} h(v) = |\{ v \in V \setminus \{0\} : g' \cdot v = 1 \}|.
\]
The set \( \{ v \in V : g' \cdot v = 1 \} \) has size \( 2^{s-1} \). Therefore, we have \( L_{g'} = 2^{s-1} \).

Then
\[
\frac{1}{2} L_{g'} = 2^{s-2}.
\]
Since \( s \geq 3 \), we have \( 2^{s-2} \geq 2 \). As \( h(g) = 1 \), the inequality
\[
1 < 2^{s-2}
\]
holds. Hence, the deformation inequality is satisfied. \\

Now we prove both the conditions for configuration $(2)$.
Then $|L| = 2^k$, $|A| = 2^s - 2^k$.
We have:
\[
\sum_{a \in A} a = \sum_{v \in V} v - \sum_{l \in L} l.
\]
Since $\sum_{v \in V} v = 0$ and $\sum_{l \in L} l = 0$ (for $k \geq 2$), 
we get $\sum_{a \in A} a = 0$. Thus the first condition is verified. \\
Now, let $H$ be an affine hyperplane with $0 \notin H$. Then $H = \{x \in V : u \cdot x = 1\}$ for some $u \neq 0$. We compute $|A \cap H| = |H| - |L \cap H| = 2^{s-1} - |L \cap H|$.

There are two subcases:

\subsubsection*{Subcase 2.1: $u \in L^\perp$}
Then for all $l \in L$, $u \cdot l = 0$, so $L \subseteq H^0$ (the linear hyperplane $\{x : u \cdot x = 0\}$).
Thus $L \cap H = \emptyset$, so $|A \cap H| = 2^{s-1} \geq 8 \geq 4$.

\subsubsection*{Subcase 2.2: $u \notin L^\perp$}
Then $\exists l \in L$ such that $u \cdot l = 1$ or $l \in L \cap H$. Further, the restriction $u|_L : L \to \mathbb{F}_2$ is a nonzero linear functional.
Its kernel has dimension $k-1$, so $|L \cap H^0| = 2^{k-1}$.
Since $H = H^0 + l$, and $L$ is a subspace,
we have $|L \cap H| = 2^{k-1}$. Thus $|A \cap H| = 2^{s-1} - 2^{k-1}$.
We need to show $2^{s-1} - 2^{k-1} \geq 4$ for $s \geq 4$ and $2 \leq k \leq s-1$. Since $k \leq s-1$, we have $2^{k-1} \leq 2^{s-2}$, so:
\[
2^{s-1} - 2^{k-1} \geq 2^{s-1} - 2^{s-2} = 2^{s-2}.
\]
For $s \geq 4$, $2^{s-2} \geq 4$. Equality occurs when $s = 4$ and $k = s-1 = 3$. Thus in all cases, $|A \cap H| \geq 4$.
\end{proof}

\begin{example}
For $s = 4$, we get the following $7$ irreducible families of threefolds with ample canonical bundle, which are general in their moduli. A general $X$ in this family is a $\mathbb{Z}_2^4$ cover of a weighted projective threefold $\mathbb{WP}^3$, where
    \begin{enumerate}
        \item $\mathbb{WP}^3$ is either $\mathbb{P}(1,1,1,1)$ or $ \mathbb{P}(1,1,1,2)$ and the divisorial part of the branch locus consists of $12$ hyperplanes, each corresponding to a non-zero element of $\mathbb{Z}_2^4 \setminus L$, where $L$ is a subspace of dimension $2$. Different choices of subspaces give the same component. 
        \begin{enumerate}
            \item The cover $f: X \to \mathbb{P}(1,1,1,1)$ coincides with the canonical map of $X$.
            \item The cover $f: X \to \mathbb{P}(1,1,1,2)$ coincides with the bicanonical map of $X$.
        \end{enumerate}
        When $\mathbb{WP}^3 = \mathbb{P}(1,1,1,2)$, these are examples of non-flat pluricanonical covers since the pushforward of the structure sheaf decomposes as 
        
        \[f_*\mathcal{O}_X = \mathcal{O}_{\mathbb{P}} \oplus\left(\bigoplus_{\chi \in L^{\perp} \setminus \{1\}} \mathcal{O}_{\mathbb{P}}(-4)\right) \oplus \left( \bigoplus_{\chi \in G^* \setminus L^{\perp}} \mathcal{O}_{\mathbb{P}}(-3) \right)\]
        
        \item $\mathbb{WP}$ is either $\mathbb{P}(1,1,1,1), \mathbb{P}(1,1,1,2), \mathbb{P}(1,1,1,3), \mathbb{P}(1,1,1,4), \mathbb{P}(1,1,2,3)$ and the divisorial part of the branch locus consists of $15$ hyperplanes, each corresponding to a non-zero element of $\mathbb{Z}_2^4$.
        When $\mathbb{WP}^3 = \mathbb{P}(1,1,1,3)$ or $\mathbb{P}(1,1,2,3)$, the above once again are examples of non-flat covers since the pushforward of the structure sheaf decomposes as 
        \[
        f_*\mathcal{O}_X = \mathcal{O}_{\mathbb{P}} \oplus\left(\bigoplus_{\chi \in G^* \setminus \{1\}} \mathcal{O}_{\mathbb{P}}(-4)\right)
        \]
    \end{enumerate}

\end{example}

\section{Pluricanonical $\mathbb{Z}_2^s$ covers of weighted projective spaces}\label{sec:pluricanonical}

Let $X$ be a normal variety. If $\mathcal{O}_X(mK_X)$ is an ample and base point line bundle then we call the morphism given by the complete linear series of $\mathcal{O}_X(mK_X)$ the $m$-canonical morphism of $X$. The main theme of this section is the proof of Theorem \ref{thm:flat_pluricanonical_covers} which classifies $m$-canonical $\mathbb{Z}_2^s$ covers of weighted projective threefolds. If $f: X \to Y$ is a $\mathbb{Z}_2^s$ cover then, we say that $f$ is an $m$-canonical cover if $f$ fits into a diagram as follows for some $m \geq 1$.
\[
\begin{tikzcd}
    X \arrow[d, "f"] \arrow[dr, "\varphi_{|mK_X|}"] & \\
    Y \arrow[r, hook] & \mathbb{P}^N
\end{tikzcd}
\]
The strategy is to start with a numerical criterion for a $Z_2^s$-cover to be the $m$-canonical morphism and then, to classify the branch data that satisfy such a criterion.

\begin{proposition}\label{prop:pluricanonical_maps_of_threefolds}
Let $f: X \to Y$ be a $\mathbb{Z}_2^s$ cover where $Y = \mathbb{P}(a_0,...,a_r)$. Let $f_*\mathcal{O}_X = \mathcal{O}_Y \oplus \mathcal{E}$, where $\mathcal{E} = \oplus_{\chi \in G^*\setminus\{1\}} L_{\chi}^{-1}$, where $L_{\chi}^{-1}$ is a reflexive sheaf of rank one for all $\chi \in G^*$. Further let the divisors associated to the data of the abelian cover be $\{D_g \in |\mathcal{O}_{\mathbb{P}}(d_g)|\}_{g \in G-\{0\}}$. Let $m \geq 1$ be an integer. Then $\mathcal{O}_X(mK_X)$ is an ample and base point free line bundle and $f$ is the $m-$ canonical morphism of $X$ if and only if
\begin{enumerate}
    \item $\forall \chi \in G^*-\{0\}$, $$h^0(\mathcal{O}_Y(\displaystyle\frac{1}{2}m \Sigma_{g \in G-\{0\}} d_g - m \Sigma_j a_j -\displaystyle\frac{1}{2} \Sigma_{g
     \mid \chi \cdot g = 1} d_g)) = 0$$ 
     This happens in particular if $\displaystyle\frac{1}{2}m \Sigma_{g \in G-\{0\}} d_g - m \Sigma_j a_j < \displaystyle\frac{1}{2} \Sigma_{g
     \mid \chi \cdot g = 1} d_g $ for all $\chi \in G^*-\{0\}$
    \item $\displaystyle\frac{1}{2}m \Sigma_{g \in G-\{0\}} d_g - m \Sigma_j a_j$ is a positive integer multiple of $\operatorname{lcm}(a_0,...,a_r)$.
\end{enumerate}
Further in this case, if $d = \displaystyle\frac{1}{2}m \Sigma_{g \in G-\{0\}} d_g - m \Sigma_j a_j$, then $h^0(mK_X) = p_m(X) = h^0(\mathcal{O}_{\mathbb{P}}(d))-1$.

\end{proposition}

\begin{proof}
 $f: X \to Y$ is the $m-$ canonical morphism if and only if the morphism $\varphi_{|mK_X|}$ induced by the complete linear series $|mK_X|$ for some $m \geq 1$ factors as follows.
\[
\begin{tikzcd}
    X \arrow[d, "f"] \arrow[dr, "\varphi_{|mK_X|}"] & \\
    Y \arrow[r, "\varphi_{|\mathcal{O}_{Y}(d)|}", hook] & \mathbb{P}^N
\end{tikzcd}
\]

Then we have the following two constraints which are in fact sufficient conditions as well :
\begin{enumerate}
    \item[(i)] $f^*(\mathcal{O}_{Y}(d)) = mK_X$
    \item[(ii)] $h^0(\mathcal{O}_{Y}(d)) = N+1 = h^0(mK_X)$
\end{enumerate}

We first analyze (i). This gives
\begin{align*}
    f^*(\mathcal{O}_Y(d)) = f^* 
       \left( 
    mK_Y + m\sum_{H_i, \Psi_i} \frac{(|H|-1)}{|H|} D_{H_i, \Psi_i}
       \right)
   \end{align*}
Since $\operatorname{Pic}(Y) = \mathbb{Z}\mathcal{O}_Y(\operatorname{lcm}(a_0,...,ar))$ (see \cite[Theorem $3.2.4$]{dolgachevwps82}), there are no degree zero line bundles on $Y$. Therefore, the above is equivalent to having 
$$d = \displaystyle\frac{1}{2}m \Sigma_{g \in G-\{0\}} d_g - m \Sigma_j a_j$$
Since we are assuming $\mathcal{O}_Y(d)$ to be a line bundle, we have the second statement that $\displaystyle\frac{1}{2}m \Sigma_{g \in G-\{0\}} d_g - m \Sigma_j a_j$ should be a positive integer multiple of $\operatorname{lcm}(a_0,...,a_r)$.
Now condition (ii) implies that $h^0(f^*(\mathcal{O}_Y(d))) = h^0(\mathcal{O}_Y(d))$. By pushing forward, this is equivalent to $h^0(\mathcal{O}_Y(d) \otimes L_{\chi}^{-1}) = 0$ $ \forall \chi \in G^*-\{0\}$. By the fundamental relations, this implies that 
$$d < \displaystyle\frac{1}{2} \Sigma_{g
     \mid \chi(g) = 1} d_g $$
and this implies the first condition.
\end{proof}

\subsection{Fourier Transform of functions on $\mathbb{Z}_2^s$}

In this section, we prove some results on the Fourier transform of functions on $\mathbb{Z}_2^s$ which are crucially used for the classification of flat pluricanonical covers of weighted projective threefolds.

Let \( G = (\mathbb{Z}_2)^s \) with additive notation.  
Let \( G^* \) be the dual group of \( G \).  
Fix an isomorphism \( G \to G^* \) given by \( g \mapsto \chi_g \), where \( \chi_g(x) = (-1)^{g \cdot x} \), with \( g \cdot x = \sum_{i=1}^s g_i x_i \pmod{2} \).  
Via this isomorphism, we identify \( G^* \) with \( G \), so that a character \( \chi \in G^* \) corresponds to a unique \( g \in G \), and we write \( \chi \cdot x = g \cdot x \in \mathbb{Z}_2 \).

\begin{lemma}\label{lemma: Fourier}

Let \( d: G \to \mathbb{Z}_{\ge 0} \) with \( d(0) = 0 \), and let \( l: G^* \setminus \{0\} \to \mathbb{Z}_{\ge 0} \). Define the Fourier transform of \( d \) as:
\[
\widehat{d}(\chi) = \sum_{x \in G} d(x) \, (-1)^{\chi \cdot x}.
\]

Suppose for all nontrivial \( \chi \neq 0 \):
\begin{equation}\label{l_chi}
\sum_{\substack{x \in G \\ \chi \cdot x = 1}} d(x) = 2 l(\chi). 
\end{equation}

Then for all \( \chi \neq 0 \):
\[
\widehat{d}(0) - \widehat{d}(\chi) = 4 l(\chi),
\]
where \( \widehat{d}(0) = \sum_{x \in G} d(x) \), and \( 0 \) denotes the trivial character.

\end{lemma}

\begin{proof}

Let \( \chi \in G^* \setminus \{0\} \).  
The set \( \{ x \in G : \chi \cdot x = 1 \} \) has size \( |G|/2 = 2^{s-1} \), since \( \chi \cdot (\cdot) \) is a nonzero linear functional.

Split the sum for \( \widehat{d}(\chi) \):
\[
\widehat{d}(\chi) = \sum_{\chi \cdot x = 0} d(x) \cdot 1 + \sum_{\chi \cdot x = 1} d(x) \cdot (-1).
\]

Let \( S_0(\chi) = \sum_{\chi \cdot x = 0} d(x) \), \( S_1(\chi) = \sum_{\chi \cdot x = 1} d(x) \).  
Then:
\[
\widehat{d}(\chi) = S_0(\chi) - S_1(\chi).
\]

Now, \( \widehat{d}(0) = \sum_{x \in G} d(x) \, (-1)^{0 \cdot x} \).  
But \( 0 \cdot x = 0 \) for all \( x \in G \), so \( (-1)^{0 \cdot x} = 1 \).  
Hence:
\[
\widehat{d}(0) = \sum_{x \in G} d(x) = S_0(\chi) + S_1(\chi),
\]
since \( \{x : \chi \cdot x = 0\} \) and \( \{x : \chi \cdot x = 1\} \) partition \( G \).

From \Cref{l_chi}, \( S_1(\chi) = 2 l(\chi) \).

Thus:
\[
\widehat{d}(0) - \widehat{d}(\chi) = [S_0 + S_1] - [S_0 - S_1] = 2 S_1(\chi) = 4 l(\chi).
\]

This holds for all nontrivial \( \chi \), completing the proof.

\end{proof}

\begin{lemma} \label{lemma: Inverse Fourier Transform}

Let \( G = (\mathbb{Z}_2)^s \) with additive notation, and let \( G^* \) be its dual group, identified with \( G \) via \( \chi_g(x) = (-1)^{g \cdot x} \).  
Let \( d: G \to \mathbb{C} \) be any function, and define its Fourier transform \( \widehat{d}: G^* \to \mathbb{C} \) by
\[
\widehat{d}(\chi) = \sum_{x \in G} d(x) \, (-1)^{\chi \cdot x}.
\]

Then \( d \) can be recovered from \( \widehat{d} \) via the inverse formula:
\[
d(x) = \frac{1}{|G|} \sum_{\chi \in G^*} \widehat{d}(\chi) \, (-1)^{\chi \cdot x}.
\]

\end{lemma}

\begin{proof}

We substitute the definition of \( \widehat{d}(\chi) \) into the right-hand side:
\[
\frac{1}{|G|} \sum_{\chi \in G^*} \widehat{d}(\chi) \, (-1)^{\chi \cdot x}
= \frac{1}{|G|} \sum_{\chi \in G^*} \left[ \sum_{y \in G} d(y) \, (-1)^{\chi \cdot y} \right] (-1)^{\chi \cdot x}.
\]

Interchanging the order of summation:
\[
= \frac{1}{|G|} \sum_{y \in G} d(y) \sum_{\chi \in G^*} (-1)^{\chi \cdot (y + x)}.
\]

Now, the inner sum \( \sum_{\chi \in G^*} (-1)^{\chi \cdot z} \) (where \( z = y + x \)) is a standard character sum:

\begin{enumerate}
    \item If \( z = 0 \), then \( (-1)^{\chi \cdot 0} = 1 \) for all \( \chi \), so the sum equals \( |G^*| = |G| \).
    \item If \( z \neq 0 \), then \( \chi \mapsto (-1)^{\chi \cdot z} \) is a nontrivial character of \( G^* \), so the sum over \( \chi \in G^* \) is 0.
\end{enumerate}

Thus:
\[
\sum_{\chi \in G^*} (-1)^{\chi \cdot z} = 
\begin{cases}
|G| & \text{if } z = 0, \\
0 & \text{if } z \neq 0.
\end{cases}
\]

Therefore, the only \( y \) contributing to the sum is \( y = x \), giving:
\[
\frac{1}{|G|} \cdot d(x) \cdot |G| = d(x).
\]

This completes the proof.

\end{proof}

\begin{lemma}[First moment constraint for the unnormalized Fourier transform]\label{lem:first-moment}
Let $G=(\mathbb Z_2)^s$ written additively, and identify the dual group $G^*$ with $G$
via $\chi_g(x)=(-1)^{g\cdot x}$, where $g\cdot x\in\mathbb Z_2$ is the standard dot product.
For a function $f:G\to\mathbb C$, define its (unnormalized) Fourier transform by
\[
\widehat f(\chi)\;=\;\sum_{x\in G} f(x)\,(-1)^{\chi\cdot x},\qquad \chi\in G^*.
\]
Then
\[
\sum_{\chi\in G^*}\widehat f(\chi)\;=\;|G|\cdot f(0).
\]
In particular, if $f(0)=0$ then $\sum_{\chi\in G^*}\widehat f(\chi)=0$.
\end{lemma}

\begin{proof}
We expand and interchange the order of summation:
\[
\sum_{\chi\in G^*}\widehat f(\chi)
=\sum_{\chi\in G^*}\sum_{x\in G} f(x)\,(-1)^{\chi\cdot x}
=\sum_{x\in G} f(x)\left(\sum_{\chi\in G^*}(-1)^{\chi\cdot x}\right).
\]
We use the standard orthogonality of characters:
\[
\sum_{\chi\in G^*}(-1)^{\chi\cdot x}
=
\begin{cases}
|G|,& x=0,\\
0,& x\neq 0,
\end{cases}
\]
since for $x\neq 0$ the map $\chi\mapsto (-1)^{\chi\cdot x}$ is a nontrivial character of
$G^*$ and its sum over the whole group is $0$. Hence only the term $x=0$ contributes, and we
obtain
\[
\sum_{\chi\in G^*}\widehat f(\chi)=f(0)\cdot |G|.
\]
This proves the claim.
\end{proof}

\begin{lemma}[Parseval / inner product identity for the unnormalized Fourier transform]\label{lemma: Parseval}
Let $G=(\mathbb Z_2)^s$ and identify $G^*\cong G$ via
\[
\chi_g(x)=(-1)^{g\cdot x},\qquad g\cdot x\in \mathbb Z_2.
\]
For $f,g:G\to \mathbb C$, define the \emph{unnormalized} Fourier transform
\[
\widehat f(\chi)=\sum_{x\in G} f(x)\,(-1)^{\chi\cdot x},\qquad \chi\in G^*.
\]
Then one has the Parseval identity
\[
\sum_{\chi\in G^*}\widehat f(\chi)\,\overline{\widehat g(\chi)}
\;=\;
|G|\sum_{x\in G} f(x)\,\overline{g(x)}.
\]
\end{lemma}

\begin{proof}
We expand and rearrange:
\begin{align*}
\sum_{\chi\in G^*}\widehat f(\chi)\,\overline{\widehat g(\chi)}
&=
\sum_{\chi\in G^*}\left(\sum_{x\in G} f(x)(-1)^{\chi\cdot x}\right)
\overline{\left(\sum_{y\in G} g(y)(-1)^{\chi\cdot y}\right)}\\
&=
\sum_{\chi\in G^*}\left(\sum_{x\in G} f(x)(-1)^{\chi\cdot x}\right)
\left(\sum_{y\in G} \overline{g(y)}\,\overline{(-1)^{\chi\cdot y}}\right).
\end{align*}
Since $(-1)^{\chi\cdot y}\in\{\pm 1\}\subset \mathbb R$, we have
$\overline{(-1)^{\chi\cdot y}}=(-1)^{\chi\cdot y}$, hence
\begin{align*}
\sum_{\chi\in G^*}\widehat f(\chi)\,\overline{\widehat g(\chi)}
&=
\sum_{\chi\in G^*}\sum_{x\in G}\sum_{y\in G}
f(x)\,\overline{g(y)}\,(-1)^{\chi\cdot x}\,(-1)^{\chi\cdot y}\\
&=
\sum_{x\in G}\sum_{y\in G}
f(x)\,\overline{g(y)}
\sum_{\chi\in G^*}(-1)^{\chi\cdot(x+y)}.
\end{align*}
We now use the standard character-sum orthogonality:
\[
\sum_{\chi\in G^*}(-1)^{\chi\cdot z}
=
\begin{cases}
|G| & \text{if } z=0,\\
0   & \text{if } z\neq 0.
\end{cases}
\]
Applying this with $z=x+y$ (note $x+y=0$ iff $x=y$), the inner sum equals $|G|$
when $x=y$ and vanishes otherwise. Therefore the double sum collapses to
\[
\sum_{\chi\in G^*}\widehat f(\chi)\,\overline{\widehat g(\chi)}
=
|G|\sum_{x\in G} f(x)\,\overline{g(x)},
\]
as claimed.
\end{proof}

\begin{lemma}\label{Plancherel for the unnormalized Fourier transform}
With notation as above, for any $f:G\to\mathbb C$ one has
\[
\sum_{\chi\in G^*}\bigl|\widehat f(\chi)\bigr|^2
\;=\;
|G|\sum_{x\in G}|f(x)|^2.
\]
\end{lemma}

\begin{proof}
Apply Parseval with $g=f$. Then $\overline{\widehat f(\chi)}$ is the complex
conjugate of $\widehat f(\chi)$, so the left-hand side becomes
$\sum_{\chi} \widehat f(\chi)\overline{\widehat f(\chi)}
=\sum_{\chi}|\widehat f(\chi)|^2$.
Similarly, the right-hand side becomes $|G|\sum_x f(x)\overline{f(x)}
=|G|\sum_x |f(x)|^2$. This is exactly the stated identity.
\end{proof}

\begin{lemma}[Fourier transform of convolution and the cubic moment identity]\label{lem:fourier-convolution-cubic}
Let $G=(\mathbb Z_2)^s$ written additively, and identify $G^*$ with $G$ via
$\chi(x)=(-1)^{\chi\cdot x}$, where $\chi\cdot x\in\mathbb Z_2$.
For a function $f:G\to\mathbb C$, define the (unnormalized) Fourier transform
\[
\widehat f(\chi)=\sum_{x\in G} f(x)\,(-1)^{\chi\cdot x},\qquad \chi\in G^*,
\]
and define convolution by
\[
(f*g)(z)=\sum_{\substack{x,y\in G\\ x+y=z}} f(x)g(y),\qquad z\in G.
\]
Then:
\begin{enumerate}
\item \textbf{(Convolution becomes pointwise product.)} For all $\chi\in G^*$,
\[
\widehat{f*g}(\chi)=\widehat f(\chi)\,\widehat g(\chi).
\]
\item \textbf{(Cubic moment at the origin.)} For any $d:G\to\mathbb C$,
\[
(d*d*d)(0)=\frac{1}{|G|}\sum_{\chi\in G^*}\widehat d(\chi)^3.
\]
In particular, if $d=\mathbf{1}_S$ is the indicator function of a subset $S\subseteq G$,
then
\[
(d*d*d)(0)=\#\{(x,y,z)\in S^3:\ x+y+z=0\}\ \in\ \mathbb Z_{\ge 0}.
\]
\end{enumerate}
\end{lemma}

\begin{proof}
\begin{enumerate}
\item Fix $\chi\in G^*$. By definition and expanding the convolution,
\[
\widehat{f*g}(\chi)
=\sum_{z\in G} (f*g)(z)\,(-1)^{\chi\cdot z}
=\sum_{z\in G}\left(\sum_{x+y=z} f(x)g(y)\right)(-1)^{\chi\cdot z}.
\]
Reindexing by $z=x+y$ gives
\[
\widehat{f*g}(\chi)=\sum_{x,y\in G} f(x)g(y)\,(-1)^{\chi\cdot(x+y)}.
\]
Since $(-1)^{\chi\cdot(x+y)}=(-1)^{\chi\cdot x}(-1)^{\chi\cdot y}$, the sum factors:
\[
\widehat{f*g}(\chi)
=\left(\sum_{x\in G} f(x)(-1)^{\chi\cdot x}\right)
 \left(\sum_{y\in G} g(y)(-1)^{\chi\cdot y}\right)
=\widehat f(\chi)\,\widehat g(\chi),
\]
as required.

\item Apply (1) twice to obtain $\widehat{d*d*d}(\chi)=\widehat d(\chi)^3$ for all $\chi$.
Now use the inverse Fourier transform (with this unnormalized convention):
\[
h(0)=\frac{1}{|G|}\sum_{\chi\in G^*}\widehat h(\chi)\,(-1)^{\chi\cdot 0}
=\frac{1}{|G|}\sum_{\chi\in G^*}\widehat h(\chi),
\]
and take $h=d*d*d$. This yields
\[
(d*d*d)(0)=\frac{1}{|G|}\sum_{\chi\in G^*}\widehat{d*d*d}(\chi)
=\frac{1}{|G|}\sum_{\chi\in G^*}\widehat d(\chi)^3.
\]
Finally, if $d=\mathbf{1}_S$, then by the definition of convolution,
\[
(d*d*d)(0)=\sum_{x+y+z=0} d(x)d(y)d(z)
=\#\{(x,y,z)\in S^3:\ x+y+z=0\},
\]
which is a nonnegative integer.
\end{enumerate}
\end{proof}

\subsection{Lemmas related to the pluricanonical map}
Here, we prove some combinatorial lemmas required for the classification of pluricanonical covers of weighted projective threefolds.  Fix four positive integers $a_0,a_1,a_2,a_3$ satisfying
$\gcd(a_i, a_j, a_k) = 1$ for every triple of distinct indices $i,j,k$. Define:
\[
W := a_0 + a_1 + a_2 + a_3,
\quad
L := \operatorname{lcm}(a_0,a_1,a_2,a_3).
\]
For $g$ and $\chi$ in $(\mathbb{Z}_2)^s \setminus 0$ define the dot product $\chi \cdot g$ as the usual dot product modulo $2$. For each $g \in 
(\mathbb{Z}_2)^s \setminus 0$, we associate an integer $d_g \ge 0$, and we have
\begin{align*}
D := \sum_{g \in (\mathbb{Z}_2)^s \setminus 0} d_g
&&
l_\chi \;=\; 
\frac{1}{2}\sum_{\substack{g \in (\mathbb{Z}_2)^s \setminus 0 \\ \chi \cdot g = 1}} d_{g}
&&
M:= \frac{m}{2}D- mW
\end{align*}
Next, we encode in a definition the conditions given by \Cref{prop:pluricanonical_maps_of_threefolds}.

\begin{definition}
The list of numbers
$(d_g,\, m,\, s) \in \mathbb{Z}^{|2^s-1|}_{\geq 0}\times \mathbb{Z}_{+}^2$ is called a \textbf{admissible pluricanonical solution} for the weights $(a_0,a_1,a_2,a_3)$ if it satisfies:
\begin{enumerate}
\item For all $\chi \in (\mathbb{Z}_2)^s \setminus 0$, $l_{\chi}$ is an integer.
\item For all $\chi \in  (\mathbb{Z}_2)^s \setminus 0$, it holds that
$M$ is a positive integer, divisible by $L$, and strictly less than $l_{\chi}$.
\end{enumerate}
The solution is called \textbf{flat} if $l_\chi$ is divisible by $L$ for 
 all $\chi$. 
\end{definition}

\begin{lemma}
\label{lemma:ineq_comb}
Let the notations be as above in this subsection, we suppose that 
there is an admissible pluricanonical solution with respect to some weights.
Then, the following holds:
\begin{itemize}
    \item[1.] If $L \geq 2$, then $3L \geq  2L + 2 \geq  W.$
    \item[2.] Let's denote $M=kL$, if there exists an admissible flat solution for given weights, then the following inequality holds 
\begin{align*}
W
\geq 
\left( (k+1)(2 -2^{1-s})- \frac{k}{m}\right)L 
&&
l_\chi \geq (k+1)L, \; \; \forall \chi
\end{align*}
\item[3.] Assume $L \geq 2$. Then any admissible flat solution with $s \geq 2$ and $m \geq 3$ satisfies $3 \geq L$ and $M=L$.
\end{itemize}
\end{lemma}
\begin{proof}
We start with \textbf{statement $1$.} Let  $b_i:=L/a_i \in\mathbb{Z}_{>0}$. Then
\[
\frac{W}{L}=\sum_{i=0}^3\frac{a_i}{L}=\sum_{i=0}^3\frac{1}{b_i}.
\]
 If $b_i \geq 2$ for four or three indices then the inequality $\sum_i \frac{1}{b_i}\le 2+\frac{2}{L}$ follows easily. Suppose that there are exactly two $i$ such that $b_i =1$.  Without loss of generality $b_0=b_1=1$, i.e. $a_0=a_1=L$.
From $\gcd(a_0,a_1,a_2)=1$ we get $\gcd(L,a_2)=1$, and since $a_2\mid L$, we must have $a_2=1$. Similarly $a_3=1$. Thus $b_2=b_3=L$, and
\[
\sum_{i=0}^3 \frac{1}{b_i}=1+1+\frac{1}{L}+\frac{1}{L}=2+\frac{2}{L}\le 3
\quad\text{because }L\ge 2.
\]
If $b_0=b_1 =b_2=1$, then $a_0=a_1=a_2=L$, so $\gcd(a_0,a_1,a_2)=L>1$, contradicting the hypothesis on the weights. Then, this case cannot occur
when $L>1$. The same argument applies when  $b_i=1$ for all $i$.

We continue with \textbf{statement 2.}
By the definition of $l_\chi$ and the condition that $l_\chi > M$ with $M=kL$ for some positive integer $k$, we obtain that 
\begin{align*}
    2^{\,s-2}\,D \;=\; \sum_{\chi}l_\chi,
    &&
    l_\chi > kL
\end{align*}
Moreover, if the solution is flat, we have 
\begin{align}
    2^{\,s-2}\,D \;=\; \sum_{\chi}l_\chi,
    &&
    l_\chi \geq (k+1)L
\label{eq: divisor sum and line bundle sum}
\end{align}
From these two last expressions, we obtain

\begin{align}
  2^{s-2}D \geq (2^s-1)(k+1)L 
   \Longrightarrow\;\;
   D \geq (4-2^{2-s})(k+1)L
\label{eq:lower_bound_D flat}    
\end{align}
Since, we have
\begin{align}
 \frac{m}{2}D-mW  = kL     \;\;\Longrightarrow\;\;
 W =  \frac{1}{2}D - \frac{k}{m}L
\label{eq:W_D_M}
\end{align}
, we obtain the lower bound

\begin{align}
 W &=  \frac{1}{2}D - \frac{k}{m}L   
 \notag
 \\ 
 & \geq
\frac{(k+1)}{2}\left( 4 -2^{2-s}\right)L - \frac{k}{m}L
\label{eq:lower_bound_W flat}
\\ & 
\geq 
\left( (k+1)(2 -2^{1-s})- \frac{k}{m}\right)L  
\notag
\end{align}
if the solution is flat.
We finish with \textbf{statement 3:} 
By statements 1 and 2, it holds that
\begin{align*}
2 + \frac{2}{L} \geq \frac{W}{L}  \geq 
 (k+1)(2 -2^{1-s})- \frac{k}{m}
\end{align*}
For $s \geq 2$, we have that $(2 -2^{1-s}) \geq \frac{3}{2}$, therefore
\begin{align*}
(k+1)(2 -2^{1-s}) \geq \frac{3}{2}(k+1)   
   \Longrightarrow\;\;
(k+1)(2 -2^{1-s}) 
- \frac{k}{m}
\geq \frac{3}{2} (k+1)- \frac{k}{m}     
\end{align*}
the condition $m \geq 3$ implies
\begin{align*}
    \frac{1}{m} \leq \frac{1}{3}
   \Longrightarrow\;\;
    - \frac{k}{m} \geq 
    - \frac{k}{3}
       \Longrightarrow\;\;
       \frac{3}{2}(k+1) - \frac{k}{m} 
       \geq 
       \frac{3}{2}(k+1) - \frac{k}{3}
\end{align*}
Then, we have that 
\begin{align*}
\frac{W}{L} \geq
 (k+1)(2 -2^{1-s}) 
- \frac{k}{m}
\geq 
\frac{7k}{6} + \frac{3}{2}
\end{align*}
\textcolor{black}{Since $\frac{W}{L} \leq 3$, we have that $k = 1$. We further have that}

\begin{align*}
2 + \frac{2}{L} \geq  \frac{W}{L}  \geq \frac{7k}{6}+ \frac{3}{2} \geq \frac{8}{3}
\end{align*}
\textcolor{black}{which implies $L \leq 3$.} \\

\end{proof}

\begin{lemma}
\label{lemma:forbidden_flat_admissible}
Let the notations be as above in this subsection, we suppose $L \geq 2$. 
Let $s \geq 2$, and $m \geq 1$, if $s$ and $m$ satisfy one of the following inequalities. 
\begin{align*}
s \geq 2, \; m\geq 4
&& 
s \geq 3, \; m \geq 3
&&
s\geq 4, \;  m\geq 2
&& 
s \geq 6, m \geq 1
\end{align*}
Then, there are no flat admissible solutions. \\
\end{lemma}

\begin{proof}
First, we consider the case $s \geq 3$ and $m \geq 3$. 
\textcolor{black}{and suppose the solution is flat.} By statement 3 in Lemma \ref{lemma:ineq_comb}, we have that $k=1$. 
By \eqref{eq:lower_bound_D flat}, $s\geq 3$ imply
\begin{align*}
\frac{D}{L} \geq 7
&&
3 \geq  \frac{W}{L} = \frac{1}{2}\frac{D}{L} - \frac{1}{m}  
\geq
 \frac{7}{2} - \frac{1}{m}
\end{align*}
We expand the last inequality to obtain $m \leq 2$ which is a contradiction with $m \geq 3$.

Next, we consider the case $s \geq 2$ and $m \geq 4$. By Lemma~\ref{lemma:ineq_comb}(3) we have
$k=1 $  and $L\le 3.$ From \eqref{eq:lower_bound_D flat} with $s=2$ we obtain
\[
\frac{D}{L}\;\ge\;\bigl(4-2^{\,2-2}\bigr)(k+1)\;=\;3\cdot 2\;=\;6.
\]
Using $M=\frac{m}{2}D-mW=kL$ with $k=1$ we have
\begin{align}
\frac{W}{L}\;=\;\frac12\,\frac{D}{L}-\frac{1}{m}.
\label{eq:wls2m4}
\end{align}
Lemma \ref{lemma:ineq_comb}.(1) and the fact that $L\ge 2$ and $m\ge 4$ imply
\[
\frac{W}{L}\le 2+\frac{2}{L}\le 3, 
\quad
\quad 
\frac12\,\frac{D}{L}-\frac{1}{m}\;\le\;2+\frac{2}{L}
\quad\Longrightarrow\quad
\frac{D}{L}\;\le\;4+\frac{4}{L}+\frac{2}{m}\;\le\;\frac{13}{2}.
\]
For $s=2$ we have $2^{\,s-2}D=\sum_{\chi}l_\chi=D$, and by flatness $L\mid l_\chi$ for all $\chi$, so $L\mid D$; therefore $\frac{D}{L}\in\mathbb Z$. 
Then, we obtain that
\[
6 \leq 
\frac{D}{L}  \le \frac{13}{2}.
\quad\Longrightarrow\quad
\frac{D}{L}=6,\qquad D=6L.
\]
We use  $l_\chi\ge 2L$  and
\( 
\sum_{\chi}l_\chi=D=6L,
\)
it follows that
\[
l_{10}=l_{01}=l_{11}=2L.
\]
Using the fundamental relations for $G=(\mathbb Z_2)^2$,
\[
2l_{10}=d_{10}+d_{11},\qquad
2l_{01}=d_{01}+d_{11},\qquad
2l_{11}=d_{10}+d_{01},
\]
we solve and get \( d_{10}=d_{01}=d_{11}=2L.\) Finally, from 
Equation \eqref{eq:wls2m4} and $D=6L$ we obtain
\[
W=\Bigl(3-\frac{1}{m}\Bigr)L.
\]
If $L=3$, then $\frac{W}{L}=3-\frac{1}{m}\ge \frac{11}{4}>2+\frac{2}{3}$ which is a contradiction. If $L=2$, then
\[
W=6-\frac{2}{m}\notin\mathbb Z\quad\text{for every }m\ge 4,
\]
which is impossible since $W=\sum_i a_i$ is an integer. Therefore, no flat admissible solution exists when $s=2$ and $m\ge 4$.

Next, we consider the case $s\geq 4$ and $m\geq 2$. By \eqref{eq:lower_bound_W flat} we have
\[
\frac{W}{L}\;\geq\; (k+1)\!\left(2-2^{1-s}\right)-\frac{k}{m}.
\]
For $s\geq 4$ we get $2-2^{1-s}\geq 2-2^{-3}=\tfrac{15}{8}$, and for $m\geq 2$ we have $\tfrac{k}{m}\leq \tfrac{k}{2}$. Hence
\[
\frac{W}{L}\;\geq\; (k+1)\,\frac{15}{8}-\frac{k}{2}
\;=\;\frac{15}{8}+\frac{11}{8}\,k
\;\geq\;\frac{15}{8}+\frac{11}{8}
\;=\;\frac{13}{4}
\;>\;3.
\]
On the other hand, by Lemma \ref{lemma:ineq_comb}, we have $\frac{W}{L}\leq 2+\frac{2}{L}\leq 3$, a contradiction.

Finally, we work on the case $s \geq 6$ and $m=1$.  
For $s\ge 6$ we get $2-2^{1-s}\ge 2-2^{-5}=\tfrac{63}{32}$ and since $m=1$
Equation \eqref{eq:lower_bound_W flat} implies
\[
\frac{W}{L}\;\ge\;(k+1)\!\left(2-2^{1-s}\right)-k
\ge\;\frac{63}{32}(k+1)-k
=\frac{31}{32}\,k+\frac{63}{32}.
\]
This last inequality implies $k=1$. Otherwise $W/L > 3$, which is a contradiction. 
We arrive to 
\begin{align*}
 2 + \frac{2}{L} \geq \frac{W}{L} \geq   \frac{47}{16}
\end{align*}
which implies $L =2$. Using $M=\frac{m}{2}D-mW$ with $m=1$ and $M=kL=2$, we obtain
\[
\frac{1}{2}D-W=2\quad\Longrightarrow\quad W=\frac{1}{2}D-2, \quad W\le 2L+2=6
\]
Previous discussion implies
\[
\frac{47}{8} \leq W \leq 6
\quad\Longrightarrow\quad
W=6, \quad D = 16 
\]
Now note that $l_{\chi} \geq 4$ and $2 \mid l_{\chi}$ for all $\chi \neq 0$. Further $\Sigma_{\chi \in G^* -\{0\}} l_{\chi} = 2^{s-2}D = 2^{s+2}$. On the other hand there are $2^s-1$ non-zero characters. This forces the following two tuples of $\{l_{\chi}\}$
\begin{enumerate}
    \item There exists \( \chi_0 \neq 0 \) such that \( l(\chi_0) = 8 \), and \( l(\chi) = 4 \) for all \( \chi \neq \{0, \chi_0\} \) or
    \item There exist distinct \( \chi_0, \chi_1 \neq 0 \) such that \( l(\chi_0) = l(\chi_1) = 6 \), and \( l(\chi) = 4 \) for all \( \chi \neq \{0, \chi_0, \chi_1\} \).
\end{enumerate}
However, such solution does not exist by \Cref{lemma: nonexistence for s6}

\end{proof}

\begin{lemma} \label{lemma: nonexistence for s6}
Let \( G = (\mathbb{Z}_2)^s \) with \( s \ge 6 \), and let \( G^* \) be its dual group, identified with \( G \) via \( \chi_g(x) = (-1)^{g \cdot x} \).  
Let \( l: G^* \to \mathbb{Z}_{\ge 0} \) with \( l(0) = 0 \), and suppose either:
\begin{enumerate}
    \item There exists \( \chi_0 \neq 0 \) such that \( l(\chi_0) = 8 \), and \( l(\chi) = 4 \) for all \( \chi \neq \{0, \chi_0\} \) or
    \item There exist distinct \( \chi_0, \chi_1 \neq 0 \) such that \( l(\chi_0) = l(\chi_1) = 6 \), and \( l(\chi) = 4 \) for all \( \chi \neq \{0, \chi_0, \chi_1\} \).
\end{enumerate}

Then there is no function \( d: G \to \mathbb{Z}_{\ge 0} \) with \( d(0) = 0 \) satisfying
\[
\sum_{\substack{g \in G \\ \chi \cdot g = 1}} d(g) = 2 l(\chi) \quad \text{for all } \chi \neq 0.
\]

\end{lemma}

\begin{proof}

Assume such a \( d \) exists. By \Cref{lemma: Fourier},
\[
\widehat{d}(0) - \widehat{d}(\chi) = 4 l(\chi) \quad \forall \chi \neq 0,
\]
where \( \widehat{d}(\chi) = \sum_{x \in G} d(x) (-1)^{\chi \cdot x} \), and \( \widehat{d}(0) = \sum_{x} d(x) \).

\medskip

\noindent \textbf{Step 1: Determine \( \widehat{d}(0) \) from \( d(0) = 0 \)} \\

Apply the inverse Fourier transform at \( x = 0 \):
\[
d(0) = \frac{1}{|G|} \sum_{\chi \in G^*} \widehat{d}(\chi) (-1)^{\chi \cdot 0} = \frac{1}{|G|} \sum_{\chi \in G^*} \widehat{d}(\chi).
\]
Since \( d(0) = 0 \), we have
\[
\sum_{\chi \in G^*} \widehat{d}(\chi) = 0.
\]
Split the sum: \( \widehat{d}(0) + \sum_{\chi \neq 0} \widehat{d}(\chi) = 0 \).

From \( \widehat{d}(0) - \widehat{d}(\chi) = 4 l(\chi) \), we get \( \widehat{d}(\chi) = \widehat{d}(0) - 4 l(\chi) \).

So:
\[
\widehat{d}(0) + \sum_{\chi \neq 0} \left[ \widehat{d}(0) - 4 l(\chi) \right] = 0.
\]
There are \( |G| - 1 = 2^s - 1 \) nonzero \( \chi \). Thus:
\[
\widehat{d}(0) + (2^s - 1) \widehat{d}(0) - 4 \sum_{\chi \neq 0} l(\chi) = 0.
\]
That is:
\[
2^s \widehat{d}(0) - 4 \sum_{\chi \neq 0} l(\chi) = 0 \quad \Rightarrow \quad \widehat{d}(0) = \frac{4}{2^s} \sum_{\chi \neq 0} l(\chi).
\]

\medskip

\noindent \textbf{Step 2: Compute \( \widehat{d}(0) \) for each case} \\
\textbf{Case 1:}  
One \( \chi_0 \) with \( l = 8 \), and \( 2^s - 2 \) others with \( l = 4 \).  
So:
\[
\sum_{\chi \neq 0} l(\chi) = 8 + (2^s - 2) \cdot 4 = 8 + 4 \cdot 2^s - 8 = 4 \cdot 2^s.
\]
Thus \( \widehat{d}(0) = \frac{4}{2^s} \cdot 4 \cdot 2^s = 16 \). \\
\textbf{Case 2:}  
Two \( \chi \) with \( l = 6 \), and \( 2^s - 3 \) with \( l = 4 \).  
So:
\[
\sum_{\chi \neq 0} l(\chi) = 2 \cdot 6 + (2^s - 3) \cdot 4 = 12 + 4 \cdot 2^s - 12 = 4 \cdot 2^s.
\]
Again \( \widehat{d}(0) = 16 \). So in both cases, \( \widehat{d}(0) = 16 \).

\medskip

\noindent \textbf{Step 3: Determine the values of \( \widehat{d}(\chi) \) } \\

\textbf{Case 1:}  
\( \widehat{d}(\chi_0) = 16 - 4 \cdot 8 = -16 \),  
\( \widehat{d}(\chi) = 16 - 4 \cdot 4 = 0 \) for \( \chi \neq 0, \chi_0 \).

\textbf{Case 2:}  
\( \widehat{d}(\chi_0) = \widehat{d}(\chi_1) = 16 - 4 \cdot 6 = -8 \),  
\( \widehat{d}(\chi) = 0 \) for \( \chi \neq 0, \chi_0, \chi_1 \).

\medskip

\noindent \textbf{Step 4: Apply inverse Fourier transform} \\
By \Cref{lemma: Inverse Fourier Transform},:
\[
d(x) = \frac{1}{2^s} \sum_{\chi \in G^*} \widehat{d}(\chi) (-1)^{\chi \cdot x}.
\]

\textbf{Case 1:}  
Only \( \chi = 0 \) and \( \chi = \chi_0 \) contribute:
\[
d(x) = \frac{1}{2^s} \left[ 16 + (-16)(-1)^{\chi_0 \cdot x} \right] = \frac{16}{2^s} \left[ 1 - (-1)^{\chi_0 \cdot x} \right].
\]
If \( \chi_0 \cdot x = 0 \), \( d(x) = 0 \).  
If \( \chi_0 \cdot x = 1 \), \( d(x) = \frac{16}{2^s} \cdot 2 = \frac{32}{2^s} \).

For \( s \ge 6 \), \( 2^s \ge 64 \), so \( \frac{32}{2^s} \le \frac{1}{2} \), and is an integer only if \( 2^s \mid 32 \), i.e., \( s \le 5 \).  
Thus for \( s \ge 6 \), \( d(x) \) is not an integer for some \( x \).

\textbf{Case 2:}  
Contributions from \( \chi = 0, \chi_0, \chi_1 \):
\[
d(x) = \frac{1}{2^s} \left[ 16 + (-8)(-1)^{\chi_0 \cdot x} + (-8)(-1)^{\chi_1 \cdot x} \right].
\]
Let \( a = (-1)^{\chi_0 \cdot x} \), \( b = (-1)^{\chi_1 \cdot x} \), each \( \pm 1 \).

\begin{enumerate}
    \item \( a = 1, b = 1 \): \( d(x) = \frac{1}{2^s}[16 - 16] = 0 \).
    \item \( a = 1, b = -1 \): \( d(x) = \frac{1}{2^s}[16 - 0] = \frac{16}{2^s} \).
    \item \( a = -1, b = 1 \): \( d(x) = \frac{16}{2^s} \).
    \item \( a = -1, b = -1 \): \( d(x) = \frac{1}{2^s}[16 + 16] = \frac{32}{2^s} \).
\end{enumerate}
For \( s \ge 6 \), \( \frac{16}{2^s} \) and \( \frac{32}{2^s} \) are not integers.
\medskip
Thus in both cases, \( d(x) \) is not integer-valued for \( s \ge 6 \), contradicting the assumption.  
Therefore, no such \( d \) exists.

\end{proof}

\subsection{Proof Theorem \ref{thm:flat_pluricanonical_covers}}
By Lemma \ref{lemma:forbidden_flat_admissible}, we only need to consider the list of $s$ and $m$ within the statement. In principle, there is an infinite number of possible weights and degrees to consider. Our first step is to bound the number of possible weights and degrees to a final list of tuples
$(L,W,D,k)$ that must be satisfied by any solution. We then find all possible cases using computational techniques.

We start by bounding the number of possible cases.  We recall the identities and bounds already established:
\[
M=\frac{m}{2}D-mW=kL,\qquad L\ge 2,\qquad W\le 2L+2,
\]
and, writing $\beta(s):=2-2^{1-s}$,
\begin{equation}
\label{eq:WLB-again}
\frac{W}{L}\ \ge\ (k+1)\,\beta(s)\ -\ \frac{k}{m},
\qquad
D\ \ge\ \bigl(4-2^{2-s}\bigr)(k+1)\,L,
\end{equation}
which are \eqref{eq:lower_bound_W flat}, \eqref{eq:lower_bound_D flat} and Statement~1 of Lemma~\ref{lemma:ineq_comb}.
We also use the exact relation
\begin{equation}
\label{eq:Link-again}
D\ =\ 2W+\frac{2k}{m}\,L,
\end{equation}
and the divisibility $m\mid 2kL$ (since $m(D-2W)=2kL$).
For $s=2$, it also holds that $L\mid D$ because $D=\sum_{\chi}l_\chi$ and $L\mid l_\chi$ for all $\chi$. 

\smallskip 

\noindent\textbf{Proof of Case $(s,m)=(2,3)$}\label{s = 2, m = 3}
Case . By Lemma~\ref{lemma:ineq_comb}.3 (with $s\ge 2$, $m\ge 3$) we have $k=1$ and $L\le 3$.
Since $3\mid 2L$, necessarily $L=3$. From \eqref{eq:WLB-again}:
\[
\frac{W}{L}\ \ge\ 2\beta(2)-\frac{1}{3}\ =\ 2\cdot\frac{3}{2}-\frac{1}{3}
\ =\ \frac{8}{3}.
\]
Together with $W/L\le 2+\frac{2}{L}=2+\frac{2}{3}$ we obtain $W=8$. Then \eqref{eq:Link-again} gives
$D=2W+\tfrac{2}{3}L=16+2=18$. For $s=2$, $L\mid D$ holds.
This implies each $l_\chi=6$ because $\sum_\chi l_\chi=D=18$ and each $l_\chi\ge 2L=6$ and multiple of $L$. Therefore,  $d_{10}=d_{01}=d_{11}=6$. 

\smallskip

\noindent\textbf{Proof of Case $(s,m)=(2,2)$}\label{s = 2, m = 2}
\textbf{Case $(s,m)=(2,2)$}. From \eqref{eq:WLB-again} and $W/L\le 3$, we obtain $k=1$.
This implies
\[
\frac{W}{L}\ \in\Bigl[\ \tfrac{5}{2}\ ,\ 2+\tfrac{2}{L}\ \Bigr],\qquad L\ge 2.
\]
Our possible values $L\in\{2,3,4\}$. A brief check gives
\[
(L,W)\in\{(2,5),(2,6),(4,10), (3,8)\},
\]
From \eqref{eq:Link-again}, $D = 2W+L$ implies that for $(L,W) = (3,8)$, $D = 19$. But since $3 \nmid 19$, the pair $(3,8)$ is excluded. Hence again from \eqref{eq:Link-again},
$D=2W+L\in\{12,14,24\}$; 
\color{black}
So we consider the pairs $(L,W,D) = (2,5,12), (2,6,14), (4,10,24)$. Note that since $s = 2$, by \Cref{eq: divisor sum and line bundle sum}, $\Sigma_{\chi \in G^*} l_{\chi} = D$. Moreover since $l_{\chi} \geq 2L = 4$ and even, we have that in each of these cases, $\{l_{\chi}\}_{\chi \in G^*} = (4,4,4), (4,4,6), (8,8,8)$. Now calculating the branch divisors $\{d_g\}_{g \in G}$ gives the result.

\color{black}

\noindent\textbf{Proof of Case $(s,m)=(2,1)$}\label{s = 2, m = 1}
\textbf{Case $(s,m)=(2,1)$:}
From \eqref{eq:WLB-again}
\[
\frac{W}{L}\ \ge\ (k+1)\frac{3}{2}-k\ =\ \frac{k+3}{2}.
\]
Since $W/L\le 3$, necessarily $k\in\{1,2,3\}$. We enumerate per $k$, enforcing $L\ge 2$, $W\le 2L+2$, and \eqref{eq:Link-again}. First, we suppose $k=1$. \textcolor{black}{In this case, we note that $2L+2 \geq W \geq 2L$. This means $W \in \{2L, 2L+1, 2L+2\}$. If $W = 2L+1$, we have that $D = 2W+\frac{2k}{m}L = 6L+2$. Since $s = 2$, we have that $L \mid D$ and hence $L = 2$. Hence we get $(L,W,D) = (2,5,14)$. Similar analysis with $W = 2L+2$ yields $L = 2$ or $L = 4$. $L = 2$ yields $(L,W,D) = (2,6,16)$ while $L = 4$ yields $(4,10,28)$. Now consider $W = 2L$. We do not directly get an upper bound for $L$ as above, but similar analysis yields $(L,W,D) = (L,2L,6L)$. But now it is easy to see that the only solutions for a set of positive integers $(a_0,a_1,a_2,a_3)$ such that $\Sigma_ia_i = 2L$ and $\operatorname{lcm}(a_0,a_1,a_2,a_3) = L$ are $(1,1,1,3), (1,1,2,4), (1,2,3,6)$ and in particular $L = 3,4,6$. This yields the triples $(L,W,D) = (3,6,18), (4,8,24), (6,12,36)$.} Hence the triples we obtain are
\[
(L,W,D)\in\{(2,5,14),\ (2,6,16),\ (3,6,18),\ (4,8,24),\ (4,10,28),\ (6,12,36)\}.
\]
\textcolor{black}{Consider $(L,W,D) = (2,5,14)$. This yields $\mathbb{P}(1,1,1,2)$. Since $\Sigma_{\chi \in G^*-\{0\}}l_{\chi} = D = 14$, $2 \mid l_{\chi}$ and $l_{\chi} \geq 2L = 4$, we have $\{l_{\chi}\} = (4,4,6)$ which yields $d_g = (2,6,6)$. Consider $(L,W,D) = (2,6,16)$. This yields $\mathbb{P}(1,1,2,2)$ and $d_g = (8,8,0)$ or $(4,4,8)$. Consider $(L,W,D) = (4,10,28)$. This gives $\mathbb{P}(1,1,4,4)$ and $d_g = (4,12,12)$. Consider $(L,W,D) = (3,6,18)$. This gives $\mathbb{P}(1,1,1,3)$ and $d_g = (6,6,6)$. Consider $(L,W,D) = (4,8,24)$. This gives $\mathbb{P}(1,1,2,4)$ and $d_g = (8,8,8)$. Consider $(L,W,D) = (6,12,36)$. This gives $\mathbb{P}(1,2,3,6)$ and $d_g = (12,12,12)$.}

Next, we suppose $k=2$. We obtain
\[
(L,W,D)\in\{(2,5,18),\ (2,6,20),\ (4,10,36),\ \textcolor{black}{(3,9,30),\ (4,12,40)} \},
\]
Note that there is no choice of $\mathbb{P}(a_0,a_1,a_2,a_3)$ with $W = 9$ and $L = 3$ or $W = 12$ and $L = 4$. Now $(2,5,18)$ yields $d=(6,6,6)$, $(2,6,20)$ yields $d = (4,8,8)$ and $(4,10,36)$ yields $d=(12,12,12)$   

Finally, we suppose $k=3$. We obtain uniquely
\[
(L,W,D)=(2,6,24),
\]
with $d=(8,8,8)$ forced.

\smallskip

\noindent\textbf{Proof of Case $(s,m)=(3,2)$}\label{s = 3, m = 2}
\textbf{Case $(s,m)=(3,2)$:}
From \eqref{eq:WLB-again} we have
\[
\frac{W}{L}\ \ge\ (k+1)\frac{7}{4}-\frac{k}{2}\ =\ \frac{5k+7}{4}.
\]
The bound $W/L\le 3$ forces $k=1$, and then $W/L\ge 3$. Since $2+\frac{2}{L}\le 3$, this yields $L=2$, $W=6$.
By \eqref{eq:Link-again}, $D=14$. Setting $N_{\chi} = 2l_{\chi}$, we have $\sum_\chi N_\chi=2^{2}D=56$ and $N_\chi\in 4\mathbb Z$, $N_\chi\ge 8$;
hence each $N_\chi=8$ and $l_\chi=4$. 

\smallskip

\noindent\textbf{Proof of Case $(s,m)=(3,1)$}\label{s = 3, m = 1}
\textbf{Case $(s,m)=(3,1)$.}
From \eqref{eq:WLB-again} we have $\frac{W}{L}\ge (k+1)\frac{7}{4}-k=\frac{3k+7}{4}$; thus $k=1$.
Hence
\[
\frac{W}{L}\in\Bigl[\tfrac{5}{2},\ 2+\tfrac{2}{L}\Bigr],\qquad L\ge 2,
\]
\textcolor{black}{which yields $(L,W,D)\in\{(2,5,14),(2,6,16),(3,8,22), (3,9,24), (4,10,28), (4,11,30), (4,12,32)\}$. But now there are no solutions of well-formed $\mathbb{P}(a_0,a_1,a_2,a_3)$ with $(W = 9, L = 3)$, $(W = 11, L = 4)$ and $(W = 12, L = 4)$. The triple $(2,5,14)$ yields $\mathbb{P}(1,1,1,2)$ and $\{l_{\chi}\} = (4,4,4,4,4,4,4)$, the triple $(2,6,16)$ yields $\mathbb{P}(1,1,2,2)$ and $\{l_{\chi}\} = (8,4,4,4,4,4,4)$ or $(4,4,6,6,4,4,4)$. The triple $(3,8,22)$ yields $\Sigma_{\chi \in G^*-\{0\}} l_{\chi} = 44$, $l_{\chi} \geq 6$ and $3 \mid l_{\chi}$ for all $\chi$. This does not yield any solution for $\{l_{\chi}\}$. Finally the triple $(4,10,28)$ yields $\mathbb{P}(1,1,4,4)$ and $\{l_{\chi}\} = (8,8,8,8,8,8,8)$. Now the values of $d_g$ can be computed from $\{l_{\chi}\}$ using the fundamental equations and the value of $D$.}

\smallskip

\noindent\textbf{Proof of Case $(s,m)=(4,1), (5,1)$}\label{s = 4,5 m = 1}

We first show that in both cases, we must have $L = 2$, $k = 1$, $W = 6$ and $D = 16$. We show this for $s = 4$. For $s = 5$, the computation is verbatim.
From \eqref{eq:WLB-again}, $\frac{W}{L}\ge (k+1)\frac{15}{8}-k=\frac{7k+15}{8}$; thus $k=1$. Then
$\frac{W}{L}\ge \frac{11}{4}$, and since $2+\frac{2}{L}\ge \frac{11}{4}$ forces $L=2$, we get $W=6$ and, by \eqref{eq:Link-again},
$D=16$.

Assume flatness and the standing notation of this subsection. Then any flat admissible solution with $s=4$ and $m=1$
satisfies
\[
k=1,\qquad L=2,\qquad W=6,\qquad D=16,
\]

Now note that $l_{\chi} \geq (k+1)L = 4$ and $2 \mid l_{\chi}$. Further $\Sigma_{\chi \in G^*-\{0\}}l_{\chi} = 2^{s-2}D = 64$. Hence there are two cases

\begin{enumerate}
    \item Case $(1)$: There exists a $\chi_0 \in G^*-\{0\}$ such that $l_{\chi_0} = 8$ and $l_{\chi} = 4$ for $\chi \neq \chi_0$ and $\chi \in G^*-\{0\}$.
    \item Case $(2)$: There exists a $\chi_0, \chi_1 \in G^*-\{0\}$ such that $l_{\chi_0} = l_{\chi_1} = 6$ and $l_{\chi} = 4$ for $\chi \neq \chi_0$, $\chi \neq \chi_1$ and $\chi \in G^*-\{0\}$.
\end{enumerate}

For $s = 4$, applying \Cref{lemma: Inverse Fourier Transform}, we get 

\begin{enumerate}
    \item In case $(1)$, $d_g = 0$ if $\chi_0 \cdot g = 0$ (there are $7$ such values of $g$) and $d_g = 2$ if $\chi_0 \cdot g = 1$ (there are $8$ such values of $g$).
    \item In case $(2)$, $d_g = 0$ if $\chi_0 \cdot g = 0$ and $\chi_1 \cdot g = 0$ (there are $3$ such values of $g$) , $d_g = 1$ if $\chi_0 \cdot g = 1$ and $\chi_1 \cdot g = 0$ or vice versa (there are $8$ such values of $g$) and finally $d_g = 2$ if $\chi_0 \cdot g = 1$ and $\chi_1 \cdot g = 1$ (there are $4$ such values of $g$).
    
\end{enumerate}

For $s = 5$, applying \Cref{lemma: Inverse Fourier Transform} and \Cref{lemma: Fourier}, we get 

\begin{enumerate}
    \item In case $(1)$, $d_g = 0$ if $\chi_0 \cdot g = 0$ (there are $15$ such values of $g$) and $d_g = 1$ if $\chi_0 \cdot g = 1$ (there are $16$ such values of $g$).
    \item In case $(2)$, there are no integer solutions for $\{d_g\}$.
    
\end{enumerate}

\subsection{Pluricanonical covers of $\mathbb{P}(1,1,1,1)$}

We continue with the notation of this section to classify the pluricanonical $\mathbb{Z}_2^s$ covers of $\mathbb{P}(1,1,1,1)$. Hence $W = 4$ and $L = 1$. Note that in this case $M = k = m(\displaystyle\frac{1}{2}D-4)$, where $D = \sum_g d_g$ and hence $D = 8+\frac{2k}{m}$. Further recall that $s = 1$ has already been classified.

\begin{lemma}\label{m atleast 3}
   \begin{enumerate}
       \item If $m \geq 3$ and $s \geq 2$, then $k = 2$, $m = 4$ and $\sum_g d_g = 9$
       \item If $m \geq 3$ and $s \geq 3$, then there are no admissible solutions.
       \item If $m \geq 5$ and $s \geq 2$, then there are no admissible solutions.
       
   \end{enumerate}
    
\end{lemma}

\begin{proof}
  \noindent\textit{Proof of $(1)$}:  By \Cref{lemma:ineq_comb}, $(2)$, we have that 
    \begin{align*}
    W \geq \left( (k+1)(2 -2^{1-s})- \frac{k}{m}\right)L &&
l_\chi \geq (k+1)L, \; \; \forall \chi
\end{align*}

Since $W = 4$ and $L = 1$, we have 
\begin{equation*}
    \left( (k+1)(2 -2^{1-s})- \frac{k}{m}\right) \leq 4
\end{equation*}

Since $s \geq 2$ and $m \geq 3$, we have that the LHS $ \geq \displaystyle\frac{7k}{6}+\frac{3}{2}$. Hence $k \leq 2$. \\

Suppose that $k = 1$. Then $M = kL = 1$. Recall that $M = m(\displaystyle\frac{1}{2} \sum_g d_g-4)$. Hence we have 
\begin{equation*}
m(\displaystyle\frac{1}{2} \sum_g d_g-4) = 1
\end{equation*}

This implies that $\sum_g d_g = 8+\displaystyle\frac{2}{m}$. This implies that $m = 1$ or $m = 2$, contradicting $m \geq 3$. \\

If $k = 2$, then $M = 2$ and $\sum_g d_g = 8+\displaystyle\frac{4}{m}$. Then $m \mid 4$ and $m \geq 3 \implies m = 4$. Therefore $\sum_g d_g = 9$. \\

\noindent\textit{Proof of $(2)$.} Now suppose that $m \geq 3$ and $s \geq 3$. From part $(1)$, $k = 1$ or $k = 2$. Suppose that $k = 2$. Then we have from \Cref{lemma:ineq_comb}, $(2)$, 

\begin{equation*}
    \left( (3)(2 -2^{1-s})- \frac{2}{m}\right) \leq 4
\end{equation*}

which, since $s \geq 3$, gives $m \leq \displaystyle\frac{8}{5}$ contradicting $m \geq 3$. \\
Now assume $k = 1$. This implies $M = 1$ and hence $\sum_g d_g = 8+\displaystyle\frac{2}{m}$. This implies $m \mid 2$ and contradicts $m \geq 3$. \\

\noindent\textit{Proof of $(3)$.} If $m \geq 5$ and $s \geq 2$, once again by part $(1)$, we have $k = 1$ or $k = 2$. $k = 2 \implies M = 2 \implies \sum_g d_g = 8+\displaystyle\frac{4}{m} \implies m \mid 4$, which is a contradiction. Similarly $k = 1 \implies m \mid 2$, which leads to a contradiction. \\



\end{proof}

\begin{lemma}\label{m atmost 2}
    \begin{enumerate}
        \item If $m = 2$ and $k \geq 2$, then $s \leq 2$
        \item If $m = 2$ and $k \geq 3$, then $s \leq 1$
        \item If $m = 1$ and $k \geq 3$, then $s \leq 3$
        \item If $m = 1$ and $k \geq 4$, then $s \leq 2$
        \item If $m = 1$ and $k \geq 6$, then $s \leq 1$
    \end{enumerate}
\end{lemma}

\begin{proof}
\noindent\textit{Proof of $(1)$ and $(2)$.} Recall that we have 
\begin{equation*}
\left( (k+1)(2 -2^{1-s})- \frac{k}{m}\right) \leq 4
\end{equation*}

If $m = 2$, the above gives

\begin{equation*}
    2^s \leq \displaystyle\frac{4k+4}{3k-4}
\end{equation*}

Hence $k \geq 2 \implies s \leq 2$ and if $k \geq 3 \implies s \leq 1$. \\

\noindent\textit{Proof of $(3)$, $(4)$ and $(5)$.} 
Using 
\begin{equation*}
\left( (k+1)(2 -2^{1-s})- \frac{k}{m}\right) \leq 4
\end{equation*}

if $m = 1$, we get

\begin{equation*}
    2^s \leq \displaystyle\frac{2k+2}{k-2}
\end{equation*}

Now the conclusion follows by plugging in the corresponding values of $k$.
    
\end{proof}

\begin{corollary}\label{list of cases}
  The values of $m$, $s$, $k$ and $D$ for which there are possible admissible solutions are as follows: 
    \begin{enumerate}
        \item $m = 1$, $s \geq 2$, $k = 1 \iff D = 10$
        \item $m = 1$, $s \geq 2$, $k = 2 \iff D = 12$
        \item $m = 1$, $s = 2$, $k = 3 \iff D = 14$
        \item $m = 1$, $s = 3$, $k = 3 \iff D = 14$
        \item $m = 1$, $s = 2$, $k = 4 \iff D = 16$
        \item $m = 1$, $s = 2$, $k = 5 \iff D = 18$
        \item $m =2$, $s \geq 2$, $k = 1 \iff D = 9$
        \item $m = 2$, $s = 2$, $k = 2 \iff D = 10$
        \item $m = 4$, $s = 2$, $k = 2 \iff D = 9$
        
    \end{enumerate}
\end{corollary}

\begin{proof}
   Follows by combining \Cref{m atleast 3} and \Cref{m atmost 2}. 
\end{proof}

\begin{theorem}
Upto $\operatorname{GL}(s,\mathbb{F}_2)$ action
 \begin{enumerate}
     \item For $m = 4$, $s=2$ and $k = 2$ and $D = 9$, $(d_{10}, d_{01}, d_{11}) = (3,3,3)$ 
     \item For $m = 2$, $s=2$ and $k = 2$ and $D = 10$, $(d_{10}, d_{01}, d_{11}) = (4,4,2)$ 
     \item For $m = 1$, $s=2$ and $k = 5$ and $D = 18$, $(d_{10}, d_{01}, d_{11}) = (6,6,6)$
     \item For $m = 1$, $s=2$ and $k = 4$ and $D = 16$, $(d_{10}, d_{01}, d_{11}) = (6,6,4)$
     \item For $m = 1$, $s=2$ and $k = 3$ and $D = 14$, $(d_{10}, d_{01}, d_{11}) = (6,6,2)$ or $(d_{10}, d_{01}, d_{11}) = (4,4,6)$
     \item For $m = 1$, $s=3$ and $k = 3$ and $D = 14$, 
     $(d_{100},d_{010},d_{110},d_{001},d_{101},d_{011},d_{111}) = (2,2,2,2,2,2,2)$

 \end{enumerate}   
    
\end{theorem}

\begin{proof}
    The proof follows from \Cref{list of cases}, and the facts that $\sum_{\chi \in G^*} l_{\chi} = 2^{s-2}D$, $l_{\chi} \geq (k+1)$ for $\chi \neq 0$, and $2l_{\chi} = \sum_{\{g \in G|\chi \cdot g = 1\}} d_g$.
\end{proof}

The remaining cases are $(1)$, $(2)$ and $(7)$ of \Cref{list of cases}. However, since $m = 1, k = 1$ corresponds to canonical covers of $\mathbb{P}^3$ with $p_g = 4$, this is done in \cite{dugao16}. So we analyze cases $(2)$ and $(7)$ of \Cref{list of cases}. 

We first state a Proposition to bound the value of $s$, given an upper bound to $D$.

\begin{proposition}\label{bound to $s$}
Let $G=(\mathbb Z_2)^s$. Let $d:G\to\mathbb Z_{\ge 0}$ satisfy
\[
d(0)=0,\qquad \sum_{g\in G} d(g)=C,
\]
and suppose there exists an integer $p\ge 1$ such that for every nonzero
character $\chi\in G^*$,
\[
\sum_{\chi(g)=1} d(g)\ge p.
\tag{$\star$}
\]
Let
\[
S:=\{g\in G : d(g)>0\}
\]
Then
\begin{enumerate}
    \item  $s \leq |S|$
    \item Suppose $s \geq t+1$ and let $g_1,\dots,g_t \in S$ be distinct. Then 
    \[
    \sum_{i=1}^t d(g_i)\le C-p.
    \]
    \item In particular, 
    \[ s \le C-p+1. \]
\end{enumerate}

\end{proposition}

\begin{proof}
Let
\[
S:=\{g\in G : d(g)>0\}
\]
denote the support of $d$. Then $|S|\le C$.

\medskip

\noindent\textit{Proof of $(1)$.} Suppose $\langle S\rangle$ is a proper subspace of $G$. Then there exists a nonzero character $\chi\in G^*$ such that $\chi(g)=0$ for all $g\in S$. For
this $\chi$,
\[
\sum_{\chi(g)=1} d(g)=0,
\]
contradicting $(\star)$. Hence $\langle S\rangle=G$, and therefore
\[
s=\dim G\le |S|\le C.
\]

\medskip

\noindent
\noindent\textit{Proof of $(2)$.} Let $t$ be a positive integer. If $s\ge t+1$, then for any choice of $t$ distinct elements $g_1,\dots,g_t\in G$, there exists a nonzero character
$\chi\in G^*$ such that
\[
\chi(g_1)=\cdots=\chi(g_t)=0.
\]
Assume $s\ge t+1$ and let $g_1,\dots,g_t\in S$ be distinct. Choose $\chi$ as in
Step~$2$. Then
\[
\sum_{\chi(g)=1} d(g)
\leq C - \sum_{i=1}^t d(g_i).
\]
By $(\star)$, we have
\[
\sum_{i=1}^t d(g_i)\le C-p.
\tag{1}
\]
Thus, if $s\ge t+1$, every $t$-subset of the support has total weight at most
$C-p$.

\noindent\textit{Proof of $(3)$.} Set
\[
t := C-p+1.
\]
If $s\ge t+1=C-p+2$, then by (1) every subset of $S$ of size $t$ has total
weight at most $C-p$ which is a contradiction since $d(g) > 0$ for $g \in S$.
\end{proof}

Now we use the above proposition as a base case and sharpen the upper bound to $s$ for each of cases $(2)$ and $(7)$ in \Cref{list of cases}. From now on, define $m_i:=\#\{g\neq 0:d(g)=i\}$ and $n_i:=\#\{\chi\neq 0:\ell(\chi)=i\}$. Recall that $d(0) = \ell(0) = 0$
Then

\begin{lemma}\label{divisor counts in list of cases}

Let $S:=\{g\in G : d(g)>0\}$. In \Cref{list of cases},

\begin{enumerate}
    \item in item $(2)$, $s \leq 7$
    \item in item $(2)$, if $s = 7$, $|S| = 12$ and $d_g = 1$ for all $g \in S$.
    \item in item $(2)$, if $s = 6$, then $d_g \in \{1, 2\}$. If $m_i = | g \in S, d_g = i|$, then the possible tuples are $(m_1) = 12$ $(m_1, m_2) = (10, 1)$.
    \item in item $(2)$, if $s = 5$, then for $g \in S$, $d_g \in \{1,2,3\}$. If $m_i = | g \in S, d_g = i|$, then the possible tuples $(m_1, m_2, m_3)$ are $(m_1, m_3) = (9,1)$, $(m_1, m_2) = (8,2)$, $(m_1, m_2) = (10,1)$, $(m_1) = (12)$.
    \item in item $(2)$, if $s = 4$, then for $g \in S$, $d_g \in \{1,2,3, 4\}$. If $m_i = | g \in S, d_g = i|$, then the possible tuples $(m_1, m_2, m_3, m_4)$ are $(m_1, m_4)= (8,1)$, $(m_1,m_2,m_3) = (7,1,1)$, $(m_1,m_3) = (9,1)$, $(m_2) = (6)$, $(m_1, m_2) = (2,5)$, $(m_1, m_2) = (4,4)$, $(m_1, m_2) = (6,3)$, $(m_1, m_2) = (8,2)$, $(m_1, m_2) = (10,1)$, $(m_1) = (12)$.

    \item in item $(7)$, $s \leq 6$.
    \item in item $(7)$, if $s = 6$, then $|S| = 9$ and $d_g = 1$ for all $g \in S$.
    \item in item $(7)$, if $s = 5$, then $d_g \in \{1, 2\}$. If $m_i = | g \in S, d_g = i|$, then the possible tuples $(m_1, m_2)$ are $(9,0)$ and $(7,1)$.
    \item in item $(7)$, if $s = 4$, then $d_g \in \{1, 2, 3\}$. If $m_i = | g \in S, d_g = i|$, then the possible tuples $(m_1, m_2, m_3)$ are $(9,0,0)$, $(7,1,0)$, $(5,2,0)$, $(6,0,1)$.
    
\end{enumerate}

\end{lemma}

\begin{proof}
    The proofs of all these statements come from the repeated use of \Cref{bound to $s$}, part $(2)$ and $(3)$.
\end{proof}


\begin{proposition}\label{m = 2; k = 1; D = 9 impossibility}
Suppose $m = 2, k = 1$ and $D = 9$. Let $S:=\{g\in G : d(g)>0\}$ and $m_i = | g \in S, d_g = i|$.  Then there do not exist admissible solutions for 
\begin{enumerate}
    \item $s = 6, (m_1) = (9)$
    \item $s = 5, (m_1) = (9)$
    \item $s = 5, (m_1,m_2) = (7,1)$
    \item $s = 4,  (m_1, m_2) = (5,2)$
    \item $s = 4,  (m_1, m_3) = (6,1)$
    
\end{enumerate}
\end{proposition}

\begin{proof}
\noindent\textit{Proof of $(1)$} Let $\widehat d$ be the unnormalized Fourier transform
\[
\widehat d(\chi)=\sum_{x\in G} d(x)\,(-1)^{\chi\cdot x}.
\]
By \Cref{lemma: Fourier}, for every $\chi\neq 0$ we have
\[
\widehat d(0)-\widehat d(\chi)=4\ell(\chi).
\]
Since $\widehat d(0)=\sum_x d(x)=D=9$, it follows that for $\chi\neq 0$,
\[
\widehat d(\chi)=9-4\ell(\chi).
\]
Moreover $\ell(\chi)=\tfrac12\sum_{\chi\cdot x=1} d(x)$, and the sum on the right is at most
$D=9$, hence $\ell(\chi)\le 4$. Together with $\ell(\chi)\ge 2$, we conclude
\[
\ell(\chi)\in\{2,3,4\}
\qquad\Longrightarrow\qquad
\widehat d(\chi)\in\{1,-3,-7\}\quad(\chi\neq 0).
\]

Let
\[
n_2=\#\{\chi\neq 0:\ell(\chi)=2\},\quad
n_3=\#\{\chi\neq 0:\ell(\chi)=3\},\quad
n_4=\#\{\chi\neq 0:\ell(\chi)=4\}.
\]
Then
\[
n_2+n_3+n_4=63.
\]

\medskip\noindent
\emph{Step 1: first and second moment constraints.}
Since $d(0)=0$, Lemma~\ref{lem:first-moment} gives the first moment constraint
\[
\sum_{\chi\in G^*}\widehat d(\chi)=|G|\cdot d(0)=0.
\]
Using $\widehat d(0)=9$, this becomes
\[
9+\sum_{\chi\neq 0}\widehat d(\chi)=0
\qquad\Longrightarrow\qquad
9+(n_2-3n_3-7n_4)=0,
\]
i.e.
\begin{equation}\label{eq:first-moment-ns}
n_2-3n_3-7n_4=-9.
\end{equation}

Next, \Cref{Plancherel for the unnormalized Fourier transform} yields
\[
\sum_{\chi\in G^*}|\widehat d(\chi)|^2 = |G|\sum_{x\in G}|d(x)|^2.
\]
Here $|G|=64$ and $d(x)^2=d(x)$ since $d$ is $\{0,1\}$--valued, so $\sum_x d(x)^2=D=9$.
Hence
\[
\sum_{\chi\in G^*}|\widehat d(\chi)|^2=64\cdot 9=576.
\]
Subtracting the $\chi=0$ term ($|\widehat d(0)|^2=81$) gives
\[
\sum_{\chi\neq 0}|\widehat d(\chi)|^2=576-81=495,
\]
i.e.
\begin{equation}\label{eq:plancherel-ns}
n_2\cdot 1^2+n_3\cdot 3^2+n_4\cdot 7^2
=
n_2+9n_3+49n_4
=
495.
\end{equation}

Solving the system consisting of
\[
n_2+n_3+n_4=63,
\qquad
\eqref{eq:first-moment-ns},
\qquad
\eqref{eq:plancherel-ns},
\]
gives
\[
n_4=9,\qquad n_3=0,\qquad n_2=54.
\]

In particular,
\begin{equation}\label{eq:rigid-spectrum}
\widehat d(\chi)=
\begin{cases}
-7,& \text{for exactly }9\text{ characters }\chi\neq 0,\\
\ \ 1,& \text{for the remaining }54\text{ characters }\chi\neq 0,
\end{cases}
\qquad\text{and }\widehat d(0)=9.
\end{equation}

\medskip\noindent
\emph{Step 2: cubic moment contradiction.}
Consider the triple convolution $d*d*d$ and evaluate at $0$.
If $S=\{x\in G:d(x)=1\}$, then by
Lemma~\ref{lem:fourier-convolution-cubic}(2),
\[
(d*d*d)(0)=\#\{(x,y,z)\in S^3:\ x+y+z=0\}\in\mathbb Z_{\ge 0}.
\]
On the other hand, the cubic moment identity in the same lemma gives
\[
(d*d*d)(0)=\frac{1}{|G|}\sum_{\chi\in G^*}\widehat d(\chi)^3.
\]
Using \eqref{eq:rigid-spectrum} and $|G|=64$, we compute
\[
\sum_{\chi\in G^*}\widehat d(\chi)^3
=
9^3+54\cdot 1^3+9\cdot(-7)^3
=
729+54-3087=-2304.
\]
Therefore
\[
(d*d*d)(0)=\frac{-2304}{64}=-36,
\]
which contradicts $(d*d*d)(0)\in\mathbb Z_{\ge 0}$. This contradiction shows that no such
function $d$ can exist.  

\noindent\textit{Proof of $(2)$.} In this case, we have
\begin{equation}\label{eq:n1}
n_2+n_3+n_4=31.
\end{equation}

We have
\[
\sum_{\chi\neq 0}\ell_\chi = 2^{s-2}D.
\]
Here $s=5$ and $D=9$, so $\sum_{\chi\neq 0}\ell_\chi=2^{3}\cdot 9=72$, i.e.
\begin{equation}\label{eq:n2}
2n_2+3n_3+4n_4=72.
\end{equation}

Next apply Parseval:
\[
\sum_{\chi\in G^*}\widehat d(\chi)^2
=|G|\sum_{x\in G} d(x)^2.
\]
Since $d\in\{0,1\}$, we have $\sum_x d(x)^2=\sum_x d(x)=9$, and $|G|=32$, so
\[
\sum_{\chi}\widehat d(\chi)^2 = 32\cdot 9=288.
\]
Subtracting the $\chi=0$ term gives
\[
\sum_{\chi\neq 0}\widehat d(\chi)^2
=288-\widehat d(0)^2
=288-81
=207.
\]
But for $\chi\neq 0$, $\widehat d(\chi)=9-4\ell_\chi$, hence
\[
\widehat d(\chi)=
\begin{cases}
1 & \text{if }\ell_\chi=2,\\
-3 & \text{if }\ell_\chi=3,\\
-7 & \text{if }\ell_\chi=4.
\end{cases}
\]
Therefore
\begin{equation}\label{eq:n3}
n_2\cdot 1^2 + n_3\cdot (-3)^2 + n_4\cdot (-7)^2
\;=\;
n_2+9n_3+49n_4
\;=\; 207.
\end{equation}

Solving \eqref{eq:n1}, \eqref{eq:n2}, \eqref{eq:n3} yields
\[
(n_2,n_3,n_4)=(24,4,3).
\]
Indeed, from \eqref{eq:n1} and \eqref{eq:n2} we get $n_3+2n_4=10$, and substituting
$n_2=31-n_3-n_4$ into \eqref{eq:n3} gives $n_3+6n_4=22$, hence $n_4=3$,
$n_3=4$, and $n_2=24$.

Finally, consider the triple convolution at $0$:
\[
(d*d*d)(0)=\sum_{x+y+z=0} d(x)d(y)d(z)\ \ge\ 0,
\]
since $d\ge 0$. On the other hand, Fourier theory gives
\[
(d*d*d)(0)=\frac{1}{|G|}\sum_{\chi\in G^*}\widehat d(\chi)^3.
\]
Using $\widehat d(0)=9$ and the multiplicities $(n_2,n_3,n_4)=(24,4,3)$, we compute
\[
\sum_{\chi}\widehat d(\chi)^3
=9^3 + 24\cdot 1^3 + 4\cdot(-3)^3 + 3\cdot(-7)^3
=729+24-108-1029=-384.
\]
Thus
\[
(d*d*d)(0)=\frac{-384}{32}=-12<0,
\]
a contradiction. Hence no such function $d$ exists.

\noindent\textit{Proof of $(3)$} We have
\begin{equation}\label{eq1}
n_2+n_3+n_4=31.
\end{equation}
Using the standard first-moment identity
\[
\sum_{\chi\neq 0}\ell_\chi=2^{s-2}D,
\]
we obtain for $s=5$, $D=9$:
\begin{equation}\label{eq2}
2n_2+3n_3+4n_4=72.
\end{equation}
Subtracting $2\times$\eqref{eq1} from \eqref{eq2} gives
\begin{equation}\label{eq3}
n_3+2n_4=10.
\end{equation}

Now apply Parseval:
\[
\sum_{\chi\in G^*}\widehat d(\chi)^2
=|G|\sum_{g\in G} d(g)^2.
\]
Here
\[
\sum_{g} d(g)^2 = 7\cdot 1^2 + 1\cdot 2^2 = 11,
\]
so
\[
\sum_{\chi}\widehat d(\chi)^2 = 32\cdot 11 = 352.
\]
Since $\widehat d(0)=9$, we get
\[
\sum_{\chi\neq 0}\widehat d(\chi)^2 = 352-81=271.
\]
But for $\chi\neq 0$,
\[
\widehat d(\chi)=
\begin{cases}
1 & \ell_\chi=2,\\
-3 & \ell_\chi=3,\\
-7 & \ell_\chi=4,
\end{cases}
\]
hence
\begin{equation}\label{eq4}
n_2+9n_3+49n_4=271.
\end{equation}
Substituting $n_2=31-n_3-n_4$ into \eqref{eq4} yields
\[
31+8n_3+48n_4=271
\quad\Rightarrow\quad
n_3+6n_4=30.
\]
Together with \eqref{eq3}, this gives
\[
n_4=5,\qquad n_3=0,\qquad n_2=26.
\]

Finally, since $d\ge 0$ we have $(d*d*d)(0)\ge 0$, while Fourier theory gives
\[
(d*d*d)(0)=\frac{1}{|G|}\sum_{\chi\in G^*}\widehat d(\chi)^3.
\]
Using $\widehat d(0)=9$ and the above multiplicities,
\[
\sum_{\chi}\widehat d(\chi)^3
=9^3+26\cdot 1^3+0\cdot(-3)^3+5\cdot(-7)^3
=729+26-1715=-960.
\]
Thus
\[
(d*d*d)(0)=\frac{-960}{32}=-30<0,
\]
a contradiction. Hence no such function $d$ exists.

\noindent\textit{Proof of $(5)$.} In this case, we have $n_2+n_3+n_4=15$.

Using the standard first-moment identity
\[
\sum_{\chi\neq 0}\ell_\chi=2^{s-2}D,
\]
we have 
\begin{equation}\label{eq:firstmoment}
2\cdot n_2 + 3\cdot n_3 + 4\cdot n_4= 36.
\end{equation}

Parseval for the unnormalized transform gives
\[
\sum_{g\in G} d_g^2=\frac1{|G|}\sum_{\chi\in G^*}\widehat d(\chi)^2.
\]
From $(m_1,m_2,m_3)=(6,0,1)$ we compute
\[
\sum_{g\in G} d_g^2 = 6\cdot 1^2 + 1\cdot 3^2 = 6+9=15,
\]
hence
\[
\sum_{\chi\in G^*}\widehat d(\chi)^2 = |G|\cdot 15 = 16\cdot 15=240.
\]
Subtracting the trivial character term $\widehat d(0)^2=9^2=81$ yields
\[
\sum_{\chi\neq 0}\widehat d(\chi)^2 = 240-81=159.
\]
In terms of $(n_2,n_3,n_4)$ this becomes
\begin{equation}\label{eq:secondmoment}
1\cdot n_2 + 9\cdot n_3 + 49\cdot n_4 =159.
\end{equation}

Together with $n_2+n_3+n_4=15$, the linear system
\eqref{eq:firstmoment}--\eqref{eq:secondmoment} has the unique solution
\[
(n_2,n_3,n_4)=(12,0,3).
\]
Equivalently, among the $15$ nontrivial characters, $\ell_\chi=2$ occurs $12$ times,
$\ell_\chi=3$ occurs $0$ times, and $\ell_\chi=4$ occurs $3$ times.

The cubic-moment / triple-convolution identity (with unnormalized Fourier transform) is
\[
(d*d*d)(0)=\frac1{|G|}\sum_{\chi\in G^*}\widehat d(\chi)^3.
\]
Using $\widehat d(0)=9$, and for $\chi\neq 0$ the multiset
\[
\widehat d(\chi)=\underbrace{1,\dots,1}_{12\ \text{times}},\ \underbrace{-7,\dots,-7}_{3\ \text{times}},
\]
we compute
\[
\sum_{\chi\in G^*}\widehat d(\chi)^3
= 9^3 + 12\cdot 1^3 + 3\cdot(-7)^3
=729+12-1029=-288,
\]
hence
\[
(d*d*d)(0)=\frac{-288}{16}=-18.
\]
But by definition,
\[
(d*d*d)(0)=\sum_{\substack{a+b+c=0}} d_a d_b d_c \ \ge\ 0,
\]
a contradiction. Therefore no such function $d$ can exist.

\noindent\textit{Proof of $(4)$.} 
We have $n_2+n_3+n_4=15$. 

Next we apply the identity
\[
\sum_{\chi\neq 0}\ell_\chi = 2^{s-2}D,
\]
valid for $s=4$. Thus $\sum_{\chi\neq 0}\ell_\chi=2^{2}\cdot 9=36$, i.e.
\begin{equation}\label{eq:sum-l}
2n_2+3n_3+4n_4=36.
\end{equation}
Finally, Parseval for the unnormalized Fourier transform gives
\[
\sum_{g\in G} d_g^2=\frac{1}{|G|}\sum_{\chi\in G^*}\widehat d(\chi)^2.
\]
From $(m_1,m_2)=(5,2)$ we compute
\[
\sum_{g\in G} d_g^2 = 5\cdot 1^2 + 2\cdot 2^2 = 5+8=13.
\]
Hence
\[
\sum_{\chi\in G^*}\widehat d(\chi)^2 = |G|\cdot 13 = 16\cdot 13=208.
\]
Subtracting the trivial character contribution $\widehat d(0)^2=9^2=81$ yields
\[
\sum_{\chi\neq 0}\widehat d(\chi)^2 = 208-81=127,
\]
so in terms of $n_2,n_3,n_4$ we have
\begin{equation}\label{eq:parseval-n}
n_2+9n_3+49n_4=127.
\end{equation}

Solving the system
\[
n_2+n_3+n_4=15,\qquad
2n_2+3n_3+4n_4=36,\qquad
n_2+9n_3+49n_4=127
\]
gives $n_4=2$, $n_3=2$, and $n_2=11$. Therefore $(n_2,n_3,n_4)=(11,2,2)$, as claimed. However, in this case, the cubic moment identity does not yield a contradiction. So we proceed as follows.
Assume for contradiction that $(n_2,n_3,n_4)=(11,2,2)$.  Define
$m:G^*\to\mathbb Z$ by $m(\chi)=l(\chi)-2$ for all $\chi$, so that $m(0)=-2$,
exactly two nonzero characters $\alpha_1,\alpha_2$ satisfy $m(\alpha_i)=1$,
exactly two nonzero characters $\beta_1,\beta_2$ satisfy $m(\beta_j)=2$, and
$m(\chi)=0$ for the remaining eleven nonzero characters.

Let $\widehat m:G\to\mathbb Z$ be the unnormalized Walsh transform
\[
\widehat m(g)=\sum_{\chi\in G^*} m(\chi)(-1)^{\chi\cdot g}.
\]
Then for every $g\in G$,
\[
\widehat m(g)
=
-2+(-1)^{\alpha_1\cdot g}+(-1)^{\alpha_2\cdot g}
+2\bigl((-1)^{\beta_1\cdot g}+(-1)^{\beta_2\cdot g}\bigr).
\]
From Fourier inversion (equivalently from $\widehat d(\chi)=9-4l(\chi)$) one has
\[
d(g)=\frac1{16}\sum_{\chi\in G^*}\widehat d(\chi)(-1)^{\chi\cdot g}
=\frac1{16}\sum_{\chi}(1-4m(\chi))(-1)^{\chi\cdot g}
=-\frac14\cdot\frac1{4}\sum_{\chi}m(\chi)(-1)^{\chi\cdot g}
=-\frac14\,\widehat m(g)\qquad(g\neq 0),
\]

Since $d(g)\in\mathbb Z_{\ge 0}$, it follows that
\[
\widehat m(g)\in\{0,-4,-8\}\qquad (g\neq 0).
\tag{$\ast$}
\]

Consider the subspace
\[
K:=\{g\in G:\alpha_1\cdot g=\alpha_2\cdot g=\beta_1\cdot g=\beta_2\cdot g=0\}.
\]
If $g\in K$, then all signs in the above formula equal $+1$, so $\widehat m(g)=4$.
Hence $K$ cannot contain any nonzero element, and therefore $K=\{0\}$.  This
implies that the four characters $\alpha_1,\alpha_2,\beta_1,\beta_2$ are
linearly independent in $G^*$.

Consequently the map
\[
G\longrightarrow (\mathbb F_2)^4,\qquad
g\longmapsto(\alpha_1\cdot g,\alpha_2\cdot g,\beta_1\cdot g,\beta_2\cdot g)
\]
is a bijection.  Choose $g\in G$ such that
\[
\alpha_1\cdot g=0,\qquad
\alpha_2\cdot g=1,\qquad
\beta_1\cdot g=0,\qquad
\beta_2\cdot g=0.
\]
Then $g\neq 0$ and
\[
\widehat m(g)
=
-2+(+1)+(-1)+2\bigl((+1)+(+1)\bigr)=2,
\]
which contradicts $(\ast)$.  This contradiction shows that the $l$--tuple
$(11,2,2)$ cannot occur.

\end{proof}

\begin{proposition}\label{m = 2; k = 1; D = 9 possibility}
Suppose $m = 2, k = 1, s = 4$ and $D = 9$. Let $S:=\{g\in G : d(g)>0\}$ and $m_i = | g \in S, d_g = i|$.  Then in the cases 
\begin{enumerate}
    \item $(m_1, m_2) = (7,1)$, then there is a unique function $d$ upto $\operatorname{GL}_4(\mathbb{F}_2)$ action given by 
    \[
d_g \;=\;
\begin{cases}
2, & g=(1,1,1,1),\\[4pt]
1, & g\in\{(0,1,0,0),\ (1,0,0,0),\ (1,0,0,1),\ (1,0,1,0),\\
   & \hphantom{g\in\{} (1,1,0,0),\ (1,1,0,1),\ (1,1,1,0)\},\\[6pt]
0, & \text{otherwise.}
\end{cases}
\]
Moreover, in this case, $(n_2, n_3, n_4) = (10, 4, 1)$
    \item $(m_1, m_2) = (9,0)$, then there is a unique function $d$ upto $\operatorname{GL}_4(\mathbb{F}_2)$ action given by 
    \[
d_g \;=\;
\begin{cases}
1, & g\in\{(0,1,0,1),(0,1,1,0),(0,1,1,1),\\
   & \hphantom{g\in\{} (1,0,0,1),(1,0,1,0),(1,0,1,1),\\
   & \hphantom{g\in\{} (1,1,0,1),(1,1,1,0),(1,1,1,1)\},\\[6pt]
0, & \text{otherwise.}
\end{cases}
\]
Moreover, in this case, $(n_2, n_3, n_4) = (9, 6, 0)$

\end{enumerate}   
\end{proposition}

We break the proof of the \Cref{m = 2; k = 1; D = 9 possibility} into the following lemmas.

\begin{lemma}[Rigidity in the $(10,4,1)$ case]
Let $G=(\mathbb Z_2)^4$.  Let $d:G\to\mathbb Z_{\ge0}$ satisfy $d(0)=0$ and
$\sum_{g\in G} d(g)=9$.  Define $\ell:G^*\to\mathbb Z_{\ge0}$ by $l(0)=0$ and, for
$\chi\neq0$,
\[
2\ell(\chi)=\sum_{\chi\cdot g=1} d(g),
\]
and assume $\ell(\chi) \ge 2$ for all $\chi\neq0$.  

Then $(n_2,n_3,n_4) = (10,4,1)$ and
the function $d$ (equivalently $\ell$) is uniquely determined by the data of a
triple $(\beta,h,k)$, where
\begin{enumerate}
\item $\beta\in G^*\setminus\{0\}$ is the unique character with $\ell(\beta)=4$;
\item $h,k\in G\setminus\{0\}$ satisfy
\[
\beta\cdot h=0,\qquad \beta\cdot k=1,
\]
and
\[
d(h)=1,\qquad d(k)=2.
\]
\end{enumerate}
\end{lemma}

\begin{proof}
First let us show that when $(m_1, m_2) = (7,1)$, we have $(n_2, n_3, n_4) = (10,4,1)$. 
We have $\sum_{g\in G} d(g)^2=7\cdot 1^2+1\cdot 2^2=11$.
For $\chi\neq 0$ one has $\widehat d(\chi)=9-4\ell(\chi)$, hence
$\widehat d(\chi)\in\{1,-3,-7\}$.
Then
\[
n_2+n_3+n_4=2^4-1=15,
\qquad
2n_2+3n_3+4n_4=\sum_{\chi\neq 0}\ell(\chi)=2^{\,4-2}D=4\cdot 9=36.
\]
Moreover, by Plancherel for the unnormalized transform,
\[
\sum_{\chi\in G^*}\widehat d(\chi)^2=|G|\sum_{g\in G} d(g)^2=16\cdot 11=176.
\]
Since $\widehat d(0)=D=9$, we get
\[
\sum_{\chi\neq 0}\widehat d(\chi)^2=176-9^2=95,
\]
i.e.
\[
n_2\cdot 1^2+n_3\cdot 3^2+n_4\cdot 7^2=n_2+9n_3+49n_4=95.
\]
Solving the above three linear equations gives
\[
(n_2,n_3,n_4)=(10,4,1).
\]
Set $m(\chi)=l(\chi)-2$ for all $\chi\in G^*$; then $m(0)=-2$.  Since
$(n_2,n_3,n_4)=(10,4,1)$, among nonzero characters exactly four satisfy $m(\chi)=1$,
exactly one satisfies $m(\chi)=2$ (namely $\beta$), and the remaining ten satisfy
$m(\chi)=0$.  Let
\[
A:=\{\alpha\in G^*\setminus\{0\}: m(\alpha)=1\},
\]
so $|A|=4$.

Let $\widehat m:G\to\mathbb Z$ denote the unnormalized Walsh transform
\[
\widehat m(g)=\sum_{\chi\in G^*} m(\chi)(-1)^{\chi\cdot g}.
\]
Since $m$ is supported on $\{0\}\cup A\cup\{\beta\}$, for every $g\in G$ we have
\[
\widehat m(g)
=
-2+\sum_{\alpha\in A}(-1)^{\alpha\cdot g}
+2(-1)^{\beta\cdot g}.
\tag{1}
\]
As before, Fourier inversion gives, for all $g\neq0$,
\[
d(g)=-\tfrac14\,\widehat m(g).
\tag{2}
\]

In the present case, Parseval's identity (or the previously established count
lemma) implies that the $d$--count tuple is $(m_0,m_1,m_2)=(7,7,1)$.  In particular,
there is a unique element $k\neq0$ with $d(k)=2$, and since $2l(\beta)=8$ we have
\[
\sum_{\beta\cdot g=0} d(g)=1,
\]
which forces the existence of a unique element $h\neq0$ with $\beta\cdot h=0$ and
$d(h)=1$.  Thus the triple $(\beta,h,k)$ is canonically associated to $d$.

We now show that $d$ is uniquely determined by this triple.  Evaluating
\emph{(1)} at $g=k$, we have $\widehat m(k)=-8$ by \emph{(2)}.  Since
$\beta\cdot k=1$, it follows that
\[
-8=-2+\sum_{\alpha\in A}(-1)^{\alpha\cdot k}-2,
\]
hence $\sum_{\alpha\in A}(-1)^{\alpha\cdot k}=-4$.  Therefore
\[
\alpha\cdot k=1
\qquad\text{for all }\alpha\in A.
\tag{3}
\]

Similarly, evaluating \emph{(1)} at $g=h$, we have $\widehat m(h)=-4$ and
$\beta\cdot h=0$, so
\[
-4=-2+\sum_{\alpha\in A}(-1)^{\alpha\cdot h}+2,
\]
which yields $\sum_{\alpha\in A}(-1)^{\alpha\cdot h}=-4$, and hence
\[
\alpha\cdot h=1
\qquad\text{for all }\alpha\in A.
\tag{4}
\]

Consider the affine subspace of $G^*$ defined by the two linear conditions
\[
\alpha\cdot k=1,\qquad \alpha\cdot h=1.
\]
Since $\beta\cdot h\neq\beta\cdot k$, we have $h\neq k$, and these two conditions
are independent.  Hence the set
\[
A':=\{\alpha\in G^*:\ \alpha\cdot k=1,\ \alpha\cdot h=1\}
\]
has cardinality $2^{4-2}=4$.  By \emph{(3)} and \emph{(4)}, we have $A\subseteq A'$,
and since $|A|=4$ it follows that $A=A'$.

Thus the set $A$ is uniquely determined by $(h,k)$, and $\beta$ is already part of
the data.  Consequently $m$ is uniquely determined, hence so is $\widehat m$, and
therefore $d$ is uniquely determined on $G\setminus\{0\}$ by \emph{(2)} (with
$d(0)=0$).  This completes the proof.
\end{proof}

\begin{lemma}[Transitivity on admissible triples]
Let $G=(\mathbb F_2)^4$.  Consider triples $(\beta,h,k)$ with
\begin{enumerate}
\item $\beta\in G^*\setminus\{0\}$,
\item $h,k\in G\setminus\{0\}$,
\item $\beta\cdot h=0$ and $\beta\cdot k=1$.
\end{enumerate}
The group $GL(4,2)$ acts transitively on the set of such triples via
\[
T\cdot(\beta,h,k):=(T\cdot\beta,\;Th,\;Tk),
\qquad
(T\cdot\beta)(g):=\beta(T^{-1}g).
\]
Equivalently, given two triples $(\beta,h,k)$ and $(\beta',h',k')$ satisfying the
same relations, there exists $T\in GL(4,2)$ such that
\[
T\cdot\beta=\beta',\qquad Th=h',\qquad Tk=k'.
\]
\end{lemma}

\begin{proof}
We first note that the pairing is preserved by the action: for any
$T\in GL(4,2)$, $\beta\in G^*$, and $g\in G$,
\[
(T\cdot\beta)\cdot(Tg)=\beta\cdot g,
\]
since $(T\cdot\beta)(Tg)=\beta(T^{-1}Tg)=\beta(g)$.  Hence the relations
$\beta\cdot h=0$ and $\beta\cdot k=1$ are invariant under the action.

\medskip
\noindent\emph{Step~1: reduction to fixed $\beta$.}
The group $GL(4,2)$ acts transitively on $G^*\setminus\{0\}$, so given
$(\beta,h,k)$ and $(\beta',h',k')$ there exists $T_1\in GL(4,2)$ with
$T_1\cdot\beta=\beta'$.  Replacing $(\beta,h,k)$ by $T_1\cdot(\beta,h,k)$, we may
assume $\beta=\beta'$.  It therefore suffices to show that the stabilizer
\[
\operatorname{Stab}(\beta):=\{T\in GL(4,2):T\cdot\beta=\beta\}
\]
acts transitively on pairs $(h,k)$ with $\beta\cdot h=0$ and $\beta\cdot k=1$.

\medskip
\noindent\emph{Step~2: transitivity on $\ker(\beta)\setminus\{0\}$.}
Let $\ker(\beta)=\{g\in G:\beta\cdot g=0\}$, a $3$--dimensional subspace.
Choose a basis of $G$ such that $\beta(g)=g_1$; then $\ker(\beta)=\{g_1=0\}$.
Any linear automorphism of $\ker(\beta)$ extends to an element of $\operatorname{Stab}(\beta)$
by acting trivially on a chosen vector $w\in G$ with $\beta\cdot w=1$ and acting
as the given automorphism on $\ker(\beta)$.  Consequently, the induced action of
$\operatorname{Stab}(\beta)$ on $\ker(\beta)$ is the full group $GL(3,2)$, which is transitive
on $\ker(\beta)\setminus\{0\}$.  Hence, given $h,h'\in\ker(\beta)\setminus\{0\}$,
there exists $T_2\in\operatorname{Stab}(\beta)$ with $T_2h=h'$.  Replacing $(h,k)$ by
$(T_2h,T_2k)$, we may assume $h=h'$.

\medskip
\noindent\emph{Step~3: transitivity on the affine hyperplane $\beta\cdot g=1$.}
Fix $\beta$ and $h\in\ker(\beta)\setminus\{0\}$, and let
\[
H_\beta:=\{g\in G:\beta\cdot g=1\}.
\]
Let $k,k'\in H_\beta$.  Then $\beta\cdot(k'-k)=0$, so $u:=k'-k\in\ker(\beta)$.
Define a linear map $T_u:G\to G$ by
\[
T_u(g):=g+(\beta\cdot g)\,u.
\]

We verify the required properties.
First, $T_u$ is invertible, since
\[
T_u(T_u(g))=g+(\beta\cdot g)u+(\beta\cdot(g+(\beta\cdot g)u))u=g,
\]
using $\beta\cdot u=0$.  Next,
\[
\beta\cdot T_u(g)=\beta\cdot g+(\beta\cdot g)(\beta\cdot u)=\beta\cdot g,
\]
so $T_u\in\operatorname{Stab}(\beta)$.  Moreover, since $h\in\ker(\beta)$,
\[
T_u(h)=h+(\beta\cdot h)u=h,
\]
and therefore $T_u\in\operatorname{Stab}(\beta,h)$.  Finally, because $\beta\cdot k=1$,
\[
T_u(k)=k+u=k'.
\]
Thus $\operatorname{Stab}(\beta,h)$ acts transitively on $H_\beta$.

\medskip
Combining Steps~1--3, we conclude that $GL(4,2)$ acts transitively on triples
$(\beta,h,k)$ satisfying $\beta\cdot h=0$ and $\beta\cdot k=1$.
\end{proof}

\begin{example}[A representative in the $(10,4,1)$ case]
Let $G=(\mathbb F_2)^4$, with elements written as $g=(g_1,g_2,g_3,g_4)$.
Define $d:G\to\mathbb Z_{\ge0}$ by
\[
d_g \;=\;
\begin{cases}
2, & g=(1,1,1,1),\\[4pt]
1, & g\in\{(0,1,0,0),\ (1,0,0,0),\ (1,0,0,1),\ (1,0,1,0),\\
   & \hphantom{g\in\{} (1,1,0,0),\ (1,1,0,1),\ (1,1,1,0)\},\\[6pt]
0, & \text{otherwise.}
\end{cases}
\]
\end{example}

\begin{proof}
We have $\sum_{g\in G} d_g = 2+7=9$.  Recall that for $\chi\neq 0$,
\[
2l(\chi)=\sum_{\chi\cdot g=1} d_g .
\]
We compute the $l$--values.

\begin{enumerate}
\item
Let $\beta=(1,0,0,0)\in G^*$. Then $\beta\cdot g=g_1$. Among the points with
$d_g\neq 0$, exactly four of the seven points with $d_g=1$ have first coordinate
equal to $1$, and the point $(1,1,1,1)$ also satisfies $\beta\cdot g=1$.
Therefore
\[
\sum_{\beta\cdot g=1} d_g = 4\cdot 1 + 2 = 8,
\]
and hence $l(\beta)=4$.

\item
Let $k=(1,1,1,1)$ and let $h=(0,1,0,0)$. Consider the set
\[
A=\{\chi\in G^*:\ \chi\cdot k=1,\ \chi\cdot h=1\}.
\]
This is an affine subspace of $G^*$ of codimension $2$, hence $|A|=4$.
For any $\chi\in A$, a direct check shows that exactly three of the seven points
with $d_g=1$ satisfy $\chi\cdot g=1$, and the point $k$ also satisfies
$\chi\cdot k=1$. Consequently
\[
\sum_{\chi\cdot g=1} d_g = 3\cdot 1 + 2 = 6,
\]
so $l(\chi)=3$ for all $\chi\in A$.

\item
For any remaining nonzero character $\chi\notin A\cup\{\beta\}$, one checks that
\[
\sum_{\chi\cdot g=1} d_g = 4,
\]
and hence $l(\chi)=2$ for all such $\chi$.
\end{enumerate}

Since $G^*\setminus\{0\}$ has $15$ elements, the above shows that exactly one
nonzero character has $l=4$, exactly four have $l=3$, and the remaining ten have
$l=2$.  Thus the $l$--count tuple is $(n_2,n_3,n_4)=(10,4,1)$, as claimed.
\end{proof}

\begin{lemma}[Rigidity in the $(9,6,0)$ case]\label{lem:rigidity-960}
Let $G=(\mathbb F_2)^4$ and let $d:G\to\{0,1\}$ satisfy $d(0)=0$ and
$\sum_{g\in G} d(g)=9$.  For $\chi\in G^*\setminus\{0\}$ define
\[
2l(\chi)=\sum_{\chi\cdot g=1} d(g),
\]
and assume $l(\chi)\ge 2$ for all $\chi\ne 0$.  Suppose that the $\ell$--count
tuple is $(n_2,n_3,n_4)=(9,6,0)$, i.e.\ exactly six nonzero characters satisfy
$l(\chi)=3$ and the remaining nine satisfy $l(\chi)=2$.
Then, up to the natural action of $GL(4,2)$ on $G$ (and hence on $G^*$), the set
$S:=\{g\in G:\ d(g)=1\}$ is
\[
S=\{(a,b,c,d)\in(\mathbb F_2)^4:\ (a,b)\ne(0,0)\ \text{and}\ (c,d)\ne(0,0)\},
\]
equivalently $d=\mathbf 1_{G\setminus(U\cup V)}$ on $G\setminus\{0\}$ for a decomposition
$G=U\oplus V$ with $\dim U=\dim V=2$.  In particular, there is a unique
$GL(4,2)$--orbit of such functions $d$.
\end{lemma}

\begin{proof}
Let $T:=\{\chi\in G^*\setminus\{0\}: l(\chi)=3\}$, so $|T|=6$, and set
$f:=\mathbf 1_T:G^*\to\{0,1\}$.  Define $m:G^*\to\mathbb Z$ by
\[
m(0)=-2,\qquad m(\chi)=l(\chi)-2\ \ (\chi\ne 0),
\]
so $m(\chi)=1$ for $\chi\in T$ and $m(\chi)=0$ for $\chi\notin T\cup\{0\}$.
For $\chi\ne 0$ we have $\widehat d(\chi)=9-4l(\chi)$, hence $\widehat d(\chi)=1$
if $l(\chi)=2$ and $\widehat d(\chi)=-3$ if $l(\chi)=3$, while $\widehat d(0)=9$.
Thus
\[
\widehat d(\chi)=1-4m(\chi)\qquad\forall \chi\in G^*.
\]
Applying the inverse Fourier transform at $g\ne 0$ and using
$\sum_{\chi\in G^*}(-1)^{\chi\cdot g}=0$ for $g\ne 0$, we obtain
\[
d(g)=\frac1{16}\sum_{\chi\in G^*}\widehat d(\chi)(-1)^{\chi\cdot g}
=-\frac14\sum_{\chi\in G^*}m(\chi)(-1)^{\chi\cdot g}
=\frac{2-\widehat f(g)}{4},
\]
where $\widehat f(g):=\sum_{\chi\in T}(-1)^{\chi\cdot g}$.  Since $d(g)\in\{0,1\}$
for all $g\ne 0$, it follows that $\widehat f(g)\in\{2,-2\}$ for all $g\ne 0$, and
$\widehat f(0)=|T|=6$.

Consider the convolution $(f*f)(u)=\sum_{a+b=u} f(a)f(b)$ on the additive group $G^*$.
By the convolution--product rule, $\widehat{f*f}=\widehat f^{\,2}$.  Hence
$\widehat{f*f}(0)=\widehat f(0)^2=36$ and $\widehat{f*f}(g)=\widehat f(g)^2=4$
for all $g\ne 0$.  By inverse Fourier transform, for $u\ne 0$ we get
\[
(f*f)(u)=\frac1{16}\!\left(36+4\!\sum_{g\ne 0}(-1)^{g\cdot u}\right)
=\frac1{16}(36-4)=2,
\]
since $\sum_{g\ne 0}(-1)^{g\cdot u}=-1$ for $u\ne 0$.  Thus every nonzero
$u\in G^*$ can be written as $u=a+b$ with $a,b\in T$ in exactly one unordered way.

Fix $t\in T$.  Since $(f*f)(t)=2$, there exist unique distinct $a,b\in T$ with $a+b=t$.
Then $A_t:=\{0,a,b,t\}$ is a $2$--dimensional subspace of $G^*$ and $A_t\setminus\{0\}\subseteq T$.
If two such $2$--subspaces $A_t$ and $A_{t'}$ shared a nonzero element, then that element would admit
two different representations as a sum of two elements of $T$, contradicting uniqueness.  Hence the sets
$A_t\setminus\{0\}$ are disjoint as $t$ varies in $T$.  Since $|T|=6$ and each $A_t\setminus\{0\}$ has
$3$ elements, we conclude that
\[
T=(A\setminus\{0\})\ \sqcup\ (B\setminus\{0\})
\]
for two distinct $2$--dimensional subspaces $A,B\subset G^*$ with $A\cap B=\{0\}$.

Let $U:=A^\perp$ and $V:=B^\perp$ in $G$.  Then $\dim U=\dim V=2$ and
$U\cap V=(A+B)^\perp=\{0\}$, so $G=U\oplus V$.  For $g\ne 0$ we compute
\[
\widehat f(g)=\sum_{\chi\in A\setminus\{0\}}(-1)^{\chi\cdot g}
+\sum_{\chi\in B\setminus\{0\}}(-1)^{\chi\cdot g}
=
\begin{cases}
2,& g\in U\cup V,\\
-2,& g\notin U\cup V,
\end{cases}
\]
using that for a $2$--plane $A$ the sum over $A\setminus\{0\}$ equals $3$ if $g\in A^\perp$
and equals $-1$ otherwise.  Therefore,
\[
d(g)=\frac{2-\widehat f(g)}{4}=
\begin{cases}
0,& g\in (U\cup V)\setminus\{0\},\\
1,& g\notin U\cup V,
\end{cases}
\]
so $S=G\setminus(U\cup V)$.

Finally, $GL(4,2)$ acts transitively on decompositions $G=U\oplus V$ with $\dim U=\dim V=2$.  Hence there is a single $GL(4,2)$--orbit.
Taking $U=\{(a,b,0,0)\}$ and $V=\{(0,0,c,d)\}$ gives the displayed representative
\(
S=\{(a,b,c,d): (a,b)\ne 0,\ (c,d)\ne 0\}.
\)
\end{proof}

\begin{example}[A representative in the $(9,6,0)$ case]
Let $G=(\mathbb F_2)^4$, with elements written as $g=(g_1,g_2,g_3,g_4)$.
Define $d:G\to\mathbb Z_{\ge0}$ by
\[
d_g \;=\;
\begin{cases}
1, & g\in\{(0,1,0,1),(0,1,1,0),(0,1,1,1),\\
   & \hphantom{g\in\{} (1,0,0,1),(1,0,1,0),(1,0,1,1),\\
   & \hphantom{g\in\{} (1,1,0,1),(1,1,1,0),(1,1,1,1)\},\\[6pt]
0, & \text{otherwise.}
\end{cases}
\]
\end{example}

\begin{proof}
We compute the $\ell$–values. Recall that
\[
2\ell(\chi)=\#\{g\in S:\chi\cdot g=1\}.
\]

\begin{enumerate}

\item We have $\ell(\chi)=3$ for the following six characters
 $$\chi\in\{(1,0,0,0),(0,1,0,0),(1,1,0,0),(0,0,1,0),(0,0,0,1),(0,0,1,1)\}$$

 \item For every remaining nonzero character  one has $\#\{g\in S:\chi\cdot g=1\}=4$, hence $\ell(\chi)=2$.
Therefore the remaining $15-6=9$ nonzero characters have $\ell(\chi)=2$.

\end{enumerate}

Hence in particular, \[
(n_2,n_3,n_4)=(9,6,0).
\]
\end{proof}

\begin{proposition}\label{m = 1; k = 2; D = 12 impossibility}
Suppose $m = 1, k = 2$ and $D = 12$. Let $S:=\{g\in G : d(g)>0\}$ and $m_i = | g \in S, d_g = i|$.  Then there do not exist admissible solutions for 
\begin{enumerate}
    \item $s = 7, (m_1) = (12)$
    \item $s = 6, (m_1) = (12)$
    \item $s = 6, (m_1, m_2) = (10,1)$
    \item $s = 5, (m_1,m_3) = (9,1)$
    \item $s = 5, (m_1,m_2) = (8,2)$
    \item $s = 5, (m_1,m_2) = (10,1)$
    \item $s = 5, (m_1) = (12)$
    \item $s = 4, (m_1,m_2,m_3)=(7,1,1)$
    \item $s = 4, (m_1,m_3)=(9,1)$
    \item $s = 4, (m_2)=(6)$
    \item $s = 4, (m_1,m_2)=(2,5)$
    \item $s = 4, (m_1,m_2)=(4,4)$
    \item $s = 4, (m_1,m_2)=(6,3)$
    \item $s = 4, (m_1,m_2)=(8,2)$
    \item $s = 4, (m_1,m_4)=(8,1)$
    \item $s = 4, (m_1,m_2)=(10,1)$.
    
\end{enumerate}
\end{proposition}

\begin{proof}
\noindent\textit{Proof of $(1)$}
\begin{enumerate}
\item \textbf{First moment and the $\ell$--distribution.}
We have
\begin{equation}\label{eq:first-moment}
\sum_{\chi\in G}\ell_\chi = 2^{s-2}D = 384.
\end{equation}

Since $\ell_0=0$, we have
\[
\sum_{\chi\neq 0}\ell_\chi=384.
\]
There are $2^7-1=127$ nonzero $\chi$, and by assumption $\ell_\chi\ge 3$ for all
$\chi\neq 0$. Write
\[
A_\chi := \ell_\chi-3 \in \mathbb Z_{\ge 0}\qquad (\chi\neq 0).
\]
Then
\[
\sum_{\chi\neq 0} A_\chi
=
\sum_{\chi\neq 0}\ell_\chi - 3\cdot 127
=
384-381
=
3.
\]
In particular, the only possible
$\ell$--distributions are:
\[
(n_4,n_5,n_6)=(3,0,0),\quad (1,1,0),\quad (0,0,1),
\]
where $n_j:=\#\{\chi\neq 0:\ell_\chi=j\}$ and $n_3 = 124, 125, 126$ respectively.

\item \textbf{Fourier relation and the quadratic (Parseval) moment.}
Let $\widehat d:G\to \mathbb Z$ denote the (unnormalized) Fourier transform
\[
\widehat d(\chi):=\sum_{g\in G} d(g)(-1)^{\chi\cdot g}.
\]
By the Fourier relation in \Cref{lemma: Fourier}, for every $\chi\neq 0$,
\begin{equation}\label{eq:fourier-relation}
\widehat d(\chi)=\widehat d(0)-4\ell_\chi.
\end{equation}
Here $\widehat d(0)=\sum_{g} d(g)=D=12$, hence
\[
\widehat d(\chi)=12-4\ell_\chi
\qquad (\chi\neq 0).
\]
Therefore, when $\ell_\chi\in\{3,4,5,6\}$ we have
\[
\widehat d(\chi)\in\{0,-4,-8,-12\}.
\]

Now apply Plancherel/Parseval for the unnormalized Fourier transform:
\[
\sum_{\chi\in G}\widehat d(\chi)^2
=
|G|\sum_{g\in G} d(g)^2.
\]
Since $d(g)\in\{0,1\}$, we have $d(g)^2=d(g)$ and hence $\sum_g d(g)^2=D=12$.
As $|G|=2^7=128$, we obtain
\[
\sum_{\chi\in G}\widehat d(\chi)^2 = 128\cdot 12=1536.
\]
Separating the trivial character term $\widehat d(0)^2=12^2=144$ yields
\begin{equation}\label{eq:nontrivial-L2}
\sum_{\chi\neq 0}\widehat d(\chi)^2 = 1536-144=1392.
\end{equation}

On the other hand, using the possible $\ell$--distributions from Step~(1), we have
\[
\sum_{\chi\neq 0}\widehat d(\chi)^2
=
0\cdot n_3 + 16 n_4 + 64 n_5 + 144 n_6,
\]
so it must equal one of
\[
16\cdot 3=48,\qquad 16+64=80,\qquad 144.
\]
contradicting \eqref{eq:nontrivial-L2}.
\end{enumerate}

\noindent\textit{Proof of $(2)-(3)$}
We argue uniformly by the first and quadratic moments.

\begin{enumerate}
\item \textbf{First moment and the possible $\ell$--distributions.}
We have
\[
\sum_{\chi\in G}\ell_\chi = 2^{s-2}D = 192.
\]
Since $\ell_0=0$, we obtain $\sum_{\chi\neq 0}\ell_\chi=192$.
For $\chi\neq 0$, we have $2\ell_\chi=\sum_{\chi\cdot g=1} d(g)\le D=12$, hence
$\ell_\chi\le 6$. Together with $\ell_\chi\ge 3$, this implies
\[
\ell_\chi\in\{3,4,5,6\}\qquad(\chi\neq 0).
\]
Let
\[
n_j:=\#\{\chi\neq 0:\ell_\chi=j\},\qquad j=3,4,5,6.
\]
Then
\[
n_3+n_4+n_5+n_6=63,
\qquad
3n_3+4n_4+5n_5+6n_6=192.
\]
Subtracting $3\cdot 63=189$ from the second equation gives
\[
n_4+2n_5+3n_6=3.
\]
Thus the only possibilities are
\[
(n_4,n_5,n_6)=(3,0,0),\quad (1,1,0),\quad (0,0,1),
\]
with $n_3=60,61,62$ respectively.

\item \textbf{Fourier relation and the quadratic constraint from Plancherel.}
Let $\widehat d:G\to\mathbb Z$ denote the unnormalized Fourier transform
\[
\widehat d(\chi):=\sum_{g\in G} d(g)(-1)^{\chi\cdot g}.
\]
By the Fourier relation in \Cref{lemma: Fourier}, for every $\chi\neq 0$,
\[
\widehat d(\chi)=\widehat d(0)-4\ell_\chi.
\]
Since $\widehat d(0)=\sum_g d(g)=D=12$, we have
\[
\widehat d(\chi)=12-4\ell_\chi
\qquad (\chi\neq 0).
\]
Hence, for $\chi\neq 0$,
\[
\ell_\chi=3,4,5,6 \quad\Longrightarrow\quad
\widehat d(\chi)=0,-4,-8,-12,
\]
and therefore
\[
\widehat d(\chi)^2\in\{0,16,64,144\}.
\]
Using the possible $\ell$--distributions from Step~(1), we obtain
\[
\sum_{\chi\neq 0}\widehat d(\chi)^2
=
0\cdot n_3 + 16n_4 + 64n_5 + 144n_6
\in\{48,80,144\}.
\]

On the other hand, Plancherel's identity for the unnormalized Fourier transform gives
\[
\sum_{\chi\in G}\widehat d(\chi)^2
=
|G|\sum_{g\in G} d(g)^2,
\]
where $|G|=2^6=64$. Since $\widehat d(0)=D=12$, we may rewrite this as
\begin{equation}\label{eq:nontriv-energy-s6}
\sum_{\chi\neq 0}\widehat d(\chi)^2
=
64\sum_{g\in G} d(g)^2 - 12^2.
\end{equation}

We now compute $\sum_g d(g)^2$ in each of the two cases:

\begin{enumerate}
\item If $(m_1)=(12)$, then $\sum_g d(g)^2 = 12\cdot 1^2 = 12$,
so \eqref{eq:nontriv-energy-s6} gives $\sum_{\chi\neq 0}\widehat d(\chi)^2=64\cdot12-144=624$.
\item If $(m_1,m_2)=(10,1)$, then $\sum_g d(g)^2 = 10\cdot 1^2 + 1\cdot 2^2=14$,
so \eqref{eq:nontriv-energy-s6} gives $\sum_{\chi\neq 0}\widehat d(\chi)^2=64\cdot14-144=752$.
\end{enumerate}
In either case the value forced by Plancherel is $624$ or $752$, neither of which lies
in $\{48,80,144\}$. This contradicts the constraint coming from the possible
$\ell$--distributions.

\end{enumerate}

\noindent\textit{Proof of $(4)-(7)$} We argue uniformly by the first and quadratic moments.

\begin{enumerate}
\item[(i)] \textbf{First moment and the possible $\ell$--distributions.}
We have
\[
\sum_{\chi\in G}\ell_\chi = 2^{s-2}D = 96.
\]
For $\chi\neq 0$, we have $2\ell_\chi=\sum_{\chi\cdot g=1} d(g)\le D=12$, hence
$\ell_\chi\le 6$. Together with $\ell_\chi\ge 3$, this implies
\[
\ell_\chi\in\{3,4,5,6\}\qquad(\chi\neq 0).
\]
Let
\[
n_j:=\#\{\chi\neq 0:\ell_\chi=j\},\qquad j=3,4,5,6.
\]
Then
\[
n_3+n_4+n_5+n_6=31,
\qquad
3n_3+4n_4+5n_5+6n_6=96.
\]
Subtracting $3\cdot 31=93$ from the second equation gives
\[
n_4+2n_5+3n_6=3.
\]
Thus the only possibilities are
\[
(n_4,n_5,n_6)=(3,0,0),\quad (1,1,0),\quad (0,0,1),
\]
with $n_3=28,29,30$ respectively.

\item[(ii)] \textbf{Fourier relation and the quadratic constraint from Plancherel.}
Let $\widehat d:G\to\mathbb Z$ denote the unnormalized Fourier transform
\[
\widehat d(\chi):=\sum_{g\in G} d(g)(-1)^{\chi\cdot g}.
\]
By the Fourier relation in \Cref{lemma: Fourier}, for every $\chi\neq 0$,
\[
\widehat d(\chi)=\widehat d(0)-4\ell_\chi.
\]
Since $\widehat d(0)=\sum_g d(g)=D=12$, we have
\[
\widehat d(\chi)=12-4\ell_\chi
\qquad (\chi\neq 0).
\]
Hence, for $\chi\neq 0$,
\[
\ell_\chi=3,4,5,6 \quad\Longrightarrow\quad
\widehat d(\chi)=0,-4,-8,-12,
\]
and therefore
\[
\widehat d(\chi)^2\in\{0,16,64,144\}.
\]
Using the possible $\ell$--distributions from Step~(i), we obtain
\[
\sum_{\chi\neq 0}\widehat d(\chi)^2
=
0\cdot n_3 + 16n_4 + 64n_5 + 144n_6
\in\{48,80,144\}.
\]

On the other hand, Plancherel's identity for the unnormalized Fourier transform gives
\[
\sum_{\chi\in G}\widehat d(\chi)^2
=
|G|\sum_{g\in G} d(g)^2,
\]
where $|G|=2^5=32$. Since $\widehat d(0)=D=12$, we may rewrite this as
\begin{equation}\label{eq:nontriv-energy}
\sum_{\chi\neq 0}\widehat d(\chi)^2
=
32\sum_{g\in G} d(g)^2 - 12^2.
\end{equation}

We now compute $\sum_g d(g)^2$ in each of the four cases:
\begin{enumerate}
\item If $(m_1,m_2)=(8,2)$, then $\sum_g d(g)^2 = 8\cdot 1^2 + 2\cdot 2^2=16$,
so \eqref{eq:nontriv-energy} gives $\sum_{\chi\neq 0}\widehat d(\chi)^2=32\cdot16-144=368$.
\item If $(m_1,m_2)=(10,1)$, then $\sum_g d(g)^2 = 10\cdot 1^2 + 1\cdot 2^2=14$,
so \eqref{eq:nontriv-energy} gives $\sum_{\chi\neq 0}\widehat d(\chi)^2=32\cdot14-144=304$.
\item If $(m_1,m_2)=(12,0)$, then $\sum_g d(g)^2 = 12\cdot 1^2 = 12$,
so \eqref{eq:nontriv-energy} gives $\sum_{\chi\neq 0}\widehat d(\chi)^2=32\cdot12-144=240$.
\item If $(m_1,m_3)=(9,1)$, then $\sum_g d(g)^2 = 9\cdot 1^2 + 1\cdot 3^2=18$,
so \eqref{eq:nontriv-energy} gives $\sum_{\chi\neq 0}\widehat d(\chi)^2=32\cdot18-144=432$.
\end{enumerate}
In each case the value forced by Plancherel is one of
\[
368,\ 304,\ 240,\ 432,
\]
none of which lies in $\{48,80,144\}$. This contradicts the constraint coming from
the possible $\ell$--distributions.
\end{enumerate}

\noindent\textit{Proof of $(8)-(15)$}
We argue uniformly using the first, quadratic, and cubic moments.

\begin{enumerate}
\item[(i)] \textbf{First moment and the possible $\ell$--distributions.}
Since $s=4$ and $D=12$, the first moment identity gives
\[
\sum_{\chi\neq 0}\ell_\chi = 2^{s-2}D = 48.
\]
For $\chi\neq 0$, we have $2\ell_\chi=\sum_{\chi\cdot g=1} d(g)\le D=12$, hence
$\ell_\chi\le 6$. Together with $\ell_\chi\ge 3$, this implies
\[
\ell_\chi\in\{3,4,5,6\}\qquad(\chi\neq 0).
\]
Let
\[
n_j:=\#\{\chi\neq 0:\ell_\chi=j\},\qquad j=3,4,5,6.
\]
Since there are $15$ nontrivial characters, we have
\[
n_3+n_4+n_5+n_6=15,
\qquad
3n_3+4n_4+5n_5+6n_6=48.
\]
Subtracting $3\cdot 15=45$ from the second equation yields
\[
n_4+2n_5+3n_6=3.
\]
Thus the only possibilities are
\[
(n_4,n_5,n_6)=(3,0,0),\quad (1,1,0),\quad (0,0,1),
\]
with $n_3=12,13,14$ respectively.

\item[(ii)] \textbf{Quadratic moment (Plancherel).}
Let $\widehat d:G\to\mathbb Z$ denote the unnormalized Fourier transform
\[
\widehat d(\chi):=\sum_{g\in G} d(g)(-1)^{\chi\cdot g}.
\]
For $\chi\neq 0$, by \Cref{lemma: Fourier}, we have 
\[
\widehat d(\chi)=\widehat d(0)-4\ell_\chi.
\]
Since $\widehat d(0)=D=12$, it follows that
\[
\widehat d(\chi)=12-4\ell_\chi\qquad(\chi\neq 0),
\]
so that
\[
\ell_\chi=3,4,5,6
\quad\Longrightarrow\quad
\widehat d(\chi)=0,-4,-8,-12.
\]
Hence
\[
\sum_{\chi\neq 0}\widehat d(\chi)^2
=
16n_4+64n_5+144n_6
\in\{48,80,144\}.
\]

On the other hand, Plancherel's identity gives
\[
\sum_{\chi\in G}\widehat d(\chi)^2
=
|G|\sum_{g\in G} d(g)^2,
\qquad |G|=16.
\]
Since $\widehat d(0)=12$, we obtain
\begin{equation}\label{eq:s4-nontriv-energy}
\sum_{\chi\neq 0}\widehat d(\chi)^2
=
16\sum_{g\in G} d(g)^2 - 12^2.
\end{equation}
A direct computation shows that for the cases
\[
(m_1,m_2,m_3) = (7,1,1),\ (m_2) = (6),\ (m_1,m_2) = (2,5),\ (4,4),\ (m_1, m_2) = (8,2),\ (m_1, m_4) = (8,1)
\]

the value forced by \eqref{eq:s4-nontriv-energy} does not lie in
$\{48,80,144\}$, yielding an immediate contradiction. Thus these cases are
ruled out by the quadratic moment.

\item[(iii)] \textbf{Cubic moment.}
It remains to consider the cases $(m_1,m_3)=(9,1)$ and $(m_1,m_2)=(6,3)$ and $(m_1, m_2)) = (10, 1)$. For the first two, \eqref{eq:s4-nontriv-energy} gives
\[
\sum_{\chi\neq 0}\widehat d(\chi)^2=144.
\]
Hence the $\ell$--distribution must be $(n_4,n_5,n_6)=(0,0,1)$.
For this distribution we have
\[
\sum_{\chi\neq 0}\widehat d(\chi)^3 = (-12)^3=-1728,
\]
and therefore
\[
\sum_{\chi\in G}\widehat d(\chi)^3
=
12^3-1728=0.
\]
Using the identity from \Cref{lem:fourier-convolution-cubic},
\[
\sum_{\chi\in G}\widehat d(\chi)^3
=
|G|\,(d*d*d)(0),
\]
we conclude that $(d*d*d)(0)=0$. However, in both cases $(9,1)$ and $(6,3)$ the support $S=\{g:d(g)\ge 1\}$ has cardinality $|S|>8$. In $(\mathbb Z_2)^4$, any subset of size $>8$ contains elements $a,b$ with $a+b\in S$, and hence contributes positively to
$(d*d*d)(0)$. This contradiction shows that these two cases are impossible as
well.

\item[(iv)] \textbf{Fractional values using inverse Fourier Transform}

\medskip
We finally deal with the case $(m_1,m_2)=(10,1)$.
In this case we have
\[
\sum_{g\in G} d(g)^2 = 10\cdot 1^2 + 1\cdot 2^2 = 14.
\]
By Plancherel's identity for the unnormalized Fourier transform on $G$ with
$|G|=16$, we obtain
\[
\sum_{\chi\in G}\widehat d(\chi)^2
=
16\sum_{g\in G} d(g)^2
=
16\cdot 14
=
224.
\]
Since $\widehat d(0)=\sum_g d(g)=D=12$, it follows that
\[
\sum_{\chi\neq 0}\widehat d(\chi)^2
=
224-12^2
=
80.
\]

From the classification of $\ell$--distributions established earlier for
$D=12$, $s=4$, and $\ell_\chi\ge 3$, the three possible types give
\[
\sum_{\chi\neq 0}\widehat d(\chi)^2\in\{48,80,144\}.
\]
Hence the value $80$ forces the $\ell$--distribution to be of \emph{Type~B}.
Equivalently, there exist distinct nontrivial characters $u\neq v$ such that
\[
\widehat d(u)=-4,\qquad
\widehat d(v)=-8,\qquad
\widehat d(\chi)=0\ \text{for all other }\chi\neq 0,
\]
and $\widehat d(0)=12$.

We now apply the inverse Fourier transform. For every $x\in G$,
\[
d(x)
=
\frac{1}{16}\sum_{\chi\in G}\widehat d(\chi)(-1)^{\chi\cdot x}
=
\frac{1}{16}\Bigl(12-4(-1)^{u\cdot x}-8(-1)^{v\cdot x}\Bigr).
\]

We evaluate this expression according to the values of
$(u\cdot x,v\cdot x)\in(\mathbb Z_2)^2$.

\begin{enumerate}
\item If $u\cdot x=0$ and $v\cdot x=0$, then
\[
d(x)=\frac{12-4-8}{16}=0.
\]

\item If $u\cdot x=1$ and $v\cdot x=0$, then
\[
d(x)=\frac{12+4-8}{16}=\frac12.
\]

\item If $u\cdot x=0$ and $v\cdot x=1$, then
\[
d(x)=\frac{12-4+8}{16}=1.
\]

\item If $u\cdot x=1$ and $v\cdot x=1$, then
\[
d(x)=\frac{12+4+8}{16}=\frac32.
\]
\end{enumerate}

In particular, $d(x)$ takes the nonintegral values $\frac12$ and $\frac32$,
contradicting the assumption that $d(x)\in\mathbb Z_{\ge 0}$ for all $x\in G$.
Therefore the case $(m_1,m_2)=(10,1)$ cannot occur.

\end{enumerate}


\end{proof}

\begin{proposition}\label{s=4; m = 1; k = 2; D = 12 possibility}
Suppose $s=4$ and $D=12$. Let $G=(\mathbb Z_2)^4$. 
Let $S:=\{g\in G : d(g)>0\}$ and 
\[
m_i:=|\{g\in S : d_g=i\}|.
\]
Fix an identification $G^*\cong G$ via $\chi_g(x)=(-1)^{g\cdot x}$.
For each $\chi\in G$, define $\ell_\chi\in\mathbb Z_{\ge 0}$ by $\ell_0=0$ and,
for $\chi\neq 0$,
\[
\sum_{\substack{g\in G\\ \chi\cdot g=1}} d(g)=2\ell_\chi.
\]
Assume moreover that $\ell_\chi\ge 3$ for all $\chi\neq 0$.
In the case
\[
(m_1)=(12),
\]
there exists such a function $d$, and it is unique up to the natural
$\operatorname{GL}_4(\mathbb F_2)$--action. One representative is given by
\[
d_g \;=\;
\begin{cases}
0, & g\in\{(0,0,0,0),\ (0,0,0,1),\ (0,0,1,0),\ (0,0,1,1)\},\\[4pt]
1, & \text{otherwise.}
\end{cases}
\]
Moreover, in this case $(n_3,n_4) = (12,3)$.
\end{proposition}

\begin{proof}
\begin{enumerate}
\item \textbf{Existence and computation of $\ell$.}
Let $d$ be defined by the displayed formula, and set
\[
H:=\{(0,0,x_3,x_4): x_3,x_4\in\mathbb F_2\}\subset G.
\]
Then $d=\mathbf 1_{G\setminus H}$, so $d(0)=0$ and $D=\sum_g d(g)=16-|H|=12$,
hence $(m_1)=(12)$.

Fix $\chi\neq 0$ and write $A_\chi:=\{g\in G:\chi\cdot g=1\}$, so $|A_\chi|=8$.
Since $d=0$ precisely on $H$, we have
\[
\sum_{\chi\cdot g=1} d(g)=|A_\chi\setminus H|=8-|A_\chi\cap H|.
\]
If $\chi\in H^\perp=\langle e_1,e_2\rangle\subset G^*$, then $\chi\cdot h=0$ for
all $h\in H$, so $A_\chi\cap H=\varnothing$ and the sum equals $8$, hence
$\ell_\chi=4$.  If $\chi\notin H^\perp$, then $\chi|_H$ is a nonzero linear
functional on the $2$--dimensional space $H$, so it takes the value $1$ on
exactly half of $H$, i.e.\ on $2$ points. Thus $|A_\chi\cap H|=2$, the sum
equals $6$, and $\ell_\chi=3$. Therefore $\ell_\chi\ge 3$ for all $\chi\neq 0$,
and $n_4=3$, $n_3=12$.

\item \textbf{Uniqueness and parametrization by a pair $(u,v)$.}
Let $d$ be any function satisfying the hypotheses with $(m_1)=(12)$.
From the moment analysis for $(D,s)=(12,4)$ and $\ell_\chi\ge 3$ established
earlier, this case forces the $\ell$--distribution to be of Type~A, namely
$\ell_\chi=4$ for exactly three nontrivial characters and $\ell_\chi=3$ for the
other twelve. Equivalently, writing the unnormalized Fourier transform
\[
\widehat d(\chi)=\sum_{g\in G} d(g)(-1)^{\chi\cdot g},
\]
the Fourier relation $\widehat d(\chi)=12-4\ell_\chi$ gives
\[
\widehat d(0)=12,\qquad
\widehat d(\chi)=
\begin{cases}
-4 & \text{for exactly three nontrivial }\chi,\\
0  & \text{for all other }\chi\neq 0.
\end{cases}
\]
Let $u_1,u_2,u_3$ be the three nontrivial characters with $\widehat d(u_i)=-4$.
By the inverse Fourier formula,
\[
d(x)=\frac{1}{16}\Bigl(12
-4\bigl((-1)^{u_1\cdot x}+(-1)^{u_2\cdot x}+(-1)^{u_3\cdot x}\bigr)\Bigr).
\]
Set $S(x):=(-1)^{u_1\cdot x}+(-1)^{u_2\cdot x}+(-1)^{u_3\cdot x}$.
Since $d(x)\in\{0,1\}$, we must have $S(x)\in\{3,-1\}$ for all $x$, which is
equivalent to
\[
u_1\cdot x+u_2\cdot x+u_3\cdot x\equiv 0 \pmod 2\qquad\text{for all }x.
\]
Hence $(u_1+u_2+u_3)\cdot x=0$ for all $x$, so $u_1+u_2+u_3=0$.
After relabeling, $u_3=u_1+u_2$. Thus the three exceptional characters are
exactly the nonzero elements of the $2$--dimensional subspace
$U:=\langle u_1,u_2\rangle\subset G^*$.

Let $u=u_1$ and $v=u_2$. Substituting $u,v,u+v$ into the inverse formula shows
that
\[
d(x)=0 \iff u\cdot x=v\cdot x=0,
\]
and therefore $d=\mathbf 1_{G\setminus H_{u,v}}$ where
$H_{u,v}:=\{x: u\cdot x=v\cdot x=0\}$ is a codimension--$2$ subspace of $G$.
In particular, $d$ is determined by the choice of the $2$--plane
$U=\langle u,v\rangle$ (equivalently by any choice of linearly independent
$u,v$ spanning it).

Finally, $\operatorname{GL}_4(\mathbb F_2)$ acts transitively on ordered pairs of
linearly independent vectors in $G^*$, hence transitively on $2$--dimensional
subspaces. Thus all such functions $d$ lie in a single
$\operatorname{GL}_4(\mathbb F_2)$--orbit, and the displayed $d$ is a
representative.
\end{enumerate}
\end{proof}

\subsection{Unboundedness of non-flat}

In contrast to the case of flat pluricanonical covers, we show by means of the following examples that there exists canonical and bicanonical $\mathbb{Z}_2^s$ covers of weighted projective threefolds for arbitrarily large $s$.

\begin{example}\label{unbounded_non-flat_canonical}
Let $G = (\mathbb{Z}_2)^s$ with $s \ge 2$. Fix $\chi_0 = (1,0,\dots,0) \in G^*$ and let
\[
S = \{ g \in G : g_1 = 1 \}, \quad |S| = 2^{s-1}.
\]
Define $d_g$ and $L$ as follows:

\begin{enumerate}
    \item If $s$ is even and $s \geq 4$:
        \[
        L = \frac{2^s - 4}{6}, \quad d_g = 
        \begin{cases}
        2 & \text{if } g \in S, \\
        0 & \text{otherwise}.
        \end{cases}
        \]
    \item If $s$ is odd and $s \ge 5$:
        \[
        L = \frac{2^{s-1} - 4}{6}, \quad d_g = 
        \begin{cases}
        1 & \text{if } g \in S, \\
        0 & \text{otherwise}.
        \end{cases}
        \]
\end{enumerate}

Then:
\begin{enumerate}
    \item \( d_g \in \mathbb{Z}_{\ge 0} \) for all \( g \in G \), and \( d_0 = 0 \),
    \item $\displaystyle\sum_g d_g = 6L + 4$,
    \item \( l_\chi \in \mathbb{Z}_{\ge 0} \) for all \( \chi \in G^* \), and \( l_0 = 0 \)
    \item For all $\chi \ne 0$, $\displaystyle l_\chi = \frac12 \sum_{g:\, \chi \cdot g = 1} d_g \ge L+1$ and $\l_{\chi}$ is an integer.
\end{enumerate}

Consequently, by \Cref{prop:pluricanonical_maps_of_threefolds}, the $\mathbb{Z}_2^s$ cover $f: X \to Y = \mathbb{P}(1,1,L,L)$ constructed using branch divisors of degree $D_g \in |\mathcal{O}_Y(d_g)|$ is actually the canonical map
\[
\begin{tikzcd}
    X \arrow[d, "f"] \arrow[dr, "\varphi_{|K_X|}"] & \\
    Y \arrow[r, hook, "\mathcal{O}_{Y}(L)"] & \mathbb{P}^N
\end{tikzcd}
\]

\end{example}

\begin{proof}
We first verify each condition.

\noindent We first show that $L$ is a positive integer 
\begin{enumerate}
    \item For even $s$ and $s \geq 4$, $s = 2t$, $t \geq 1$. Hence $2^s - 4 = 4^t - 4$. Since $4^t \equiv 4 \pmod{6}$ for all $t \ge 2$, we have $6 \mid 4^t - 4$ and hence $L$ is a positive integer $L \geq 2$.
    \item For odd $s \ge 5$, $s-1 \ge 4$ is even, so $2^{s-1} - 4$ is divisible by 6 by the same reasoning.
\end{enumerate}

\noindent We now show that $\sum_g d_g = 6L+4$.
\begin{enumerate}
    \item For even $s$: $\sum d_g = 2 \cdot 2^{s-1} = 2^s = 6L+4$. 
    \item For odd $s$: $\sum d_g = 1 \cdot 2^{s-1} = 2^{s-1} = 6L+4$.
\end{enumerate}

\noindent We verify that $l_{\chi}$ is an integer and $l_{\chi} \geq L+1$.

Suppose that $\chi = \chi_0$.
\begin{enumerate}
    \item For even $s$, $2l_{\chi_0} = 2 \cdot 2^{s-1} = 2^s \Rightarrow l_{\chi_0} = 2^{s-1} \ge L+1 = \frac{2^s + 2}{6}$ for $s \geq 1$. 
    \item For odd $s$, $2l_{\chi_0} = 1 \cdot 2^{s-1} = 2^{s-1} \Rightarrow l_{\chi_0} = 2^{s-2} \ge L+1 = \frac{2^{s-1} + 2}{6}$ for $s \geq 2$. 
\end{enumerate}

Suppose that $\chi \neq \chi_0$ Then set $\{ g \in S : \chi \cdot g = 1, \chi_0 \cdot g = 1 \}$ has size $2^{s-2}$.
\begin{enumerate}
    \item For even $s$, $2l_\chi = 2 \cdot 2^{s-2} = 2^{s-1} \Rightarrow l_\chi = 2^{s-2} \ge \frac{2^s + 2}{6}$ for $s \geq 2$. 
    \item For odd $s$, $2l_\chi = 1 \cdot 2^{s-2} = 2^{s-2} \Rightarrow l_\chi = 2^{s-3} \ge \frac{2^{s-1} + 2}{6}$ for $s \geq 3$. 
\end{enumerate}

\end{proof}

\begin{example}\label{unbounded_non-flat_bicanonical}
Let \( G = (\mathbb{Z}_2)^s \) with \( s \ge 3 \).  
Fix \( \chi_0 = (1,0,\dots,0) \in G^* \).  
Let \( S = \{ g \in G : g_1 = 1 \} \), so \( |S| = 2^{s-1} \).

Define \( t \) as follows:
\begin{enumerate}
    \item If \( s = 3 \): \( t = 6 \)
    \item If \( s = 4 \): \( t = 3 \)
    \item If \( s \ge 5 \):
    \begin{enumerate}
        \item \( s \equiv 0 \pmod{4} \): \( t = 3 \)
        \item \( s \equiv 1 \pmod{4} \): \( t = 4 \)
        \item \( s \equiv 2 \pmod{4} \): \( t = 2 \)
        \item \( s \equiv 3 \pmod{4} \): \( t = 1 \)
    \end{enumerate}
\end{enumerate}

Then define:
\[
d_g = 
\begin{cases}
t & \text{if } g \in S, \\
0 & \text{otherwise},
\end{cases}
\quad
L = \frac{t \cdot 2^{s-1} - 4}{5}.
\]

This construction satisfies:
\begin{enumerate}
    \item \( d_g \in \mathbb{Z}_{\ge 0} \) for all \( g \in G \), and \( d_0 = 0 \)
    \item \( \sum_{g \in G} d_g = 5L + 4 \)
    \item For all \( \chi \ne 0 \), \( l_\chi = \frac12 \sum_{g:\, \chi \cdot g = 1} d_g > L \)
    \item \( l_\chi \in \mathbb{Z}_{\ge 0} \) for all \( \chi \in G^* \), and \( l_0 = 0 \)
    \item \( L \in \mathbb{Z} \) and \( L \ge 2 \)
\end{enumerate}
\end{example}

Consequently, the $\mathbb{Z}_2^s$ cover $f: X \to Y = \mathbb{P}(1,1,L,L)$ constructed using branch divisors of degree $D_g \in |\mathcal{O}_Y(d_g)|$ is actually the bi-canonical map
\[
\begin{tikzcd}
    X \arrow[d, "f"] \arrow[dr, "\varphi_{|2K_X|}"] & \\
    Y \arrow[r, "\mathcal{O}_{Y}(L)"] & \mathbb{P}^N
\end{tikzcd}
\]

\begin{proof}
We verify each condition in order.

\noindent \textbf{1. Nonnegativity and \( d_0 = 0 \):}  
By definition, \( d_g = t \) or \( 0 \), and \( t \ge 1 \) in all cases, so \( d_g \ge 0 \).  
The zero vector \( 0 \) has first coordinate $0$, so \( 0 \notin S  \Rightarrow  d_0 = 0 \).

\noindent \textbf{2. Sum condition:}  
\[
\sum_{g \in G} d_g = t \cdot |S| = t \cdot 2^{s-1}.
\]
From the definition of \( L \), \( t \cdot 2^{s-1} = 5L + 4 \), so the sum equals \( 5L + 4 \).

\noindent \textbf{3. Lower bound on \( l_\chi \):}  
We have \( l_\chi = \frac12 \sum_{g:\, \chi \cdot g = 1} d_g \).

\textbf{Case 1: \( \chi = \chi_0 \)}  
Here \( \{ g : \chi_0 \cdot g = 1 \} = S \), so  
\[
l_{\chi_0} = \frac12 \sum_{g \in S} t = \frac12 \cdot t \cdot 2^{s-1} = \frac{5L+4}{2}.
\]
We require \( \frac{5L+4}{2} > L \), i.e., \( 5L+4 > 2L \), i.e., \( 3L > -4 \), which is true for \( L \ge 2 \).

\textbf{Case 2: \( \chi \ne \chi_0, 0 \)}  
The hyperplane \( \{ g : \chi \cdot g = 1 \} \) intersects \( S \) in exactly \( 2^{s-2} \) elements, so  
\[
l_\chi = \frac12 \cdot t \cdot 2^{s-2} = \frac{t \cdot 2^{s-1}}{4} = \frac{5L+4}{4}.
\]
We require \( \frac{5L+4}{4} > L \), i.e., \( 5L+4 > 4L \), i.e., \( L > -4 \), which is true for \( L \ge 2 \).

\noindent \textbf{4. Integrality of \( l_\chi \) and \( l_0 = 0 \):}  
\begin{enumerate}
    \item \( l_{\chi_0} = t \cdot 2^{s-2} \) is integer since \( s \ge 3 \)
    \item \( l_{\chi} = t \cdot 2^{s-3} \) is integer since \( s \ge 3 \)
    \item By definition, \( l_0 = 0 \)
\end{enumerate}

\noindent \textbf{5. Integrality and bound on \( L \):}  
By construction, \( t \) is chosen so that \( t \cdot 2^{s-1} \equiv 4 \pmod{5} \),  
so \( t \cdot 2^{s-1} - 4 \) is divisible by \(5 \implies L \in \mathbb{Z} \).

We verify \( L \ge 2 \):
\begin{enumerate}
    \item \( s = 3 \): \( t = 6 \), \( 6 \cdot 4 = 24 \), \( L = (24-4)/5 = 4 \ge 2 \)
    \item \( s = 4 \): \( t = 3 \), \( 3 \cdot 8 = 24 \), \( L = 4 \ge 2 \)
    \item \( s \ge 5 \): \( t \ge 1 \), \( 2^{s-1} \ge 16 \), so \( t \cdot 2^{s-1} \ge 16 \),  
    \( L \ge (16-4)/5 = 2.4  \implies  L \ge 3 \ge 2 \)
\end{enumerate}

All conditions are satisfied for every \( s \ge 3 \).
\end{proof}

\bibliographystyle{alpha}
\bibliography{biblio.bib}

\end{document}